\documentclass[12pt]{amsart}
\usepackage{amsfonts, amssymb, latexsym, amsthm}

\usepackage{float}
\usepackage[colorlinks=true,citecolor=black,linkcolor=black,urlcolor=blue]{hyperref}
\usepackage{color, graphicx}
  \theoremstyle{plain}
    \newtheorem{theorem}{Theorem}[section]

  \newtheorem{lemma}[theorem]{Lemma}

\theoremstyle{definition}
  \newtheorem{definition}[theorem]{Definition}

\newcommand{\excise}[1]{}

\newcommand{\rdots}{\mathinner{%
  \mkern1mu\raise1pt\hbox{.}%
  \mkern2mu\raise4pt\hbox{.}%
  \mkern2mu\raise7pt\vbox{\kern7pt\hbox{.}}\mkern1mu}}

\numberwithin{equation}{section}

\def\ZZ{{\mathbb Z}}

\newcommand{\cellsize}{15}
\newlength{\cellsz} 
\newsavebox{\cell}
\sbox{\cell}{\begin{picture}(\cellsize,\cellsize)
\put(0,0){\line(1,0){\cellsize}}
\put(0,0){\line(0,1){\cellsize}}
\put(\cellsize,0){\line(0,1){\cellsize}}
\put(0,\cellsize){\line(1,0){\cellsize}}
\end{picture}}
\newcommand\cellify[1]{\def\thearg{#1}\def\nothing{}%
\ifx\thearg\nothing
\vrule width0pt height\cellsz depth0pt\else
\hbox to 0pt{\usebox{\cell} \hss}\fi%
\vbox to \cellsz{
\vss
\hbox to \cellsz{\hss$#1$\hss}
\vss}}
\newcommand\tableau[1]{\vtop{\let\\\cr
\baselineskip -16000pt \lineskiplimit 16000pt \lineskip 0pt
\ialign{&\cellify{##}\cr#1\crcr}}}
\newcommand{\ww}{t}
\newcommand{\tw}{r}

\newcommand{\kellsize}{18}
\newlength{\kellsz} 
\newsavebox{\kell}
\sbox{\kell}{\begin{picture}(\kellsize,\kellsize)
\put(0,0){\line(1,0){\kellsize}}
\put(0,0){\line(0,1){\kellsize}}
\put(\kellsize,0){\line(0,1){\kellsize}}
\put(0,\kellsize){\line(1,0){\kellsize}}
\end{picture}}
\newcommand\kellify[1]{\def\thearg{#1}\def\nothing{}%
\ifx\thearg\nothing
\vrule width0pt height\kellsz depth0pt\else
\hbox to 0pt{\usebox{\kell} \hss}\fi%
\vbox to \kellsz{
\vss
\hbox to \kellsz{\hss$#1$\hss}
\vss}}
\newcommand\ktableau[1]{\vtop{\let\\\cr
\baselineskip -16000pt \lineskiplimit 16000pt \lineskip 0pt
\ialign{&\kellify{##}\cr#1\crcr}}}
\font\co=lcircle10

\def\jr{\smash{\raise2pt\hbox{\co \rlap{\rlap{\char'005} \char'007}}
               \raise6pt\hbox{\rlap{\vrule height5pt}}
               \raise2pt\hbox{\rlap{\hskip4pt \vrule height0.4pt depth0pt
                width5.7pt}}
               \raise2pt\hbox{\rlap{\hskip-9.5pt \vrule height.4pt depth0pt
                width6.2pt}}
               \lower6pt\hbox{\rlap{\vrule height4.5pt}}}}
\def\rj{\smash{\raise2pt\hbox{\co \rlap{\rlap{\char'004} \char'006}}
               \raise6pt\hbox{\rlap{\vrule height5pt}}
               \raise2pt\hbox{\rlap{\hskip4pt \vrule height0.4pt depth0pt
                width5.7pt}}
               \raise2pt\hbox{\rlap{\hskip-9.5pt \vrule height.4pt depth0pt
                width6.2pt}}
               \lower6pt\hbox{\rlap{\vrule height4.5pt}}}}
\def\je{\smash{\raise2pt\hbox{\co \rlap{\rlap{\char'005}
                \phantom{\char'007}}}\raise6pt\hbox{\rlap{\vrule height5pt}}
               \raise2pt\hbox{\rlap{\hskip-9.5pt \vrule height.4pt depth0pt
                width6.2pt}}}}
\def\ej{\smash{\raise2pt\hbox{\co \rlap{\rlap{\char'004}\phantom{\char'006}}}
               \raise2pt\hbox{\rlap{\hskip-9.5pt \vrule height.4pt depth0pt
                width6.2pt}}
               \lower6pt\hbox{\rlap{\vrule height4.5pt}}}}
\def\er{\smash{\raise2pt\hbox{\co \rlap{\rlap{\phantom{\char'005}} \char'007}}
               \raise2pt\hbox{\rlap{\hskip4pt \vrule height0.4pt depth0pt
                width5.7pt}}
               \lower6pt\hbox{\rlap{\vrule height4.5pt}}}}
\def\re{\smash{\raise2pt\hbox{\co \rlap{\rlap{\phantom{\char'004}} \char'006}}
               \raise6pt\hbox{\rlap{\vrule height5pt}}
               \raise2pt\hbox{\rlap{\hskip4pt \vrule height0.4pt depth0pt
                width5.7pt}}}}
\def\+{\smash{\lower6pt\hbox{\rlap{\vrule height17pt}}
                \raise2pt
                \hbox{\rlap{\hskip-9pt \vrule height.4pt depth0pt
                width18.7pt}}}}
\def\hor{\smash{\raise2pt\hbox{\rlap{\hskip-9.5pt \vrule height.4pt depth0pt
                width19.2pt}}}}
\def\ver{\smash{\lower6pt\hbox{\rlap{\vrule height17pt}}}}
\def\ho{\smash{\hbox{\rlap{\vrule height5pt}}
                \raise2pt
                \hbox{\rlap{\hskip-9pt \vrule height.4pt depth0pt
                width18.7pt}}}}

\def\textcross{\ \smash{\lower4pt\hbox{\rlap{\hskip4.15pt\vrule height14pt}}
                \raise2.8pt\hbox{\rlap{\hskip-3pt \vrule height.4pt depth0pt
                width14.7pt}}}\hskip12.7pt}
\def\textelbow{\ \hskip.1pt\smash{\raise2.75pt%
                \hbox{\co \hskip 4.15pt\rlap{\rlap{\char'004} \char'006}
                \lower6.8pt\rlap{\vrule height3.5pt}
                \raise3.6pt\rlap{\vrule height3.5pt}}
                \raise2.8pt\hbox{%
                  \rlap{\hskip-7.15pt \vrule height.4pt depth0pt width3.5pt}%
                  \rlap{\hskip4.05pt \vrule height.4pt depth0pt width3.5pt}}}
                \hskip8.7pt}

\begin{document}
\pagestyle{plain}
\title{ On $(t,r)$ Broadcast Domination Numbers of Grids}
\author{David Blessing}
\address{Department of Mathematics\\
Florida Gulf Coast University \\ Fort Myers, FL 33965}
\email{dcblessi@eagle.fgcu.edu}
\author{Erik Insko}
\address{Department of Mathematics\\
Florida Gulf Coast University \\ Fort Myers, FL 33965}
\email{einsko@fgcu.edu}
\author{Katie Johnson}
\address{Department of Mathematics\\
Florida Gulf Coast University \\ Fort Myers, FL 33965}
\email{kjohnson@fgcu.edu}
\author{Christie Mauretour} 
\address{Department of Mathematics\\
Florida Gulf Coast University \\ Fort Myers, FL 33965}
\email{cjmauret@eagle.fgcu.edu}

\date{\today}

\begin{abstract}
The domination number of a graph $G = (V,E)$ is the minimum cardinality of any subset $S \subset V$ such that  every vertex in $V$ is in $S$ or adjacent to an element of $S$.
Finding the domination numbers of $m$ by $n$ grids was an open problem for nearly 30 years and was finally solved in 2011 by Goncalves, Pinlou, Rao, and Thomass\'e.  
Many variants of domination number on graphs have been defined and studied, but exact values have not yet been obtained for grids.
We will define a family of domination theories parameterized by pairs of positive integers $(t,r)$ where $1 \leq r \leq t$ which generalize 
domination and distance domination theories for graphs.
We call these domination numbers the $(t,r)$ broadcast domination numbers.  
We give the exact values of $(t,r)$ broadcast domination numbers for small grids, and 
we identify upper bounds for the $(t,r)$ broadcast domination numbers for large grids and conjecture that these bounds are tight for sufficiently large grids.
\end{abstract}

\maketitle

\section{Introduction}
A \emph{dominating set} in a graph $G$ is a subset of vertices $S$ such that every vertex
in $G$ is either in $S$ or is adjacent to some vertex in $S$. 
The \emph{domination number} of $G$, denoted $\gamma(G)$, is
the minimum size of a dominating set of $G$.  
For a comprehensive study of domination and its variants on graphs see the two texts by Haynes, Hedetniemi and Slater \cite{HayHedSla98,HayHedSla98b}. 
In this paper we focus on generalizations of domination number for grid graphs.  

Finding the specific domination number for any $m \times n$ grid graph proved to be a challenging task. 
Indeed it was an open problem for over a quarter century.
In 1984, Jacobson and Kinch \cite{JacKin84} started the investigation by publishing the specific values of $\gamma(G_{2,n})$, $\gamma(G_{3,n}),$ and $\gamma(G_{4,n})$. 
In 1993, Chang, Clark, and Hare \cite{ChaCla93} extended these results by finding the exact values of $\gamma(G_{5,n})$ and $\gamma(G_{6,n})$. 
In his Ph.D. thesis, Chang \cite{Cha92} constructed efficient dominating sets proving that when $m$ and $n$ are greater than $8$, the domination number $\gamma(G_{m,n})$ is bounded by the formula
\begin{equation} \gamma(G_{m,n}) \leq \left \lfloor \frac{(n+2)(m+2)}{ 5}\right \rfloor - 4 . \end{equation} \label{chang} Chang also conjectured that 
equality holds in Equation (\ref{chang}) when $n \geq m \geq 16$.  
 
In an effort to confirm Chang's conjecture, a number of mathematicians and computer scientists began exhaustively computing the values of $\gamma(G_{m,n})$.
In 1995, Hare, Hare, and  Hedetniemi \cite{HarHarHed86} developed a polynomial time algorithm to compute $\gamma(G_{m,n} )$ when $m$ is fixed.  
Spalding's 1998 Ph.D. thesis \cite{Spa98} computed $\gamma(G_{m,n})$ for $m \leq 19$ and all $n$, and  Alanko, Crevals, Isopoussu, \"Ostergard, and Petterson \cite{Ala11} computed $\gamma(G_{m,n})$ for $m,n \leq 29$ in addition to $m \leq 27$ and $n \leq 1000$.

In 2004, Guichard \cite{Gui04} proved the following bound for $n \ge m \ge 16$: 
\[ \gamma(G_{m,n}) \geq \left \lfloor \frac{(n+2)(m+2)}{ 5}\right \rfloor - 9 . \]
Finally in 2011, Gon\c{c}alves, Pinlou, Rao, and Thomass\'{e} \cite{GonPinRaoTho11} were able to adapt Guichard's ideas to confirm Chang's conjecture for all $n$.  
Their proof uses a combination of analytic and computer aided techniques for the large cases $(n \geq m \geq 24)$ and exhaustive calculations for the smaller ones.

The concept of graph domination has been generalized in over 80 ways including distance domination, $R$-domination, double-domination and $(k,r)$-domination to name just a few 
\cite{Sla76, Hen98, HarHay00, JotPusSugSwa2011}.  Relatively little is known about these other domination theories in grid graphs, but in 2013 Fata, Smith, and Sundaram 
defined an efficient algorithm for constructing dominating sets
that give a loose upper bound on the distance domination number of grids \cite[Theorem V.10]{FatSmiSun13}.    
In this paper we define a graph invariant called the $(\ww,\tw)$ broadcast domination number that generalizes the theories of domination and distance domination.  
We compute the exact values of these $(\ww,\tw)$ broadcast domination numbers on grid graphs $G_{m,n}$ for small values of $\ww$, $\tw$, and $m$, and any value of $n$. 
Then we construct dominating sets that give upper bounds for large $m \times n$ grids and small values of $\ww$ and $ \tw$. We conjecture that these bounds are tight for sufficiently large grids.

The rest of this section contains an introduction to broadcast domination.   
In Section~\ref{smallcases} we prove formulas for $(2,2)$ broadcast domination numbers of $G_{m,n}$ when $m = 3,4,$ or $5$, and $(3,1)$, and $(3,2)$ broadcast domination numbers of $G_{m,n}$ when $m = 3$ or $4$.
In Section~\ref{section:theorems} we construct sets that give upper bounds on the $(2,2)$, $(3,1)$, $(3,2)$, and $(3,3)$  broadcast domination numbers of $m \times n$ grids. 
Finally, in Section~\ref{section:algorithm} we give a table of the $(2,2)$ and $(3,1)$ broadcast domination numbers for $G_{m,n}$ when $1 \leq m \leq n \leq 10$ and list some open problems in $(\ww,\tw)$ broadcast domination.

\subsection{$(\ww,\tw)$  Broadcast Domination} \label{subsection:broadcast_domination}
In this section we define the concept of $(\ww,\tw)$  broadcast domination. 
For two vertices $u$ and $v$ in $G$, let $d(u,v)$ denote the distance, or fewest number of edges, between $u$ and $v$ in $G$.  
We say a vertex $v \in G$ is a \emph{broadcasting vertex of transmission strength} $\ww$
if it transmits a signal of strength $ \ww - d(u,v)$ to every vertex $u$ with $d(u,v) < \ww$.  
Given a broadcasting vertex $v$ of transmission strength $\ww$, we call all vertices with $d(u,v) < \ww$ the \emph{broadcast neighborhood} of $v$, denoted $N_t(v)$.  Note that a broadcasting vertex also broadcasts a signal of strength $t$ to itself, and that we are including $v$ in $N_t(v)$.

We will call a set $S$ of broadcasting vertices of strength $t$ a \emph{broadcasting set}.
We define the \emph{reception strength} $\tw(u)$ at a vertex $u \in G$ to be the sum of the transmission strengths from all surrounding broadcasting vertices, i.e.
\[ \tw(u) = \sum_{\substack{v \in S \\ u\in N_t(v)}} \left (  \ww-d(u,v) \right ) .\]

\begin{definition}
A set $S \subset V$ is called a \emph{$(\ww,\tw)$  broadcast dominating set} if every vertex $v \in V$ has a reception strength $\tw(v)$ satisfying $\tw(v) \geq \tw$.
\end{definition}

\begin{definition} 
The \emph{$(\ww,\tw)$  broadcast domination number} of a graph is the minimum size of any $(\ww,\tw)$  broadcast dominating set.
We denote this number by $\gamma_{\ww,\tw}(G)$.
\end{definition}

Figure \ref{G55} shows the reception strengths (recorded in red) for two different broadcasting sets (highlighted in blue) in the grid $G_{5,5}$ when $(\ww, \tw) = (3,2)$.  
On the right, it also shows a minimum (3,2) broadcast dominating set for $G_{5,5}$.

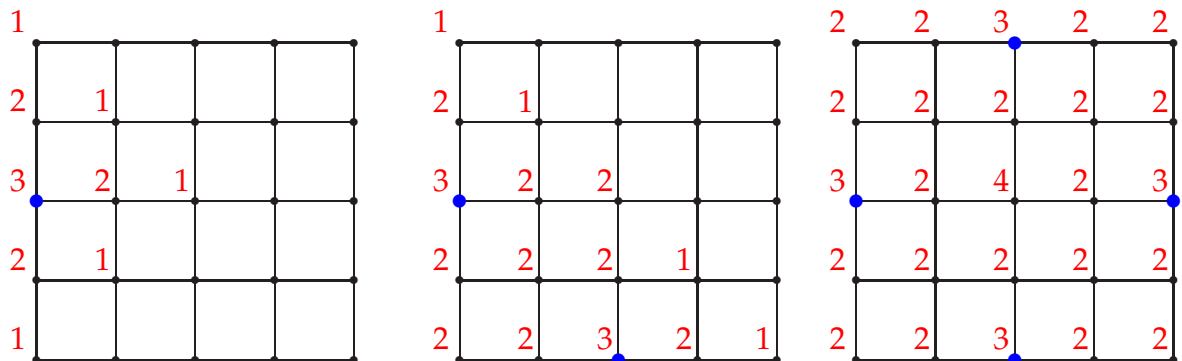
\begin{figure}[H]
\centering
\begin{picture}(450,125)(0,0)
	\multiput(10,0)(30,0){5}{\circle*{3}}
	\multiput(10,30)(30,0){5}{\circle*{3}}
	\multiput(10,60)(30,0){5}{\circle*{3}}
	\multiput(10,90)(30,0){5}{\circle*{3}}
	\multiput(10,120)(30,0){5}{\circle*{3}}
	\multiput(10,0)(0,30){5}{\line(1,0){120}}
	\multiput(10,0)(30,0){5}{\line(0,1){120}}
	\put(10,60){\color{blue}{\circle*{5}}}
	\put(0,64){\color{red}{3}}
	\put(0,94){\color{red}{2}}
	\put(0,34){\color{red}{2}}
	\put(32,64){\color{red}{2}}
	\put(62,64){\color{red}{1}}
	\put(0,4){\color{red}{1}}
	\put(32,34){\color{red}{1}}
	\put(0,124){\color{red}{1}}
	\put(32,94){\color{red}{1}}
	\multiput(170,0)(30,0){5}{\circle*{3}}
	\multiput(170,30)(30,0){5}{\circle*{3}}
	\multiput(170,60)(30,0){5}{\circle*{3}}
	\multiput(170,90)(30,0){5}{\circle*{3}}
	\multiput(170,120)(30,0){5}{\circle*{3}}
	\multiput(170,0)(0,30){5}{\line(1,0){120}}
	\multiput(170,0)(30,0){5}{\line(0,1){120}}
	\put(170,60){\color{blue}{\circle*{5}}}
	\put(230,0){\color{blue}{\circle*{5}}}
	\put(222,4){\color{red}{3}}
	\put(222,34){\color{red}{2}}
	\put(252,4){\color{red}{2}}
	\put(282,4){\color{red}{1}}
	\put(252,34){\color{red}{1}}
	\put(192,4){\color{red}{2}}
	\put(160,64){\color{red}{3}}
	\put(160,94){\color{red}{2}}
	\put(160,34){\color{red}{2}}
	\put(192,64){\color{red}{2}}
	\put(222,64){\color{red}{2}}
	\put(160,4){\color{red}{2}}
	\put(192,34){\color{red}{2}}
	\put(160,124){\color{red}{1}}
	\put(192,94){\color{red}{1}}
	\multiput(320,0)(30,0){5}{\circle*{3}}
	\multiput(320,30)(30,0){5}{\circle*{3}}
	\multiput(320,60)(30,0){5}{\circle*{3}}
	\multiput(320,90)(30,0){5}{\circle*{3}}
	\multiput(320,120)(30,0){5}{\circle*{3}}
	\multiput(320,0)(0,30){5}{\line(1,0){120}}
	\multiput(320,0)(30,0){5}{\line(0,1){120}}
	\put(320,60){\color{blue}{\circle*{5}}}
	\put(380,0){\color{blue}{\circle*{5}}}
	\put(380,120){\color{blue}{\circle*{5}}}
	\put(440,60){\color{blue}{\circle*{5}}}
 	\put(372,4){\color{red}{3}}
	\put(372,34){\color{red}{2}}
	\put(402,4){\color{red}{2}}
	\put(432,4){\color{red}{2}}
	\put(402,34){\color{red}{2}}
	\put(342,4){\color{red}{2}}
	\put(310,64){\color{red}{3}}
	\put(310,94){\color{red}{2}}
	\put(310,34){\color{red}{2}}
	\put(342,64){\color{red}{2}}
	\put(372,64){\color{red}{4}}
	\put(310,4){\color{red}{2}}
	\put(342,34){\color{red}{2}}
	\put(310,124){\color{red}{2}}
	\put(342,94){\color{red}{2}}
        \put(342,124){\color{red}{2}}
        \put(372,124){\color{red}{3}}
        \put(372,94){\color{red}{2}} 
        \put(402,124){\color{red}{2}}
        \put(432,124){\color{red}{2}}
        \put(432,94){\color{red}{2}}
	\put(432,64){\color{red}{3}}
	\put(432,34){\color{red}{2}}
        \put(402,64){\color{red}{2}}
        \put(402,94){\color{red}{2}}
\end{picture}
	\caption{Reception strengths and a $(3,2)$  broadcast dominating set of $G_{5,5}$} \label{G55}
\end{figure}

The family of $(\ww,\tw)$ broadcast domination theories generalize several well-known domination theories of interest.  For instance,
when $(\ww,\tw)=(2,1)$ every element of $G$ is either in a broadcast dominating set $S$ or is adjacent to $S$.  
So the $(2,1)$  broadcast domination number is precisely the regular domination number, i.e. $\gamma_{2,1}(G_{m,n})=\gamma(G_{m,n})$.
In $(\ww,1)$ broadcast domination theory every vertex of $G$ must be within
distance $\ww-1$ of an element of the dominating set.  Hence $(\ww,1)$ broadcast domination is equivalent to $(\ww-1)$-distance domination theory introduced by Slater \cite{Sla76} 
and studied in grids by Fata, Smith, and Sundaram \cite{FatSmiSun13}. 
Thus the family of $(\ww,\tw)$ broadcast domination theories provides a general framework for studying several domination theories of interest, and in this paper we develop techniques
that can be effectively applied in the study of any $(t,r)$ broadcast domination theory.

\section{ $(\ww,\tw)$ Broadcast Domination in Small Grids} \label{smallcases}

In this section we will find the $(2,2)$ broadcast domination numbers of $G_{m,n}$ when $m$ is $ 3,4,$ or $5$ and $(3,1)$ and $(3,2)$ broadcast domination numbers of $G_{m,n}$ when $m$ is $ 3$ or $4$. 
 In each instance, we give a construction of a ($t,r$) broadcast dominating set, and then we prove that each construction is optimal by showing all smaller sets fail to dominate $G_{m,n}$.

In describing the dominating sets of small grids, we will use the following notation.  
We call a string of integers having the form $c_1$-$c_2$-$\cdots$-$c_k$ a \emph{length-$k$ pattern}.
We say that a subset of vertices in $G_{m,n}$ satisfies the pattern $c_1$-$c_2$-$\cdots$-$c_k$ if it contains $c_1$ 
vertices in the first column, $c_2$ vertices in the second, and $c_j$ vertices in the $j$th column for $1 \leq j \leq k$.  

For instance, the first two sets shown in Figure \ref{patterns13n} are (2,2) broadcast dominating sets for $G_{3,5}$ and $G_{3,6}$, and they satisfy the patterns 
1-2-1-2-1 and 1-2-1-1-2-1 respectively; accordingly, we call these dominating patterns.
It is important to note that not every set satisfying a dominating pattern necessarily dominates $G_{m,n}$.  
For instance, the third set in Figure \ref{patterns13n} also satisfies the length-$6$ pattern 1-2-1-1-2-1 but it does not dominate $G_{3,6}$.
\begin{figure}[H]
\centering
\begin{picture}(390,30)(0,0)

	\multiput(100,0)(0,10){3}{\line(1,0){40}}
	\multiput(100,0)(10,0){5}{\line(0,1){20}}
	\put(100,10){\color{blue}{\circle*{5}}}
	\put(110,0){\color{blue}{\circle*{5}}}
	\put(110,20){\color{blue}{\circle*{5}}}
	\put(120,10){\color{blue}{\circle*{5}}}
	\put(130,20){\color{blue}{\circle*{5}}}
	\put(130,0){\color{blue}{\circle*{5}}}
	\put(140,10){\color{blue}{\circle*{5}}}
	\multiput(160,0)(0,10){3}{\line(1,0){50}}
	\multiput(160,0)(10,0){6}{\line(0,1){20}}
	\put(160,10){\color{blue}{\circle*{5}}}
	\put(170,0){\color{blue}{\circle*{5}}}
	\put(170,20){\color{blue}{\circle*{5}}}
	\put(180,10){\color{blue}{\circle*{5}}}
	\put(190,10){\color{blue}{\circle*{5}}}
	\put(200,0){\color{blue}{\circle*{5}}}
	\put(200,20){\color{blue}{\circle*{5}}}
	\put(210,10){\color{blue}{\circle*{5}}}

	\multiput(230,0)(0,10){3}{\line(1,0){50}}
	\multiput(230,0)(10,0){6}{\line(0,1){20}}
	\put(230,10){\color{blue}{\circle*{5}}}
	\put(240,10){\color{blue}{\circle*{5}}}
	\put(240,0){\color{blue}{\circle*{5}}}
	\put(250,20){\color{blue}{\circle*{5}}}
	\put(260,20){\color{blue}{\circle*{5}}}
	\put(270,0){\color{blue}{\circle*{5}}}
	\put(270,10){\color{blue}{\circle*{5}}}
	\put(280,10){\color{blue}{\circle*{5}}}

\end{picture}
\caption{Patterns for  $G_{3,n}$} \label{patterns13n}
\end{figure}
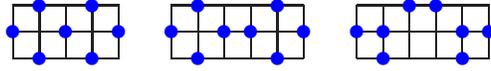

We will often focus on constructing a dominating set a few columns at a time. If a string of integers $d_1$-$d_2$-$\cdots$-$d_l$ appears as a contiguous substring of the integers $c_1$-$c_2$-$\cdots$-$c_k$ 
so that $d_1 = c_{i}, d_2 = c_{i+1}, \ldots,$ and $d_l = c_{i+l-1}$ 
for some $1 \leq i \leq k- l+1$
then we say that $d_1$-$d_2$-$\ldots$-$d_l$ is a subpattern of $c_1$-$c_2$-$\cdots$-$c_k$. 
We call a pattern a \emph{dominating subpattern} if there exists a dominating set that contains that subpattern, or equivalently,  
a dominating subpattern is a pattern that can be extended to a dominating pattern by only adding vertices to the left or right of the subpattern. 
For instance in Figure \ref{fig:extending} the pattern 3-1-2-2-2 gets extended to the (2,2) dominating pattern 2-2-1-3-1-2-2-2-2-2  of $G_{4,10} $ by adding 2-2-1- to the left and -2-2 
on the right.  
 
\begin{figure}[H]
\centering
\begin{picture}(380,30)(0,0)
	\linethickness{.1 mm}
	\multiput(10,0)(0,10){4}{\line(1,0){90}}
	\multiput(10,0)(10,0){10}{\line(0,1){30}}
	\put(40,30){\color{blue}{\circle*{5}}}
	\put(40,20){\color{blue}{\circle*{5}}}
	\put(40,0){\color{blue}{\circle*{5}}}
	\put(50,10){\color{blue}{\circle*{5}}}
	\put(60,30){\color{blue}{\circle*{5}}}
	\put(60,0){\color{blue}{\circle*{5}}}
	\put(70,20){\color{blue}{\circle*{5}}}
	\put(70,0){\color{blue}{\circle*{5}}}
	\put(80,30){\color{blue}{\circle*{5}}}
	\put(80,10){\color{blue}{\circle*{5}}}

	\linethickness{.1 mm}
	\multiput(220,0)(0,10){4}{\line(1,0){90}}
	\multiput(220,0)(10,0){10}{\line(0,1){30}}
	\put(220,20){\color{blue}{\circle*{5}}}
	\put(220,0){\color{blue}{\circle*{5}}}
	\put(230,30){\color{blue}{\circle*{5}}}
	\put(230,0){\color{blue}{\circle*{5}}}
	\put(240,10){\color{blue}{\circle*{5}}}
	\put(250,30){\color{blue}{\circle*{5}}}
	\put(250,20){\color{blue}{\circle*{5}}}
	\put(250,0){\color{blue}{\circle*{5}}}
	\put(260,10){\color{blue}{\circle*{5}}}
	\put(270,30){\color{blue}{\circle*{5}}}
	\put(270,0){\color{blue}{\circle*{5}}}
	\put(280,20){\color{blue}{\circle*{5}}}
	\put(280,0){\color{blue}{\circle*{5}}}
	\put(290,30){\color{blue}{\circle*{5}}}
	\put(290,10){\color{blue}{\circle*{5}}}
	\put(300,30){\color{blue}{\circle*{5}}}
	\put(300,0){\color{blue}{\circle*{5}}}
	\put(310,20){\color{blue}{\circle*{5}}}
	\put(310,0){\color{blue}{\circle*{5}}}

\end{picture}
\caption{A Dominating Subpattern Extended to a Dominating Pattern} \label{fig:extending}
\end{figure}
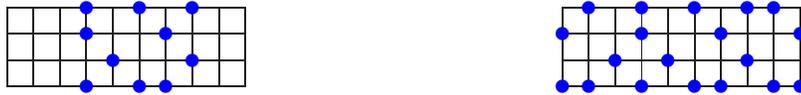

We now identify two key properties of any ($t,r$) dominating subpattern.

\vspace{12pt}
\noindent\textbf{Necessary properties of a ($t,r$) dominating subpattern $c_1$-$c_2$-$\cdots$-$c_k$:} 
\begin{enumerate}
 \item Each vertex in the interior columns of the subpattern corresponding to $c_t,$ $c_{t+1},$ $c_{t+2},$ $\ldots,$ $c_{k-t}$ has reception strength $r(v) \geq r$.
 \item Every vertex $v$ in the columns corresponding to $c_{t-\ell}$ and $c_{k-t+\ell}$ with $0 \leq \ell \leq r-1 $ must have reception strength $r(v) \geq r-\ell$.
\end{enumerate}

The first property is necessary for the pattern to dominate its interior columns, and the second property is necessary for the pattern to be extended to a dominating pattern for the larger grid.
The fact that these properties are necessary for any subpattern to be extended to a dominating pattern follows immediately from the fact that a broadcasting 
vertex of strength $t$ cannot effect the reception strength of a vertex that is more than $t$ columns away from it.
Figure \ref{keyprops} shows a (2,2) and a (3,2) dominating subpattern.  Notice that in the first column and last column of the (2,2) dominating subpattern, each vertex $v$ must
 have reception strength $r(v) \geq 2-1 =1$,
and each vertex $v$ from second through penultimate column of the pattern must have reception strength $r(v) \geq 2$ in order for the pattern to be extended to a dominating pattern.  
Similarly, each vertex $v$ in the third and third to last column in the (3,2) dominating subpattern example 
must have reception strength $r(v) \geq 2$, while each vertex in the second and second to last column must have reception strength $r(v) \geq 2-1 =1$.   
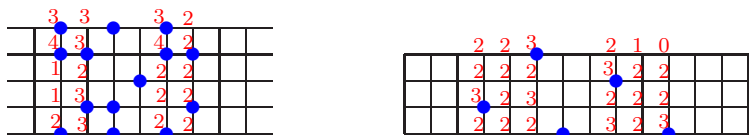
\begin{figure}[H]
\centering
\begin{picture}(330,40)(-50,0)


	\multiput(0,0)(0,10){5}{\line(1,0){100}}
	\multiput(10,0)(10,0){9}{\line(0,1){40}}
        \put(20,0){\color{blue}{\circle*{5}}}
        \put(20,30){\color{blue}{\circle*{5}}}
	\put(20,40){\color{blue}{\circle*{5}}}
        \put(25,12){\color{red}{\tiny{$3$}}}
        \put(26,21){\color{red}{\tiny{$2$}}}
        \put(26,1){\color{red}{\tiny{$3$}}}
        \put(16,2){\color{red}{\tiny{$2$}}}
        \put(16,12){\color{red}{\tiny{$1$}}}
        \put(16,22){\color{red}{\tiny{$1$}}}
        \put(15,32){\color{red}{\tiny{$4$}}}
        \put(15,42){\color{red}{\tiny{$3$}}}
        \put(25,32){\color{red}{\tiny{$3$}}}
        \put(27,42){\color{red}{\tiny{$3$}}}
        \put(30,30){\color{blue}{\circle*{5}}}
	\put(30,10){\color{blue}{\circle*{5}}}
	\put(40,40){\color{blue}{\circle*{5}}}
	\put(40,0){\color{blue}{\circle*{5}}}
	\put(40,10){\color{blue}{\circle*{5}}}
	\put(50,20){\color{blue}{\circle*{5}}}
	\put(60,30){\color{blue}{\circle*{5}}}
	\put(60,40){\color{blue}{\circle*{5}}}
	\put(60,0){\color{blue}{\circle*{5}}}
	\put(70,30){\color{blue}{\circle*{5}}}
	\put(70,10){\color{blue}{\circle*{5}}}
        \put(66,1){\color{red}{\tiny{$2$}}}
        \put(66,12){\color{red}{\tiny{$2$}}}
        \put(66,21){\color{red}{\tiny{$2$}}}
        \put(66,32){\color{red}{\tiny{$2$}}}
        \put(66,41){\color{red}{\tiny{$2$}}}
        \put(55,2){\color{red}{\tiny{$2$}}}
        \put(56,12){\color{red}{\tiny{$2$}}}
        \put(56,21){\color{red}{\tiny{$2$}}}
        \put(55,32){\color{red}{\tiny{$4$}}}
        \put(55,42){\color{red}{\tiny{$3$}}}

	\multiput(150,0)(0,10){4}{\line(1,0){130}}
	\multiput(150,0)(10,0){14}{\line(0,1){30}}
        \put(175,12){\color{red}{\tiny{$3$}}}
        \put(176,1){\color{red}{\tiny{$2$}}}
        \put(176,21){\color{red}{\tiny{$2$}}}
	\put(176,31){\color{red}{\tiny{$2$}}}
        \put(186,11){\color{red}{\tiny{$2$}}}
        \put(186,1){\color{red}{\tiny{$2$}}}
        \put(186,21){\color{red}{\tiny{$2$}}}
	\put(186,31){\color{red}{\tiny{$2$}}}

        \put(196,11){\color{red}{\tiny{$3$}}}
        \put(196,1){\color{red}{\tiny{$2$}}}
        \put(196,21){\color{red}{\tiny{$2$}}}
	\put(196,32){\color{red}{\tiny{$3$}}}
	\put(180,10){\color{blue}{\circle*{5}}}
	\put(200,30){\color{blue}{\circle*{5}}}
	\put(210,0){\color{blue}{\circle*{5}}}
	\put(230,20){\color{blue}{\circle*{5}}}
      \put(226,1){\color{red}{\tiny{$3$}}}
     \put(226,11){\color{red}{\tiny{$2$}}}
     \put(225,22){\color{red}{\tiny{$3$}}}
     \put(226,31){\color{red}{\tiny{$2$}}}
      \put(236,1){\color{red}{\tiny{$2$}}}
     \put(236,11){\color{red}{\tiny{$2$}}}
     \put(236,21){\color{red}{\tiny{$2$}}}
     \put(236,31){\color{red}{\tiny{$1$}}}

    \put(250,0){\color{blue}{\circle*{5}}}
     \put(246,2){\color{red}{\tiny{$3$}}}
     \put(246,11){\color{red}{\tiny{$2$}}}
     \put(246,21){\color{red}{\tiny{$2$}}}
     \put(246,31){\color{red}{\tiny{$0$}}}

\end{picture} 
\caption{ (2,2) dominating subpattern and (3,2) dominating subpattern} \label{keyprops}
\end{figure}

\subsection{The Minimal Pattern Search (MPS) Algorithm}
We describe minimal $(\ww,\tw)$  dominating sets in the following subsections. While it is easy to verify that these sets are dominating sets, 
it is not a priori obvious that these sets are  the smallest possible dominating sets for each $G_{m,n}$. 
To prove that a subpattern is minimal, we have written an algorithm in SAGE, which we call the Minimal Pattern Search (MPS) algorithm.  This MPS algorithm constructs all subsets of $G_{m,n}$ that satisfy a given pattern $c_1$-$c_2$-$c_3$-$\cdots$-$c_k$. 
That is, it constructs all ${m \choose c_1}{m \choose c_2} \cdots {m \choose c_k}$ sets containing $c_1$ vertices in 
the first column, $c_2$ vertices in the second, and so on. 
The program reports any subsets that satisfy the two necessary properties of a $(\ww,\tw)$ dominating subpattern.
If the algorithm does not return any subset, then we conclude that all subsets of $G_{m,n}$ satisfying $c_1$-$c_2$-$c_3$-$\cdots$-$c_k$ are not dominating sets of $G_{m,n}$.
Hence given a possibly minimal pattern  $c_1$-$c_2$-$c_3$-$\cdots$-$c_k$, we simply feed every pattern $d_1$-$d_2$-$d_3$-$\cdots$-$d_k$ with $d_1+d_2+\cdots+d_k < c_1+c_2+ \cdots +c_k$ 
to the MPS algorithm to make sure that no subset of $G_{m,n}$ satisfying the pattern $d_1$-$d_2$-$d_3$-$\cdots$-$d_k$ is a dominating set. 
Both the SAGE code and the calculations are available at {\small \tt https://cloud.sagemath.com/projects/26b983ae-a894-47c0-bd54-c73b071e555c/files/}.

For instance, Figure \ref{patterns3n} shows a (2,2) broadcast dominating set for $G_{3,n}$ that follows the pattern 1-2-1-1-2-1-1-2-1-1-$\cdots$-2-1.  
To prove that this pattern is a minimal dominating pattern for $3 \times n$ grids, we feed the MPS program the pattern $2$-$1$-$1$-$1$.  
It runs through all ${3 \choose 2}{3 \choose 1}{3 \choose 1}{3 \choose 1}$ sets, and shows that no set containing 
four columns with the pattern 2-1-1-1 can satisfy the first property of a dominating subpattern.  
Hence the pattern 2-1-1-1 cannot be part of a dominating set for $G_{m,n}$. 
 While this example was small enough to check by hand, the patterns in $4 \times n$ and $5 \times n$ grids get out of hand very quickly.

\subsection{(2,2) domination of 3 by n grids}
We begin by describing a dominating set $D_n$ for the grid $G_{3,n}$.  
As depicted in Figure \ref{patterns3n}, each set $D_n$ starts with the pattern 1-2 and ends with 2-1.  Then $D_n$ repeats the subpattern -2-1-1- as many times as possible in the middle columns.
Figure \ref{patterns29} shows a larger example.

\begin{figure}[H]
\centering
\begin{picture}(390,30)(0,0)
	\linethickness{.1 mm}
	\multiput(10,0)(0,10){3}{\line(1,0){20}}
	\multiput(10,0)(10,0){3}{\line(0,1){20}}
	\put(10,10){\color{blue}{\circle*{5}}}
	\put(20,0){\color{blue}{\circle*{5}}}
	\put(20,20){\color{blue}{\circle*{5}}}
	\put(30,10){\color{blue}{\circle*{5}}}
	\multiput(50,0)(0,10){3}{\line(1,0){30}}
	\multiput(50,0)(10,0){4}{\line(0,1){20}}
	\put(50,10){\color{blue}{\circle*{5}}}
	\put(60,0){\color{blue}{\circle*{5}}}
	\put(60,20){\color{blue}{\circle*{5}}}
	\put(70,20){\color{blue}{\circle*{5}}}
	\put(70,0){\color{blue}{\circle*{5}}}
	\put(80,10){\color{blue}{\circle*{5}}}
	\multiput(100,0)(0,10){3}{\line(1,0){40}}
	\multiput(100,0)(10,0){5}{\line(0,1){20}}
	\put(100,10){\color{blue}{\circle*{5}}}
	\put(110,0){\color{blue}{\circle*{5}}}
	\put(110,20){\color{blue}{\circle*{5}}}
	\put(120,10){\color{blue}{\circle*{5}}}
	\put(130,20){\color{blue}{\circle*{5}}}
	\put(130,0){\color{blue}{\circle*{5}}}
	\put(140,10){\color{blue}{\circle*{5}}}
	\multiput(160,0)(0,10){3}{\line(1,0){50}}
	\multiput(160,0)(10,0){6}{\line(0,1){20}}
	\put(160,10){\color{blue}{\circle*{5}}}
	\put(170,0){\color{blue}{\circle*{5}}}
	\put(170,20){\color{blue}{\circle*{5}}}
	\put(180,10){\color{blue}{\circle*{5}}}
	\put(190,10){\color{blue}{\circle*{5}}}
	\put(200,0){\color{blue}{\circle*{5}}}
	\put(200,20){\color{blue}{\circle*{5}}}
	\put(210,10){\color{blue}{\circle*{5}}}
	\multiput(230,0)(0,10){3}{\line(1,0){60}}
	\multiput(230,0)(10,0){7}{\line(0,1){20}}
	\put(230,10){\color{blue}{\circle*{5}}}
	\put(240,0){\color{blue}{\circle*{5}}}
	\put(240,20){\color{blue}{\circle*{5}}}
	\put(250,10){\color{blue}{\circle*{5}}}
	\put(260,10){\color{blue}{\circle*{5}}}
	\put(290,10){\color{blue}{\circle*{5}}}
	\put(280,0){\color{blue}{\circle*{5}}}
	\put(270,0){\color{blue}{\circle*{5}}}
	\put(270,20){\color{blue}{\circle*{5}}}
	\put(280,20){\color{blue}{\circle*{5}}}
	\multiput(310,0)(0,10){3}{\line(1,0){70}}
	\multiput(310,0)(10,0){8}{\line(0,1){20}}
	\put(310,10){\color{blue}{\circle*{5}}}
	\put(320,0){\color{blue}{\circle*{5}}}
	\put(320,20){\color{blue}{\circle*{5}}}
	\put(330,10){\color{blue}{\circle*{5}}}
	\put(340,10){\color{blue}{\circle*{5}}}
	\put(370,20){\color{blue}{\circle*{5}}}
	\put(380,10){\color{blue}{\circle*{5}}}
	\put(350,0){\color{blue}{\circle*{5}}}
	\put(350,20){\color{blue}{\circle*{5}}}
	\put(360,10){\color{blue}{\circle*{5}}}
	\put(370,0){\color{blue}{\circle*{5}}}

\end{picture}
\caption{Dominating sets $D_n$ for $G_{3,n}$} \label{patterns3n}
\end{figure}
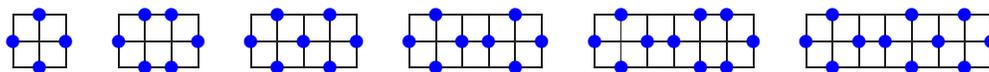

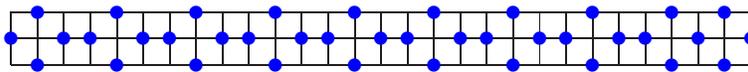
\begin{figure}[H]
\centering
\begin{picture}(300,20)(0,0)
	\linethickness{.1 mm}
	\multiput(10,0)(0,10){3}{\line(1,0){280}}
	\multiput(10,0)(10,0){29}{\line(0,1){20}}
	\put(10,10){\color{blue}{\circle*{5}}}
	\put(20,0){\color{blue}{\circle*{5}}}
	\put(20,20){\color{blue}{\circle*{5}}}
	\put(30,10){\color{blue}{\circle*{5}}}
	\put(40,10){\color{blue}{\circle*{5}}}
	\put(50,0){\color{blue}{\circle*{5}}}
	\put(50,20){\color{blue}{\circle*{5}}}
	\put(60,10){\color{blue}{\circle*{5}}}
	\put(70,10){\color{blue}{\circle*{5}}}
	\put(80,0){\color{blue}{\circle*{5}}}
	\put(80,20){\color{blue}{\circle*{5}}}
	\put(90,10){\color{blue}{\circle*{5}}}
	\put(100,10){\color{blue}{\circle*{5}}}
	\put(110,0){\color{blue}{\circle*{5}}}
	\put(110,20){\color{blue}{\circle*{5}}}
	\put(120,10){\color{blue}{\circle*{5}}}
	\put(130,10){\color{blue}{\circle*{5}}}
	\put(140,0){\color{blue}{\circle*{5}}}
	\put(140,20){\color{blue}{\circle*{5}}}
	\put(150,10){\color{blue}{\circle*{5}}}
	\put(160,10){\color{blue}{\circle*{5}}}
	\put(170,0){\color{blue}{\circle*{5}}}
	\put(170,20){\color{blue}{\circle*{5}}}
	\put(180,10){\color{blue}{\circle*{5}}}
	\put(190,10){\color{blue}{\circle*{5}}}
	\put(200,0){\color{blue}{\circle*{5}}}
	\put(200,20){\color{blue}{\circle*{5}}}
	\put(210,10){\color{blue}{\circle*{5}}}
	\put(220,10){\color{blue}{\circle*{5}}}
	\put(230,0){\color{blue}{\circle*{5}}}
	\put(230,20){\color{blue}{\circle*{5}}}
	\put(240,10){\color{blue}{\circle*{5}}}
	\put(250,10){\color{blue}{\circle*{5}}}
	\put(260,0){\color{blue}{\circle*{5}}}
	\put(260,20){\color{blue}{\circle*{5}}}
	\put(270,10){\color{blue}{\circle*{5}}}
	\put(280,0){\color{blue}{\circle*{5}}}
	\put(280,20){\color{blue}{\circle*{5}}}
	\put(290,10){\color{blue}{\circle*{5}}}

\end{picture}
\caption{Dominating set $D_{29}$ for $G_{3,29}$} \label{patterns29}
\end{figure}

The sets described in Figure \ref{patterns3n} dominate $G_{3,n}$.
The following result proves they are the smallest dominating sets of $G_{3,n}$ and determines the cardinality of each. 

\begin{theorem}
Let $n\ge 3$.  The $(2,2)$ broadcast domination number of the $3\times n$ grid is 
\[ \gamma_{2,2}(G_{3,n}) =  \left \lceil \frac{4n}{3} \right \rceil.\]
\end{theorem}

\begin{proof}

Let $D_n$ denote a dominating set for $G_{3,n}$ given by our construction.  
First we will show that $D_n$ is the smallest dominating set for $G_n$, proving $|D_n| = \gamma_{2,2}(G_{3,n})$.  Then we will show that the formula given above counts the cardinality of $D_n$.

Since each $D_n$ uses the pattern 2-1-1 as many times as possible, the only way a dominating set for $G_{3,n}$ could be more
efficient is if it used the pattern 2-1-1-1. Figure \ref{111} shows an example of a pattern starting with 2-1-1 that dominates the first two columns.
The two vertices in the third column that are circled in red only have a reception strength of 1.  Thus they require two vertices from the fourth column to be selected 
to dominate the third column.  In fact, our MPS program shows that every pattern that starts with 2-1-1 requires us to select two vertices from the fourth column 
\footnote{The calculation is available at \\
 \tt https://cloud.sagemath.com/projects/26b983ae-a894-47c0-bd54-c73b071e555c/files/(2,2)BDPC.sagews}.
Hence, no dominating set of $G_{3,n}$ can contain the pattern 2-1-1-1.  Similarly, the program also shows that the patterns 2-1-0 and 1-1-1 are not dominating subpatterns.  
This proves our construction for $D_n$ is the smallest dominating set for $G_{3,n}$.  

\begin{figure}[H]
\centering
\begin{picture}(150,20)(0,0)
	\linethickness{.1 mm}
	\multiput(5,0)(0,10){3}{\line(1,0){40}}
	\multiput(10,0)(10,0){4}{\line(0,1){20}}
	\put(10,20){\color{blue}{\circle*{5}}}
	\put(10,0){\color{blue}{\circle*{5}}}
	\put(20,10){\color{blue}{\circle*{5}}}
	\put(30,10){\color{blue}{\circle*{5}}}
	\put(30,20){\color{red}{\circle{5}}}
	\put(30,0){\color{red}{\circle{5}}}
	\multiput(105,0)(0,10){3}{\line(1,0){40}}
	\multiput(110,0)(10,0){4}{\line(0,1){20}}
	\put(110,20){\color{blue}{\circle*{5}}}
	\put(110,0){\color{blue}{\circle*{5}}}
	\put(120,10){\color{blue}{\circle*{5}}}
	\put(130,10){\color{blue}{\circle*{5}}}
	\put(140,20){\color{blue}{\circle*{5}}}
	\put(140,0){\color{blue}{\circle*{5}}}
\end{picture}
\caption{Pattern containing 2-1-1} \label{111}
\end{figure}
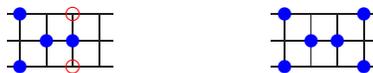

Next, we use the principle of strong mathematical induction to show that $|D_n| = \left \lceil \frac{4n}{3} \right \rceil$.
We start by verifying this formula for three base cases by inspecting $D_3,D_4,$ and $D_5$ in Figure \ref{patterns3n}.  
We thus confirm that \[ |D_3| = 4 = \left \lceil \frac{4(3)}{3} \right \rceil   \hspace{1cm} |D_4| = 6 = \left \lceil \frac{4(4)}{3} \right \rceil \hspace{.5cm}  \text{ and }   \hspace{.5cm} |D_5| = 7 = \left \lceil \frac{4(5)}{3} \right \rceil \]
Suppose by the principle of strong mathematical induction that for all $k \leq n $ the cardinality of the set $D_k$ is given by the formula
 $| D_k | =  \left \lceil \frac{4 k}{3} \right \rceil $.
Consider the set $D_{n+1}$.  It can be constructed from the set $D_{n-2}$ by using $D_{n-2}$ to dominate columns $4$ through $n+1$ of $G_{3,n+1}$
 and using the pattern for $D_3$ to dominate the first three columns of $G_{3,n+1}$. 
Thus the cardinality of $D_n$ is  \[ | D_{n+1}  | = |D_{n-2}|+|D_3| = \left \lceil \frac{4(n-2)}{3} \right \rceil + 4 = \left \lceil \frac{4(n+1)}{3} \right \rceil . \qedhere\]
\end{proof}

\subsection{(2,2) Domination of 4 by n Grids}

We now describe a pattern for constructing an efficient dominating set $D_n$ for any $4 \times n $ grid $G_{4,n}$.
By examining Figure \ref{22}, one can verify that any two adjacent columns can be dominated by a pattern with two vertices in each column.  
Thus the pattern with two vertices in each column dominates any $G_{4,n}$.  
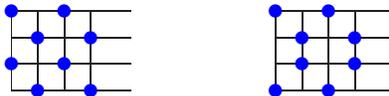
\begin{figure}[H]
\centering
\begin{picture}(165,30)(0,0)
\linethickness{.1 mm}
	\multiput(10,0)(0,10){4}{\line(1,0){45}}
	\multiput(10,0)(10,0){4}{\line(0,1){30}}
	\put(10,30){\color{blue}{\circle*{5}}}
	\put(10,10){\color{blue}{\circle*{5}}}
	\put(20,0){\color{blue}{\circle*{5}}}
	\put(20,20){\color{blue}{\circle*{5}}}
	\put(30,10){\color{blue}{\circle*{5}}}
	\put(30,30){\color{blue}{\circle*{5}}}
	\put(40,0){\color{blue}{\circle*{5}}}
	\put(40,20){\color{blue}{\circle*{5}}}
	\multiput(110,0)(0,10){4}{\line(1,0){45}}
	\multiput(110,0)(10,0){4}{\line(0,1){30}}
	\put(110,30){\color{blue}{\circle*{5}}}
	\put(110,0){\color{blue}{\circle*{5}}}
	\put(120,10){\color{blue}{\circle*{5}}}
	\put(120,20){\color{blue}{\circle*{5}}}
	\put(130,30){\color{blue}{\circle*{5}}}
	\put(130,0){\color{blue}{\circle*{5}}}
	\put(140,10){\color{blue}{\circle*{5}}}
	\put(140,20){\color{blue}{\circle*{5}}}
\end{picture}
\caption{Some dominating subpatterns.} \label{22}
\end{figure}

To obtain a more efficient construction, we start with a similar pattern that has two vertices in each column. 
Then, for every five consecutive columns between columns $2$ and $n-1$, replace the pattern 2-2-2-2-2 with the pattern 2-1-3-1-2.  For instance, the underlined portion of 
2-\underline{2-2-2-2-2}-2 is swapped for 2-\underline{2-1-3-1-2}-2, see Figure \ref{swap}.  

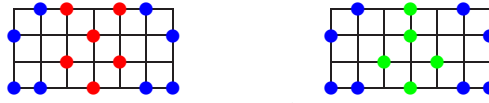
\begin{figure}[H]
\centering
\begin{picture}(200,30)(0,0)
	\linethickness{.1 mm}
	\multiput(10,0)(0,10){4}{\line(1,0){60}}
	\multiput(10,0)(10,0){7}{\line(0,1){30}}
	\put(10,20){\color{blue}{\circle*{5}}}
	\put(10,0){\color{blue}{\circle*{5}}}
	\put(20,30){\color{blue}{\circle*{5}}}
	\put(20,0){\color{blue}{\circle*{5}}}
	\put(30,10){\color{red}{\circle*{5}}}
	\put(30,30){\color{red}{\circle*{5}}}
	\put(40,20){\color{red}{\circle*{5}}}
	\put(40,0){\color{red}{\circle*{5}}}
	\put(50,10){\color{red}{\circle*{5}}}
	\put(50,30){\color{red}{\circle*{5}}}
	\put(60,30){\color{blue}{\circle*{5}}}
	\put(60,0){\color{blue}{\circle*{5}}}
	\put(70,20){\color{blue}{\circle*{5}}}
	\put(70,0){\color{blue}{\circle*{5}}}
	\multiput(130,0)(0,10){4}{\line(1,0){60}}
	\multiput(130,0)(10,0){7}{\line(0,1){30}}
	\put(130,20){\color{blue}{\circle*{5}}}
	\put(130,0){\color{blue}{\circle*{5}}}
	\put(140,30){\color{blue}{\circle*{5}}}
	\put(140,0){\color{blue}{\circle*{5}}}
	\put(150,10){\color{green}{\circle*{5}}}
	\put(160,30){\color{green}{\circle*{5}}}
	\put(160,20){\color{green}{\circle*{5}}}
	\put(160,0){\color{green}{\circle*{5}}}
	\put(170,10){\color{green}{\circle*{5}}}
	\put(180,30){\color{blue}{\circle*{5}}}
	\put(180,0){\color{blue}{\circle*{5}}}
	\put(190,20){\color{blue}{\circle*{5}}}
	\put(190,0){\color{blue}{\circle*{5}}}
\end{picture}
\caption{4 x 7 Grid: The interior allows for a swap to a more efficient pattern.} \label{swap}
\end{figure}

When doing more than one swap, it is necessary flip the orientation of the subset satisfying the pattern 2-1-3-1-2 upside down as shown in Figure \ref{flip}
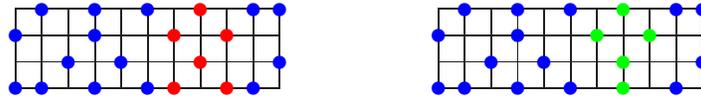
\begin{figure}[H]
\centering
\begin{picture}(280,30)(0,0)
	\linethickness{.1 mm}
	\multiput(10,0)(0,10){4}{\line(1,0){100}}
	\multiput(10,0)(10,0){11}{\line(0,1){30}}
	\put(10,20){\color{blue}{\circle*{5}}}
	\put(10,0){\color{blue}{\circle*{5}}}
	\put(20,30){\color{blue}{\circle*{5}}}
	\put(20,0){\color{blue}{\circle*{5}}}
	\put(30,10){\color{blue}{\circle*{5}}}
	\put(40,30){\color{blue}{\circle*{5}}}
	\put(40,20){\color{blue}{\circle*{5}}}
	\put(40,0){\color{blue}{\circle*{5}}}
	\put(50,10){\color{blue}{\circle*{5}}}
	\put(60,30){\color{blue}{\circle*{5}}}
	\put(60,0){\color{blue}{\circle*{5}}}
	\put(70,20){\color{red}{\circle*{5}}}
	\put(70,0){\color{red}{\circle*{5}}}
	\put(80,30){\color{red}{\circle*{5}}}
	\put(80,10){\color{red}{\circle*{5}}}
	\put(90,20){\color{red}{\circle*{5}}}
	\put(90,0){\color{red}{\circle*{5}}}
	\put(100,30){\color{blue}{\circle*{5}}}
	\put(100,0){\color{blue}{\circle*{5}}}
	\put(110,30){\color{blue}{\circle*{5}}}
	\put(110,10){\color{blue}{\circle*{5}}}
	\multiput(170,0)(0,10){4}{\line(1,0){100}}
	\multiput(170,0)(10,0){11}{\line(0,1){30}}
	\put(170,20){\color{blue}{\circle*{5}}}
	\put(170,0){\color{blue}{\circle*{5}}}
	\put(180,30){\color{blue}{\circle*{5}}}
	\put(180,0){\color{blue}{\circle*{5}}}
	\put(190,10){\color{blue}{\circle*{5}}}
	\put(200,30){\color{blue}{\circle*{5}}}
	\put(200,20){\color{blue}{\circle*{5}}}
	\put(200,0){\color{blue}{\circle*{5}}}
	\put(210,10){\color{blue}{\circle*{5}}}
	\put(220,30){\color{blue}{\circle*{5}}}
	\put(220,0){\color{blue}{\circle*{5}}}
	\put(230,20){\color{green}{\circle*{5}}}
	\put(240,30){\color{green}{\circle*{5}}}
	\put(240,10){\color{green}{\circle*{5}}}
	\put(240,0){\color{green}{\circle*{5}}}
	\put(250,20){\color{green}{\circle*{5}}}
	\put(260,30){\color{blue}{\circle*{5}}}
	\put(260,0){\color{blue}{\circle*{5}}}
	\put(270,30){\color{blue}{\circle*{5}}}
	\put(270,10){\color{blue}{\circle*{5}}}
\end{picture}
\caption{$4 \times 11$ Grid: Another swap to a more efficient pattern.} \label{flip}
\end{figure}

The resulting recursive constructions of $D_n$ are depicted in Figure \ref{fig:4xn}.
We note that the reason we start and end our pattern with 2-2 is that it is impossible for a dominating set to start off with the pattern 2-1-3, as the vertices in the first column MPS not all have reception strengths of at least 2.  Moreover, starting with 
the pattern 3-1-2 gives the same dominating number as 2-2-2, which is less efficient than our pattern of 2-2-1 once $n\ge 7$.

\begin{figure}[H]
\centering
\begin{picture}(450,30)(0,0)
	\linethickness{.1 mm}
	\multiput(10,0)(0,10){4}{\line(1,0){30}}
	\multiput(10,0)(10,0){4}{\line(0,1){30}}
	\put(10,20){\color{blue}{\circle*{5}}}
	\put(10,0){\color{blue}{\circle*{5}}}
	\put(20,30){\color{blue}{\circle*{5}}}
	\put(20,0){\color{blue}{\circle*{5}}}
	\put(30,30){\color{blue}{\circle*{5}}}
	\put(30,10){\color{blue}{\circle*{5}}}
	\put(40,20){\color{blue}{\circle*{5}}}
	\put(40,0){\color{blue}{\circle*{5}}}
	\multiput(60,0)(0,10){4}{\line(1,0){40}}
	\multiput(60,0)(10,0){5}{\line(0,1){30}}
 	\put(60,20){\color{blue}{\circle*{5}}}
 	\put(60,0){\color{blue}{\circle*{5}}}
	\put(70,30){\color{blue}{\circle*{5}}}
 	\put(70,0){\color{blue}{\circle*{5}}}
	\put(80,30){\color{blue}{\circle*{5}}}
 	\put(80,10){\color{blue}{\circle*{5}}}
 	\put(90,30){\color{blue}{\circle*{5}}}
 	\put(90,0){\color{blue}{\circle*{5}}}
	\put(100,20){\color{blue}{\circle*{5}}}
 	\put(100,0){\color{blue}{\circle*{5}}}
	\multiput(120,0)(0,10){4}{\line(1,0){50}}
 	\multiput(120,0)(10,0){6}{\line(0,1){30}}
 	\put(120,20){\color{blue}{\circle*{5}}}
 	\put(120,0){\color{blue}{\circle*{5}}}
	\put(130,30){\color{blue}{\circle*{5}}}
 	\put(130,0){\color{blue}{\circle*{5}}}
	\put(140,30){\color{blue}{\circle*{5}}}
 	\put(140,10){\color{blue}{\circle*{5}}}
	\put(150,30){\color{blue}{\circle*{5}}}
	\put(150,10){\color{blue}{\circle*{5}}}
	\put(160,30){\color{blue}{\circle*{5}}}
 	\put(160,0){\color{blue}{\circle*{5}}}
	\put(170,20){\color{blue}{\circle*{5}}}
 	\put(170,0){\color{blue}{\circle*{5}}}
	\multiput(190,0)(0,10){4}{\line(1,0){60}}
	\multiput(190,0)(10,0){7}{\line(0,1){30}}
	\put(190,20){\color{blue}{\circle*{5}}}
	\put(190,0){\color{blue}{\circle*{5}}}
	\put(200,30){\color{blue}{\circle*{5}}}
	\put(200,0){\color{blue}{\circle*{5}}}
	\put(210,10){\color{blue}{\circle*{5}}}
	\put(220,30){\color{blue}{\circle*{5}}}
	\put(220,20){\color{blue}{\circle*{5}}}
	\put(220,0){\color{blue}{\circle*{5}}}
	\put(230,10){\color{blue}{\circle*{5}}}
	\put(240,30){\color{blue}{\circle*{5}}}
	\put(240,0){\color{blue}{\circle*{5}}}
	\put(250,20){\color{blue}{\circle*{5}}}
	\put(250,0){\color{blue}{\circle*{5}}}
	\multiput(270,0)(0,10){4}{\line(1,0){70}}
	\multiput(270,0)(10,0){8}{\line(0,1){30}}
	\put(270,20){\color{blue}{\circle*{5}}}
	\put(270,0){\color{blue}{\circle*{5}}}
	\put(280,30){\color{blue}{\circle*{5}}}
	\put(280,0){\color{blue}{\circle*{5}}}
	\put(290,10){\color{blue}{\circle*{5}}}
	\put(300,30){\color{blue}{\circle*{5}}}
	\put(300,20){\color{blue}{\circle*{5}}}
	\put(300,0){\color{blue}{\circle*{5}}}
	\put(310,10){\color{blue}{\circle*{5}}}
	\put(320,30){\color{blue}{\circle*{5}}}
	\put(320,0){\color{blue}{\circle*{5}}}
	\put(330,20){\color{blue}{\circle*{5}}}
	\put(330,0){\color{blue}{\circle*{5}}}
	\put(340,30){\color{blue}{\circle*{5}}}
	\put(340,10){\color{blue}{\circle*{5}}}
	\multiput(360,0)(0,10){4}{\line(1,0){80}}
	\multiput(360,0)(10,0){9}{\line(0,1){30}}
	\put(360,20){\color{blue}{\circle*{5}}}
	\put(360,0){\color{blue}{\circle*{5}}}
	\put(370,30){\color{blue}{\circle*{5}}}
	\put(370,0){\color{blue}{\circle*{5}}}
	\put(380,10){\color{blue}{\circle*{5}}}
	\put(390,30){\color{blue}{\circle*{5}}}
	\put(390,20){\color{blue}{\circle*{5}}}
	\put(390,0){\color{blue}{\circle*{5}}}
	\put(400,10){\color{blue}{\circle*{5}}}
	\put(410,30){\color{blue}{\circle*{5}}}
	\put(410,0){\color{blue}{\circle*{5}}}
	\put(420,20){\color{blue}{\circle*{5}}}
	\put(420,0){\color{blue}{\circle*{5}}}
	\put(430,30){\color{blue}{\circle*{5}}}
	\put(430,10){\color{blue}{\circle*{5}}}
	\put(440,20){\color{blue}{\circle*{5}}}
	\put(440,0){\color{blue}{\circle*{5}}}
\end{picture}
\end{figure}

\begin{figure}[H]
\centering
\begin{picture}(380,30)(0,0)
	\linethickness{.1 mm}
	\multiput(10,0)(0,10){4}{\line(1,0){90}}
	\multiput(10,0)(10,0){10}{\line(0,1){30}}
	\put(10,20){\color{blue}{\circle*{5}}}
	\put(10,0){\color{blue}{\circle*{5}}}
	\put(20,30){\color{blue}{\circle*{5}}}
	\put(20,0){\color{blue}{\circle*{5}}}
	\put(30,10){\color{blue}{\circle*{5}}}
	\put(40,30){\color{blue}{\circle*{5}}}
	\put(40,20){\color{blue}{\circle*{5}}}
	\put(40,0){\color{blue}{\circle*{5}}}
	\put(50,10){\color{blue}{\circle*{5}}}
	\put(60,30){\color{blue}{\circle*{5}}}
	\put(60,0){\color{blue}{\circle*{5}}}
	\put(70,20){\color{blue}{\circle*{5}}}
	\put(70,0){\color{blue}{\circle*{5}}}
	\put(80,30){\color{blue}{\circle*{5}}}
	\put(80,10){\color{blue}{\circle*{5}}}
	\put(90,30){\color{blue}{\circle*{5}}}
	\put(90,0){\color{blue}{\circle*{5}}}
	\put(100,20){\color{blue}{\circle*{5}}}
	\put(100,0){\color{blue}{\circle*{5}}}
	\multiput(120,0)(0,10){4}{\line(1,0){100}}
	\multiput(120,0)(10,0){11}{\line(0,1){30}}
	\put(120,20){\color{blue}{\circle*{5}}}
	\put(120,0){\color{blue}{\circle*{5}}}
	\put(130,30){\color{blue}{\circle*{5}}}
	\put(130,0){\color{blue}{\circle*{5}}}
	\put(140,10){\color{blue}{\circle*{5}}}
	\put(150,30){\color{blue}{\circle*{5}}}
	\put(150,20){\color{blue}{\circle*{5}}}
	\put(150,0){\color{blue}{\circle*{5}}}
	\put(160,10){\color{blue}{\circle*{5}}}
	\put(170,30){\color{blue}{\circle*{5}}}
	\put(170,0){\color{blue}{\circle*{5}}}
	\put(180,20){\color{blue}{\circle*{5}}}
	\put(190,30){\color{blue}{\circle*{5}}}
	\put(190,10){\color{blue}{\circle*{5}}}
	\put(190,0){\color{blue}{\circle*{5}}}
	\put(200,20){\color{blue}{\circle*{5}}}
	\put(210,30){\color{blue}{\circle*{5}}}
	\put(210,0){\color{blue}{\circle*{5}}}
	\put(220,30){\color{blue}{\circle*{5}}}
	\put(220,10){\color{blue}{\circle*{5}}}
	\multiput(240,0)(0,10){4}{\line(1,0){110}}
	\multiput(240,0)(10,0){12}{\line(0,1){30}}
	\put(240,20){\color{blue}{\circle*{5}}}
	\put(240,0){\color{blue}{\circle*{5}}}
	\put(250,30){\color{blue}{\circle*{5}}}
	\put(250,0){\color{blue}{\circle*{5}}}
	\put(260,10){\color{blue}{\circle*{5}}}
	\put(270,30){\color{blue}{\circle*{5}}}
	\put(270,20){\color{blue}{\circle*{5}}}
	\put(270,0){\color{blue}{\circle*{5}}}
	\put(280,10){\color{blue}{\circle*{5}}}
	\put(290,30){\color{blue}{\circle*{5}}}
	\put(290,0){\color{blue}{\circle*{5}}}
	\put(300,20){\color{blue}{\circle*{5}}}
	\put(310,30){\color{blue}{\circle*{5}}}
	\put(310,10){\color{blue}{\circle*{5}}}
	\put(310,0){\color{blue}{\circle*{5}}}
	\put(320,20){\color{blue}{\circle*{5}}}
	\put(330,30){\color{blue}{\circle*{5}}}
	\put(330,0){\color{blue}{\circle*{5}}}
	\put(340,30){\color{blue}{\circle*{5}}}
	\put(340,10){\color{blue}{\circle*{5}}}
	\put(350,20){\color{blue}{\circle*{5}}}
	\put(350,0){\color{blue}{\circle*{5}}}

\end{picture}
\caption{Dominating Sets for $4 \times 4$ to $4 \times 12$} \label{fig:4xn}
\end{figure}
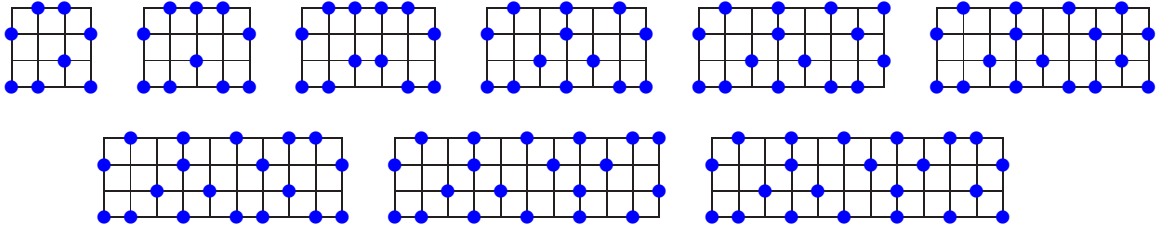

We now prove that the dominating sets $D_n$ we just described are the smallest dominating sets of $G_{4,n}$ for any $n$.

\begin{theorem}
Let $n\ge 4$.  The $(2,2)$ broadcast domination number of the $4\times n$ grid is
\[ \gamma_{2,2}(G_{4,n}) =  2n - \left \lceil\frac{n-6}{4} \right \rceil.\]
\end{theorem}

\begin{proof}
We will show that the dominating set $D_n$ described above and depicted in Figure \ref{fig:4xn} has the minimum cardinality of any (2,2) broadcast dominating set of $G_{4,n}.$
Thus its cardinality is the (2,2) broadcast domination number of $G_{4,n}$.  We will begin by showing this cardinality is given by the above formula.

It is easy to verify that for $n < 7$ the number of vertices in the dominating sets $D_n$ is $2n - \left \lceil\frac{n-6}{4} \right \rceil $.
When $n \geq 7$ the construction of $D_{n}$ from $D_{n-1}$ implies that their cardinalities satisfy the formula
\[ |D_n | = \begin{cases} |D_{n-1}| +1 &  \text{ if  } n \equiv 3\text{ mod } 4 \text{ and } n\ge 7 \\ |D_{n-1}|+2 & \text{ otherwise } \end{cases}.\] 
 Furthermore, for $n \geq 7$ the numbers $2n - \left \lceil\frac{n-6}{4} \right \rceil$ satisfy the same equations.
 Hence, a simple induction argument shows that the given formula counts the number of vertices in this dominating set $D_n$.

Next we will prove that this pattern is the most efficient way to dominate a $4 \times n$ grid $G_{4,n}$ by showing that no pattern with fewer vertices can dominate $G_{4,n}$.  
For a pattern to be more efficient than the one given, it must contain one of the subpatterns 2-1-2 or 3-1-1.  
However, similar to $3\times n$ domination, we used the MPS algorithm to examine all possible patterns that start with a 2-1 and showed that 2-1-3 is the smallest pattern that dominates three consecutive rows.  
The program also returned that any pattern starting with 3-1 must proceed as 3-1-2.    

In particular, Figure \ref{21} shows the only subpattern starting 2-1 that satisfies Property 1 of a dominating subpattern.
The vertices in the second column that are circled in red all have weight $1$ and the one in the third column has weight 0. 
(Hence this set fails to satisfy Property 2 of a dominating subpattern.)
These circled vertices require three vertices in column 3 to be selected.

\begin{figure}[H]
\centering
\begin{picture}(155,30)(0,0)
	\linethickness{.1 mm}
	\multiput(5,0)(0,10){4}{\line(1,0){45}}
	\multiput(10,0)(10,0){4}{\line(0,1){30}} 
	\put(10,30){\color{blue}{\circle*{5}}}
	\put(10,0){\color{blue}{\circle*{5}}}
	\put(20,10){\color{blue}{\circle*{5}}}
	\put(20,30){\color{red}{\circle{5}}}
	\put(20,20){\color{red}{\circle{5}}}
	\put(30,0){\color{red}{\circle{5}}}
	\multiput(105,0)(0,10){4}{\line(1,0){45}}
	\multiput(110,0)(10,0){4}{\line(0,1){30}} 
	\put(110,30){\color{blue}{\circle*{5}}}
	\put(110,0){\color{blue}{\circle*{5}}}
	\put(120,10){\color{blue}{\circle*{5}}}
	\put(130,30){\color{blue}{\circle*{5}}}
	\put(130,20){\color{blue}{\circle*{5}}}
	\put(130,0){\color{blue}{\circle*{5}}}
\end{picture}
\caption{Example of 2-1 pattern}  \label{21}
\end{figure}
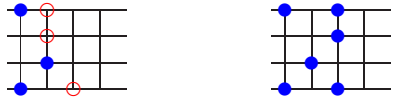
Any other 2-1 configuration does not dominate the first two columns. 
The MPS program also verifies that this is the case for all ${ 4 \choose 2}{4 \choose 1}{4 \choose 2}$ subsets starting with the pattern 2-1-2.  
Hence, a 2-1 configuration must lead to a 2-1-3 configuration. 
Thus, the smallest dominating subpattern that starts with 2-1 is 2-1-3.

Now consider the case where the graph is dominated by a set containing a column with 3 vertices and then a column with 1 vertex. 
Then the vertices circled in red will have a reception strength of $1$ or $0$ depending on whether they are in the second or third column respectively.  (Without loss of generality, these two vertices exist, although their locations may change.)
They each require a vertex in the third column to be selected to dominate them.  

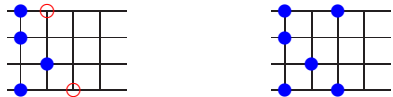
\begin{figure}[H]
\centering
\begin{picture}(155,30)(0,0)
	\linethickness{.1 mm}
	\multiput(5,0)(0,10){4}{\line(1,0){45}}
	\multiput(10,0)(10,0){4}{\line(0,1){30}}
	\put(10,30){\color{blue}{\circle*{5}}}
	\put(10,20){\color{blue}{\circle*{5}}}
	\put(10,0){\color{blue}{\circle*{5}}}
	\put(20,10){\color{blue}{\circle*{5}}}
	\put(30,0){\color{red}{\circle{5}}}
    \put(20,30){\color{red}{\circle{5}}}
	\multiput(105,0)(0,10){4}{\line(1,0){45}}
	\multiput(110,0)(10,0){4}{\line(0,1){30}} 
	\put(110,30){\color{blue}{\circle*{5}}}
	\put(110,0){\color{blue}{\circle*{5}}}
	\put(110,20){\color{blue}{\circle*{5}}}
	\put(120,10){\color{blue}{\circle*{5}}}
	\put(130,30){\color{blue}{\circle*{5}}}
	\put(130,0){\color{blue}{\circle*{5}}}
\end{picture}
\caption{Example of a 3-1 pattern}
\end{figure}
The MPS program verifies that no pattern containing 3-1-1 or 2-1-2 can dominate $G_{4,n}$.  
Hence the dominating sets $D_n$ given above, which use the subpattern 2-1-3-1 as much as possible, are the smallest (2,2) broadcast dominating sets of $G_{4,n}$. 
\end{proof}

\subsection{(2,2) Domination of 5 by n Grids}
In this section, we provide a pattern that determines a dominating set for any $5 \times n$ grid. 
Then we prove that it is the smallest dominating set of $G_{5,n}$, and thereby determine the (2,2) broadcast domination number for each grid.

For small cases with $n \leq 10$ we provide a pattern in Figure \ref{5n}.

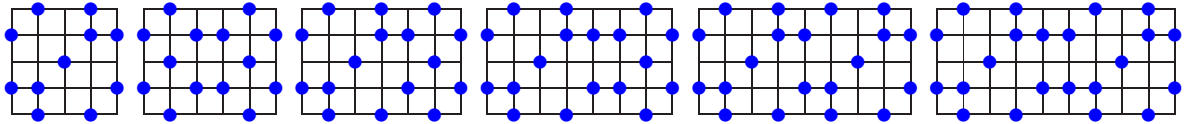
\begin{figure}[H]
\centering
\begin{picture}(450,40)(0,0)
\linethickness{.1 mm}
	\multiput(10,0)(0,10){5}{\line(1,0){40}}
	\multiput(10,0)(10,0){5}{\line(0,1){40}}
	\put(10,30){\color{blue}{\circle*{5}}}
	\put(10,10){\color{blue}{\circle*{5}}}
	\put(20,40){\color{blue}{\circle*{5}}}
	\put(20,0){\color{blue}{\circle*{5}}}
	\put(20,10){\color{blue}{\circle*{5}}}
	\put(30,20){\color{blue}{\circle*{5}}}
	\put(40,30){\color{blue}{\circle*{5}}}
	\put(40,40){\color{blue}{\circle*{5}}}
	\put(40,0){\color{blue}{\circle*{5}}}
	\put(50,30){\color{blue}{\circle*{5}}}
	\put(50,10){\color{blue}{\circle*{5}}}
	\multiput(60,0)(0,10){5}{\line(1,0){50}}
	\multiput(60,0)(10,0){6}{\line(0,1){40}}
	\put(60,30){\color{blue}{\circle*{5}}}
	\put(60,10){\color{blue}{\circle*{5}}}
	\put(70,40){\color{blue}{\circle*{5}}}
	\put(70,0){\color{blue}{\circle*{5}}}
	\put(70,20){\color{blue}{\circle*{5}}}
	\put(80,10){\color{blue}{\circle*{5}}}
	\put(80,30){\color{blue}{\circle*{5}}}
	\put(90,10){\color{blue}{\circle*{5}}}
	\put(90,30){\color{blue}{\circle*{5}}}
	\put(100,0){\color{blue}{\circle*{5}}}
	\put(100,20){\color{blue}{\circle*{5}}}
	\put(100,40){\color{blue}{\circle*{5}}}
	\put(110,10){\color{blue}{\circle*{5}}}
	\put(110,30){\color{blue}{\circle*{5}}}
	\multiput(120,0)(0,10){5}{\line(1,0){60}}
	\multiput(120,0)(10,0){7}{\line(0,1){40}}
	\put(120,30){\color{blue}{\circle*{5}}}
	\put(120,10){\color{blue}{\circle*{5}}}
	\put(130,40){\color{blue}{\circle*{5}}}
	\put(130,0){\color{blue}{\circle*{5}}}
	\put(130,10){\color{blue}{\circle*{5}}}
	\put(140,20){\color{blue}{\circle*{5}}}
	\put(150,30){\color{blue}{\circle*{5}}}
	\put(150,40){\color{blue}{\circle*{5}}}
	\put(150,0){\color{blue}{\circle*{5}}}
	\put(160,30){\color{blue}{\circle*{5}}}
	\put(160,10){\color{blue}{\circle*{5}}}
	\put(170,40){\color{blue}{\circle*{5}}}
	\put(170,20){\color{blue}{\circle*{5}}}
	\put(170,0){\color{blue}{\circle*{5}}}
	\put(180,30){\color{blue}{\circle*{5}}}
	\put(180,10){\color{blue}{\circle*{5}}}
	\multiput(190,0)(0,10){5}{\line(1,0){70}}
	\multiput(190,0)(10,0){8}{\line(0,1){40}}
	\put(190,30){\color{blue}{\circle*{5}}}
	\put(190,10){\color{blue}{\circle*{5}}}
	\put(200,40){\color{blue}{\circle*{5}}}
	\put(200,0){\color{blue}{\circle*{5}}}
	\put(200,10){\color{blue}{\circle*{5}}}
	\put(210,20){\color{blue}{\circle*{5}}}
	\put(220,30){\color{blue}{\circle*{5}}}
	\put(220,40){\color{blue}{\circle*{5}}}
	\put(220,0){\color{blue}{\circle*{5}}}
	\put(230,30){\color{blue}{\circle*{5}}}
	\put(230,10){\color{blue}{\circle*{5}}}
	\put(240,30){\color{blue}{\circle*{5}}}
	\put(240,10){\color{blue}{\circle*{5}}}
	\put(250,40){\color{blue}{\circle*{5}}}
	\put(250,20){\color{blue}{\circle*{5}}}
	\put(250,0){\color{blue}{\circle*{5}}}
	\put(260,30){\color{blue}{\circle*{5}}}
	\put(260,10){\color{blue}{\circle*{5}}}
	\multiput(270,0)(0,10){5}{\line(1,0){80}}
	\multiput(270,0)(10,0){9}{\line(0,1){40}}
	\put(270,30){\color{blue}{\circle*{5}}}
	\put(270,10){\color{blue}{\circle*{5}}}
	\put(280,40){\color{blue}{\circle*{5}}}
	\put(280,0){\color{blue}{\circle*{5}}}
	\put(280,10){\color{blue}{\circle*{5}}}
	\put(290,20){\color{blue}{\circle*{5}}}
	\put(300,30){\color{blue}{\circle*{5}}}
	\put(300,40){\color{blue}{\circle*{5}}}
	\put(300,0){\color{blue}{\circle*{5}}}
	\put(310,30){\color{blue}{\circle*{5}}}
	\put(310,10){\color{blue}{\circle*{5}}}
	\put(320,0){\color{blue}{\circle*{5}}}
	\put(320,10){\color{blue}{\circle*{5}}}
	\put(320,40){\color{blue}{\circle*{5}}}
	\put(330,20){\color{blue}{\circle*{5}}}
	\put(340,40){\color{blue}{\circle*{5}}}
	\put(340,30){\color{blue}{\circle*{5}}}
	\put(340,0){\color{blue}{\circle*{5}}}
	\put(350,30){\color{blue}{\circle*{5}}}
	\put(350,10){\color{blue}{\circle*{5}}}
	\multiput(360,0)(0,10){5}{\line(1,0){90}}
	\multiput(360,0)(10,0){10}{\line(0,1){40}}
	\put(360,30){\color{blue}{\circle*{5}}}
	\put(360,10){\color{blue}{\circle*{5}}}
	\put(370,40){\color{blue}{\circle*{5}}}
	\put(370,0){\color{blue}{\circle*{5}}}
	\put(370,10){\color{blue}{\circle*{5}}}
	\put(380,20){\color{blue}{\circle*{5}}}
	\put(390,30){\color{blue}{\circle*{5}}}
	\put(390,40){\color{blue}{\circle*{5}}}
	\put(390,0){\color{blue}{\circle*{5}}}
	\put(400,30){\color{blue}{\circle*{5}}}
	\put(400,10){\color{blue}{\circle*{5}}}
	\put(410,30){\color{blue}{\circle*{5}}}
	\put(410,10){\color{blue}{\circle*{5}}}
	\put(420,40){\color{blue}{\circle*{5}}}
	\put(420,0){\color{blue}{\circle*{5}}}
	\put(420,10){\color{blue}{\circle*{5}}}
	\put(430,20){\color{blue}{\circle*{5}}}
	\put(440,30){\color{blue}{\circle*{5}}}
	\put(440,40){\color{blue}{\circle*{5}}}
	\put(440,0){\color{blue}{\circle*{5}}}
	\put(450,30){\color{blue}{\circle*{5}}}
	\put(450,10){\color{blue}{\circle*{5}}}
\end{picture} 
\caption{Dominating sets for small cases $G_{5,n}$} \label{5n}
\end{figure}
Next we describe how to construct a dominating set for $G_{5,n}$ when $n > 10$.
We start and end the graph with 2-3-1-3 and 3-1-3-2 respectively as shown in Figure \ref{ends}.

\begin{figure}[H]
\centering
\begin{picture}(210,40)(0,0)
	\multiput(0,0)(0,10){5}{\line(1,0){190}}
	\multiput(0,0)(10,0){20}{\line(0,1){40}}
	\put(0,10){\color{blue}{\circle*{5}}}
	\put(0,30){\color{blue}{\circle*{5}}}
	\put(10,0){\color{blue}{\circle*{5}}}
	\put(10,10){\color{blue}{\circle*{5}}}
	\put(10,40){\color{blue}{\circle*{5}}}
	\put(20,20){\color{blue}{\circle*{5}}}
	\put(30,0){\color{blue}{\circle*{5}}} 
	\put(30,30){\color{blue}{\circle*{5}}}
	\put(30,40){\color{blue}{\circle*{5}}}
	\put(160,0){\color{blue}{\circle*{5}}}
	\put(160,10){\color{blue}{\circle*{5}}}
	\put(160,40){\color{blue}{\circle*{5}}}
	\put(170,20){\color{blue}{\circle*{5}}}
	\put(180,0){\color{blue}{\circle*{5}}}
	\put(180,30){\color{blue}{\circle*{5}}}
	\put(180,40){\color{blue}{\circle*{5}}}
	\put(190,10){\color{blue}{\circle*{5}}}
	\put(190,30){\color{blue}{\circle*{5}}}
\end{picture}
\caption{The two end patterns}  \label{ends}
\end{figure}
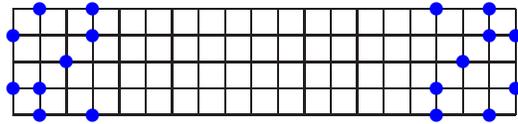
Then moving from left to right we fill in the open columns with the pattern
 2-2-2-2-3-1-3 until it no longer fits, as shown in Figure \ref{first}. Unlike the $4\times n$ case, this pattern does not require an orientation change when it is inserted multiple times.

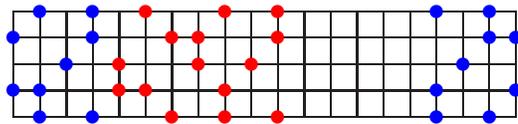
\begin{figure}[H]
\centering
\begin{picture}(210,40)(0,0)
	\multiput(0,0)(0,10){5}{\line(1,0){190}}
	\multiput(0,0)(10,0){20}{\line(0,1){40}}
	\put(0,10){\color{blue}{\circle*{5}}}
	\put(0,30){\color{blue}{\circle*{5}}}
	\put(10,0){\color{blue}{\circle*{5}}}
	\put(10,10){\color{blue}{\circle*{5}}}
	\put(10,40){\color{blue}{\circle*{5}}}
	\put(20,20){\color{blue}{\circle*{5}}}
	\put(30,0){\color{blue}{\circle*{5}}}
	\put(30,30){\color{blue}{\circle*{5}}}
	\put(30,40){\color{blue}{\circle*{5}}}
	\put(40,20){\color{red}{\circle*{5}}}
	\put(40,10){\color{red}{\circle*{5}}}
	\put(50,40){\color{red}{\circle*{5}}}
	\put(50,10){\color{red}{\circle*{5}}}
	\put(60,30){\color{red}{\circle*{5}}}
	\put(60,0){\color{red}{\circle*{5}}}
	\put(70,30){\color{red}{\circle*{5}}}
	\put(70,20){\color{red}{\circle*{5}}}
	\put(80,0){\color{red}{\circle*{5}}}
	\put(80,10){\color{red}{\circle*{5}}}
	\put(80,40){\color{red}{\circle*{5}}}
	\put(90,20){\color{red}{\circle*{5}}}	
	\put(100,0){\color{red}{\circle*{5}}}	
	\put(100,30){\color{red}{\circle*{5}}}	
	\put(100,40){\color{red}{\circle*{5}}}	
	\put(160,0){\color{blue}{\circle*{5}}}
	\put(160,10){\color{blue}{\circle*{5}}}
	\put(160,40){\color{blue}{\circle*{5}}}
	\put(170,20){\color{blue}{\circle*{5}}}
	\put(180,0){\color{blue}{\circle*{5}}}
	\put(180,30){\color{blue}{\circle*{5}}}
	\put(180,40){\color{blue}{\circle*{5}}}
	\put(190,10){\color{blue}{\circle*{5}}}
	\put(190,30){\color{blue}{\circle*{5}}}
\end{picture}
\caption{Filling the empty columns with 2-2-2-2-3-1-3} \label{first}
\end{figure}

When less than seven columns remain undominated, we fill in the remaining columns with the patterns depicted in Figures \ref{2-patterns} and \ref{fillpattern}. 

\begin{figure}[H]
\centering
\begin{picture}(330,40)(0,0)
	\multiput(0,0)(0,10){5}{\line(1,0){60}}
	\multiput(10,0)(10,0){5}{\line(0,1){40}}
	\put(10,20){\color{blue}{\circle*{5}}}
	\put(20,0){\color{blue}{\circle*{5}}}
	\put(20,30){\color{blue}{\circle*{5}}}
	\put(20,40){\color{blue}{\circle*{5}}}
	\put(30,10){\color{red}{\circle*{5}}}
	\put(30,30){\color{red}{\circle*{5}}}
	\put(40,0){\color{blue}{\circle*{5}}}
	\put(40,10){\color{blue}{\circle*{5}}}
	\put(40,40){\color{blue}{\circle*{5}}}
	\put(50,20){\color{blue}{\circle*{5}}}
	\multiput(70,0)(0,10){5}{\line(1,0){70}}
	\multiput(80,0)(10,0){6}{\line(0,1){40}}
	\put(80,20){\color{blue}{\circle*{5}}}
	\put(90,0){\color{blue}{\circle*{5}}}
	\put(90,30){\color{blue}{\circle*{5}}}
	\put(90,40){\color{blue}{\circle*{5}}}
	\put(100,10){\color{red}{\circle*{5}}}
	\put(100,30){\color{red}{\circle*{5}}}
	\put(110,10){\color{red}{\circle*{5}}}
	\put(110,30){\color{red}{\circle*{5}}}
	\put(120,0){\color{blue}{\circle*{5}}}
	\put(120,10){\color{blue}{\circle*{5}}}
	\put(120,40){\color{blue}{\circle*{5}}}
	\put(130,20){\color{blue}{\circle*{5}}}
	\multiput(150,0)(0,10){5}{\line(1,0){80}}
	\multiput(160,0)(10,0){7}{\line(0,1){40}}
	\put(160,20){\color{blue}{\circle*{5}}}
	\put(170,0){\color{blue}{\circle*{5}}}
	\put(170,30){\color{blue}{\circle*{5}}}
	\put(170,40){\color{blue}{\circle*{5}}}
	\put(180,20){\color{red}{\circle*{5}}}
	\put(180,30){\color{red}{\circle*{5}}}
	\put(190,0){\color{red}{\circle*{5}}}
	\put(190,40){\color{red}{\circle*{5}}}
	\put(200,10){\color{red}{\circle*{5}}}
	\put(200,20){\color{red}{\circle*{5}}}
	\put(210,0){\color{blue}{\circle*{5}}}
	\put(210,10){\color{blue}{\circle*{5}}}
	\put(210,40){\color{blue}{\circle*{5}}}
	\put(220,20){\color{blue}{\circle*{5}}}
	\multiput(240,0)(0,10){5}{\line(1,0){90}}
	\multiput(250,0)(10,0){8}{\line(0,1){40}}
	\put(250,20){\color{blue}{\circle*{5}}}
	\put(260,0){\color{blue}{\circle*{5}}}
	\put(260,30){\color{blue}{\circle*{5}}}
	\put(260,40){\color{blue}{\circle*{5}}}
	\put(270,20){\color{red}{\circle*{5}}}
	\put(270,30){\color{red}{\circle*{5}}}
	\put(280,0){\color{red}{\circle*{5}}}
	\put(280,30){\color{red}{\circle*{5}}}
	\put(290,10){\color{red}{\circle*{5}}}
	\put(290,40){\color{red}{\circle*{5}}}
	\put(300,10){\color{red}{\circle*{5}}}
	\put(300,20){\color{red}{\circle*{5}}}
	\put(310,0){\color{blue}{\circle*{5}}}
	\put(310,10){\color{blue}{\circle*{5}}}
	\put(310,40){\color{blue}{\circle*{5}}}
	\put(320,20){\color{blue}{\circle*{5}}}
\end{picture}
\caption{Patterns for filling 1, 2, 3, and 4 columns} \label{2-patterns}
\end{figure}
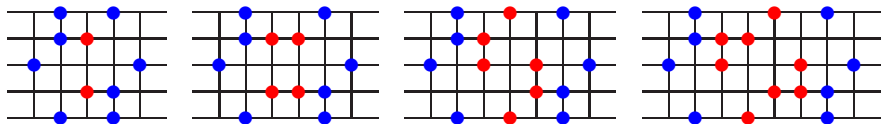

\begin{figure}[H]
\centering
\begin{picture}(330,40)(-50,0)

	\multiput(120,0)(0,10){5}{\line(1,0){110}}
	\multiput(130,0)(10,0){10}{\line(0,1){40}}
        \put(130,20){\color{blue}{\circle*{5}}}
        \put(140,0){\color{blue}{\circle*{5}}}
        \put(140,30){\color{blue}{\circle*{5}}}
	\put(140,40){\color{blue}{\circle*{5}}}
        \put(150,30){\color{red}{\circle*{5}}}
	\put(150,10){\color{red}{\circle*{5}}}
	\put(160,30){\color{red}{\circle*{5}}}
	\put(160,10){\color{red}{\circle*{5}}}
	\put(170,40){\color{red}{\circle*{5}}}
	\put(170,0){\color{red}{\circle*{5}}}
	\put(170,10){\color{red}{\circle*{5}}}
	\put(180,20){\color{red}{\circle*{5}}}
	\put(190,30){\color{red}{\circle*{5}}}
	\put(190,40){\color{red}{\circle*{5}}}
	\put(190,0){\color{red}{\circle*{5}}}
	\put(200,30){\color{red}{\circle*{5}}}
	\put(200,10){\color{red}{\circle*{5}}}
	\put(210,0){\color{blue}{\circle*{5}}}
	\put(210,10){\color{blue}{\circle*{5}}}
	\put(210,40){\color{blue}{\circle*{5}}}
	\put(220,20){\color{blue}{\circle*{5}}}

	\multiput(0,0)(0,10){5}{\line(1,0){100}}
	\multiput(10,0)(10,0){9}{\line(0,1){40}}
        \put(10,20){\color{blue}{\circle*{5}}}
        \put(20,0){\color{blue}{\circle*{5}}}
        \put(20,30){\color{blue}{\circle*{5}}}
	\put(20,40){\color{blue}{\circle*{5}}}
        \put(30,30){\color{red}{\circle*{5}}}
	\put(30,10){\color{red}{\circle*{5}}}
	\put(40,40){\color{red}{\circle*{5}}}
	\put(40,0){\color{red}{\circle*{5}}}
	\put(40,10){\color{red}{\circle*{5}}}
	\put(50,20){\color{red}{\circle*{5}}}
	\put(60,30){\color{red}{\circle*{5}}}
	\put(60,40){\color{red}{\circle*{5}}}
	\put(60,0){\color{red}{\circle*{5}}}
	\put(70,30){\color{red}{\circle*{5}}}
	\put(70,10){\color{red}{\circle*{5}}}
	\put(80,0){\color{blue}{\circle*{5}}}
	\put(80,10){\color{blue}{\circle*{5}}}
	\put(80,40){\color{blue}{\circle*{5}}}
	\put(90,20){\color{blue}{\circle*{5}}}

\end{picture} 
\caption{Patterns for filling 5 and 6 columns} \label{fillpattern}
\end{figure}
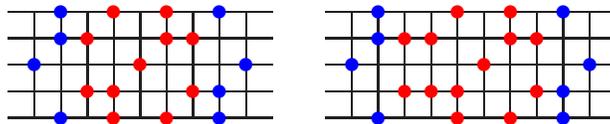

For instance, to complete the dominating set appearing in Figure \ref{first}, 
we observe that there are 5 columns to be dominated.  Hence 2-3-1-3-2 is the appropriate pattern to complete the dominating set, seen in Figure \ref{again}.

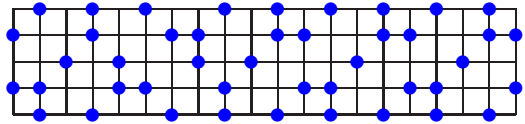
\begin{figure}[H]
\centering
\begin{picture}(210,40)(0,0)
	\multiput(0,0)(0,10){5}{\line(1,0){190}}
	\multiput(0,0)(10,0){20}{\line(0,1){40}}
	\put(0,10){\color{blue}{\circle*{5}}}
	\put(0,30){\color{blue}{\circle*{5}}}
	\put(10,0){\color{blue}{\circle*{5}}}
	\put(10,10){\color{blue}{\circle*{5}}}
	\put(10,40){\color{blue}{\circle*{5}}}
	\put(20,20){\color{blue}{\circle*{5}}}
	\put(30,0){\color{blue}{\circle*{5}}}
	\put(30,30){\color{blue}{\circle*{5}}}
	\put(30,40){\color{blue}{\circle*{5}}}
	\put(40,20){\color{blue}{\circle*{5}}}
	\put(40,10){\color{blue}{\circle*{5}}}
	\put(50,40){\color{blue}{\circle*{5}}}
	\put(50,10){\color{blue}{\circle*{5}}}
	\put(60,30){\color{blue}{\circle*{5}}}
	\put(60,0){\color{blue}{\circle*{5}}}
	\put(70,30){\color{blue}{\circle*{5}}}
	\put(70,20){\color{blue}{\circle*{5}}}
	\put(80,0){\color{blue}{\circle*{5}}}
	\put(80,10){\color{blue}{\circle*{5}}}
	\put(80,40){\color{blue}{\circle*{5}}}
	\put(90,20){\color{blue}{\circle*{5}}}	
	\put(100,0){\color{blue}{\circle*{5}}}	
	\put(100,30){\color{blue}{\circle*{5}}}	
	\put(100,40){\color{blue}{\circle*{5}}}	
	\put(110,30){\color{blue}{\circle*{5}}}
	\put(110,10){\color{blue}{\circle*{5}}}
	\put(120,40){\color{blue}{\circle*{5}}}
	\put(120,0){\color{blue}{\circle*{5}}}
	\put(120,10){\color{blue}{\circle*{5}}}
	\put(130,20){\color{blue}{\circle*{5}}}
	\put(140,30){\color{blue}{\circle*{5}}}
	\put(140,40){\color{blue}{\circle*{5}}}
	\put(140,0){\color{blue}{\circle*{5}}}
	\put(150,30){\color{blue}{\circle*{5}}}
	\put(150,10){\color{blue}{\circle*{5}}}	
	\put(160,0){\color{blue}{\circle*{5}}}
	\put(160,10){\color{blue}{\circle*{5}}}
	\put(160,40){\color{blue}{\circle*{5}}}
	\put(170,20){\color{blue}{\circle*{5}}}
	\put(180,0){\color{blue}{\circle*{5}}}
	\put(180,30){\color{blue}{\circle*{5}}}
	\put(180,40){\color{blue}{\circle*{5}}}
	\put(190,10){\color{blue}{\circle*{5}}}
	\put(190,30){\color{blue}{\circle*{5}}}
\end{picture}
\caption{ Completing the dominating set of $G_{5,20}$ } \label{again}
\end{figure}


Following this process will generate a dominating set for any $5 \times n$ grid.
The next result proves that this set is a minimal dominating set for $G_{m,n}$.  
\begin{theorem}
Let $n\ge 5$.  The $(2,2)$  broadcast domination number of the $5 \times n$ grid is 
\[ \gamma_{2,2} (G_{5,n} ) = 2 n + \left \lceil  \frac{n+2}{7} \right \rceil.\]
\end{theorem}

\begin{proof}
Note that for every $n \equiv 6 \mod 7$,  the set $D_{n}$ contains three more vertices than $D_{n-1}$.
Otherwise, for all other integers $n \geq 5$ the set $D_n$ contains just two more vertices than $D_{n-1}$.
A simple induction argument shows that the sum $2\cdot n + \left \lceil  \frac{n+2}{7} \right \rceil$ follows the same pattern, and hence they count the number elements in each set $D_n$. 

Next we show that the sets $D_n$ dominate $G_{5,n}$ and that no smaller set will dominate $G_{5,n}$.
First we justify our choice of the patterns 2-3-1-3 and 3-1-3-2 to dominate the ends of the grid $G_{5,n}$.  
Note that any dominating set must start and end with 5 vertices in the first two columns, so 2-3, 3-2, 4-1, 1-4, 5-0, or 0-5 are the only possible ends for a dominating set. 
However, 0-5 and 1-4 will not dominate the first two columns, and while the patterns 5-0 and 4-1 dominate the first two columns, they do not dominate as much of the third column
as 2-3.  Hence any minimal dominating set should begin and end with 2-3 and 3-2.
The pattern must continue 2-3-1 and 1-3-2 because 2-3-0 does not satisfy the properties of a dominating pattern. 
Furthermore we have determined, using the MPS algorithm, that any column with only one vertex must have 3 or more vertices in each adjacent column.  
Hence the most efficient pattern containing a column with just one vertex is the pattern 3-1-3 that is depicted in Figure \ref{3-1-3}. 
In fact the MPS algorithm affirms that, up to a vertical flip across the third row, this is the only
dominating subpattern of the form 3-1-3.

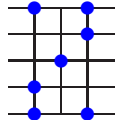
\begin{figure}[H]
\centering
\begin{picture}(40,40)(0,0)
	\multiput(0,0)(0,10){5}{\line(1,0){40}}
	\multiput(10,0)(10,0){3}{\line(0,1){40}}
	\put(10,0){\color{blue}{\circle*{5}}}
	\put(10,10){\color{blue}{\circle*{5}}}
	\put(10,40){\color{blue}{\circle*{5}}}
	\put(20,20){\color{blue}{\circle*{5}}}
	\put(30,0){\color{blue}{\circle*{5}}}
	\put(30,30){\color{blue}{\circle*{5}}}
	\put(30,40){\color{blue}{\circle*{5}}}
\end{picture}
\caption{The 3-1-3 Pattern}\label{3-1-3}
\end{figure}
Similarly, the program returns that the columns adjacent to the 3-1-3 pattern must each contain at least two vertices, so the pattern 2-3-1-3-2 is 
the smallest five-column subpattern that contains a column with just one vertex.

Next we justify our choice of the filling subpattern 2-2-2-2-3-1-3.  Note the pattern  2-2-2-2-2 uses one less vertex than 2-3-1-3-2. 
Thus it is preferable to use a pattern with as many columns containing 2 vertices as possible.
Calculations in the MPS program show that if we have four or fewer adjacent columns with just two vertices, they can be surrounded by 3-1-3 patterns.
This is the pattern used in our construction of $D_n$ in Figure \ref{2-patterns}.
Moreover, the MPS program shows that the pattern 3-1-3-2-2-2-2-2-3-1-3 with five adjacent columns with two vertices in each is not a dominating pattern.
It also shows that 2-1 is not a dominating pattern.  
Hence repeating the pattern 2-2-2-2-3-1-3 is the most efficient way to dominate the interior of the grid. 
\end{proof}

\subsection{(3,1) broadcast domination numbers of 3 by n grids}

The minimal $(3,1)$ broadcast dominating set $D_n$ follows a very simple pattern.  
As shown in Figure \ref{31patterns3n}, it contains every third vertex in the middle row of $G_{3,n}$ starting in the second column.
The cardinality of this set is described by $|D_n| = \left \lceil \frac{n}{3} \right \rceil$.
\begin{figure}[H]
\centering
\begin{picture}(405,20)(0,0)
	\linethickness{.1 mm}
	\multiput(10,0)(0,10){3}{\line(1,0){20}}
	\multiput(10,0)(10,0){3}{\line(0,1){20}}
	\put(20,10){\color{blue}{\circle*{5}}}

	\multiput(50,0)(0,10){3}{\line(1,0){30}}
	\multiput(50,0)(10,0){4}{\line(0,1){20}}
	
	\put(60,10){\color{blue}{\circle*{5}}}

	\put(80,10){\color{blue}{\circle*{5}}}
	\multiput(100,0)(0,10){3}{\line(1,0){40}}
	\multiput(100,0)(10,0){5}{\line(0,1){20}}
	\put(110,10){\color{blue}{\circle*{5}}}
	\put(130,10){\color{blue}{\circle*{5}}}
	\multiput(160,0)(0,10){3}{\line(1,0){50}}
	\multiput(160,0)(10,0){6}{\line(0,1){20}}
        \put(170,10){\color{blue}{\circle*{5}}}
	\put(200,10){\color{blue}{\circle*{5}}}
	
	\multiput(230,0)(0,10){3}{\line(1,0){60}}
	\multiput(230,0)(10,0){7}{\line(0,1){20}}
	\put(240,10){\color{blue}{\circle*{5}}}
	\put(270,10){\color{blue}{\circle*{5}}}
	\put(290,10){\color{blue}{\circle*{5}}}
	\multiput(310,0)(0,10){3}{\line(1,0){70}}
	\multiput(310,0)(10,0){8}{\line(0,1){20}}
	\put(320,10){\color{blue}{\circle*{5}}}
	\put(350,10){\color{blue}{\circle*{5}}}
	\put(380,10){\color{blue}{\circle*{5}}}

\end{picture}
\caption{(3,1) dominating sets $D_n$ for $G_{3,n}$} \label{31patterns3n}
\end{figure}

\begin{theorem} $Let n \geq 3$. 
The $(3,1)$  broadcast domination number of $G_{3,n}$ is 
\[
 \gamma_{3,1}(G_{3,n}) = \left \lceil \frac{n}{3} \right \rceil .
\]
\end{theorem}

\begin{proof}
One can easily verify that the pattern 0-1-0-0-1-0-0-1-0-$\cdots$-0-1-0  with all of the vertices in $D_n$ selected from the middle row is a dominating set.
 Since every third column contains one vertex in $D_n$, we see that $|D_n| = \left \lceil \frac{n}{3} \right \rceil$.   
 The MPS algorithm verifies that any subset of vertices satisfying the subpattern 1-0-0-0-1 results in one of the vertices in the middle column having reception strength 0. 
 Thus 1-0-0-0-1 is not a dominating subpattern, and  0-1-0-0-1-0-0-1-0-$\cdots$-0-1-0 is the most efficient pattern to dominate $G_{3,n}$. 
\end{proof}

\subsection{(3,1) broadcast domination numbers of 4 by n grids}

In this section, we find a closed formula for the (3,1) broadcast domination number of a $4 \times  n$ grid. Consider the following dominating sets for some small $n$.

\begin{figure}[H]
\centering
\begin{picture}(450,30)(0,0)
	\linethickness{.1 mm}
	\multiput(10,0)(0,10){4}{\line(1,0){30}}
	\multiput(10,0)(10,0){4}{\line(0,1){30}}
	\put(10,10){\color{blue}{\circle*{5}}}
 	\put(30,30){\color{blue}{\circle*{5}}}
 	\put(40,0){\color{blue}{\circle*{5}}}
	\multiput(60,0)(0,10){4}{\line(1,0){40}}
	\multiput(60,0)(10,0){5}{\line(0,1){30}}
 	\put(60,10){\color{blue}{\circle*{5}}}
        \put(90,30){\color{blue}{\circle*{5}}}
        \put(100,0){\color{blue}{\circle*{5}}}

	\multiput(120,0)(0,10){4}{\line(1,0){50}}
 	\multiput(120,0)(10,0){6}{\line(0,1){30}}
 	\put(120,10){\color{blue}{\circle*{5}}}
 	\put(160,0){\color{blue}{\circle*{5}}}
	\put(150,30){\color{blue}{\circle*{5}}}
 	\put(170,20){\color{blue}{\circle*{5}}}
 	
	\multiput(190,0)(0,10){4}{\line(1,0){60}}
	\multiput(190,0)(10,0){7}{\line(0,1){30}}
	\put(190,10){\color{blue}{\circle*{5}}}
	\put(220,30){\color{blue}{\circle*{5}}}
	\put(230,0){\color{blue}{\circle*{5}}}
	\put(250,20){\color{blue}{\circle*{5}}}
	
	\multiput(270,0)(0,10){4}{\line(1,0){70}}
	\multiput(270,0)(10,0){8}{\line(0,1){30}}
	\put(270,10){\color{blue}{\circle*{5}}}
	\put(300,30){\color{blue}{\circle*{5}}}
	\put(310,0){\color{blue}{\circle*{5}}}
	\put(340,20){\color{blue}{\circle*{5}}}
	\multiput(360,0)(0,10){4}{\line(1,0){80}}
	\multiput(360,0)(10,0){9}{\line(0,1){30}}
	\put(360,10){\color{blue}{\circle*{5}}}
	\put(390,30){\color{blue}{\circle*{5}}}
	\put(400,0){\color{blue}{\circle*{5}}}
	\put(430,20){\color{blue}{\circle*{5}}}
	\put(440,10){\color{blue}{\circle*{5}}}
\end{picture}
\end{figure}

\begin{figure}[H]
\centering
\begin{picture}(380,30)(0,0)
	\linethickness{.1 mm}
	\multiput(10,0)(0,10){4}{\line(1,0){90}}
	\multiput(10,0)(10,0){10}{\line(0,1){30}}
	\put(10,10){\color{blue}{\circle*{5}}}
	\put(40,30){\color{blue}{\circle*{5}}}
	\put(50,0){\color{blue}{\circle*{5}}}
	\put(80,20){\color{blue}{\circle*{5}}}
	\put(100,10){\color{blue}{\circle*{5}}}
	\multiput(120,0)(0,10){4}{\line(1,0){100}}
	\multiput(120,0)(10,0){11}{\line(0,1){30}}
	\put(120,10){\color{blue}{\circle*{5}}}
	\put(150,30){\color{blue}{\circle*{5}}}
	\put(160,0){\color{blue}{\circle*{5}}}
	\put(190,20){\color{blue}{\circle*{5}}}
	\put(210,0){\color{blue}{\circle*{5}}}
        \put(220,30){\color{blue}{\circle*{5}}}

	\multiput(240,0)(0,10){4}{\line(1,0){110}}
	\multiput(240,0)(10,0){12}{\line(0,1){30}}
	\put(240,10){\color{blue}{\circle*{5}}}
	\put(270,30){\color{blue}{\circle*{5}}}
	\put(280,0){\color{blue}{\circle*{5}}}
	\put(310,20){\color{blue}{\circle*{5}}}
	\put(340,0){\color{blue}{\circle*{5}}}
	\put(350,30){\color{blue}{\circle*{5}}}
	

\end{picture}
\end{figure}

\begin{figure}[H]
\centering
\begin{picture}(380,30)(40,0)
	\multiput(10,0)(0,10){4}{\line(1,0){120}}
	\multiput(10,0)(10,0){13}{\line(0,1){30}}
	\put(10,10){\color{blue}{\circle*{5}}}
	\put(40,30){\color{blue}{\circle*{5}}}
	\put(50,0){\color{blue}{\circle*{5}}}
	\put(80,20){\color{blue}{\circle*{5}}}
	\put(110,0){\color{blue}{\circle*{5}}}
	\put(120,30){\color{blue}{\circle*{5}}}
        \put(130,10){\color{blue}{\circle*{5}}}
	\multiput(150,0)(0,10){4}{\line(1,0){130}}
	\multiput(150,0)(10,0){14}{\line(0,1){30}}
	\put(150,10){\color{blue}{\circle*{5}}}
	\put(180,30){\color{blue}{\circle*{5}}}
	\put(190,0){\color{blue}{\circle*{5}}}
	\put(220,20){\color{blue}{\circle*{5}}}
	\put(250,0){\color{blue}{\circle*{5}}}
	\put(260,30){\color{blue}{\circle*{5}}}
        \put(280,10){\color{blue}{\circle*{5}}}
	
	\multiput(300,0)(0,10){4}{\line(1,0){140}}
	\multiput(300,0)(10,0){15}{\line(0,1){30}}
	\put(300,10){\color{blue}{\circle*{5}}}
	\put(330,30){\color{blue}{\circle*{5}}}
	\put(340,0){\color{blue}{\circle*{5}}}
	\put(370,20){\color{blue}{\circle*{5}}}
	\put(400,0){\color{blue}{\circle*{5}}}
	\put(410,30){\color{blue}{\circle*{5}}}
        \put(440,10){\color{blue}{\circle*{5}}}
\end{picture}
\caption{$4\times 4$ to $4\times15$ Domination Pattern for (3,1) domination} \label{fig:3,1-4xn}
\end{figure}

The construction of these sets follows a similar process to the construction of the (2,2) domination sets above.
Progressing from left to right, we will dominate the graph using the pattern shown in Figure \ref{1001100}.

\begin{figure}[H]
\centering
\begin{picture}(80,30)(40,0)
	\multiput(10,0)(0,10){4}{\line(1,0){80}}
	\multiput(20,0)(10,0){7}{\line(0,1){30}}
	\put(20,10){\color{blue}{\circle*{5}}}
	\put(50,30){\color{blue}{\circle*{5}}}
	\put(60,0){\color{blue}{\circle*{5}}}

	\multiput(110,0)(0,10){4}{\line(1,0){80}}
	\multiput(120,0)(10,0){7}{\line(0,1){30}}
	\put(120,20){\color{blue}{\circle*{5}}}
	\put(150,0){\color{blue}{\circle*{5}}}
	\put(160,30){\color{blue}{\circle*{5}}}
 \end{picture}
\caption{The 1-0-0-1-1-0-0 Pattern for $4\times n$ domination.} \label{1001100}
\end{figure}

The pattern can be placed in the grid with one of the two orientations shown.
In order to dominate the grid, we alternate its placement between the two orientations.  
This process is repeated until there are seven or less columns remaining to be dominated.
Then we use one of the seven patterns presented in Figure \ref{3,1-4xn finish} to fill in the remaining columns. Note that these patterns may need to be flipped vertically to fit with the last installment of the 1-0-0-1-1-0-0 pattern.

\begin{figure}[H]
\centering
\begin{picture}(320,30)(40,0)
	\multiput(10,0)(0,10){4}{\line(1,0){10}}
	\multiput(20,0)(10,0){1}{\line(0,1){30}}
	\put(20,20){\color{blue}{\circle*{5}}}

	\multiput(30,0)(0,10){4}{\line(1,0){20}}
	\multiput(40,0)(10,0){2}{\line(0,1){30}}
	\put(40,20){\color{blue}{\circle*{5}}}
	\put(50,10){\color{blue}{\circle*{5}}}

	\multiput(60,0)(0,10){4}{\line(1,0){30}}
	\multiput(70,0)(10,0){3}{\line(0,1){30}}
	\put(70,20){\color{blue}{\circle*{5}}}
	\put(90,10){\color{blue}{\circle*{5}}}

	\multiput(100,0)(0,10){4}{\line(1,0){40}}
	\multiput(110,0)(10,0){4}{\line(0,1){30}}
	\put(110,20){\color{blue}{\circle*{5}}}
	\put(130,0){\color{blue}{\circle*{5}}}
	\put(140,30){\color{blue}{\circle*{5}}}

	\multiput(150,0)(0,10){4}{\line(1,0){50}}
	\multiput(160,0)(10,0){5}{\line(0,1){30}}
	\put(160,20){\color{blue}{\circle*{5}}}
	\put(190,0){\color{blue}{\circle*{5}}}
	\put(200,30){\color{blue}{\circle*{5}}}

	\multiput(210,0)(0,10){4}{\line(1,0){60}}
	\multiput(220,0)(10,0){6}{\line(0,1){30}}
	\put(220,20){\color{blue}{\circle*{5}}}
	\put(250,0){\color{blue}{\circle*{5}}}
	\put(260,30){\color{blue}{\circle*{5}}}
	\put(270,10){\color{blue}{\circle*{5}}}

	\multiput(280,0)(0,10){4}{\line(1,0){70}}
	\multiput(290,0)(10,0){7}{\line(0,1){30}}
	\put(290,20){\color{blue}{\circle*{5}}}
	\put(320,0){\color{blue}{\circle*{5}}}
	\put(330,30){\color{blue}{\circle*{5}}}
	\put(350,10){\color{blue}{\circle*{5}}}
 \end{picture}
\caption{The ending patterns for 4 $\times$ n domination.} \label{3,1-4xn finish}
\end{figure}
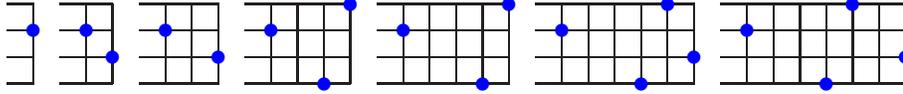

\begin{theorem}
Let $n \geq 4$.  The $(3,1)$  broadcast domination number of $G_{4,n}$ is 
\[
 \gamma_{3,1}(G_{4, n}) = \left\lfloor \frac{n+1}{7} \right\rfloor + \left\lfloor \frac{n+3}{7} \right\rfloor + \left\lfloor \frac{n+5}{7} \right\rfloor +1 .
\]
\end{theorem}
\begin{proof}
Let $D_n$ denote the dominating set constructed above for $G_{4, n}$. 
In particular, the MPS program verifies that no six-column subpattern with fewer vertices than the subpattern 1-0-1-1-0-0  can be extended to dominate $G_{4,n}$. 
All that remains to be checked is that the right end of the graph is dominated in a minimal fashion. One can verify by hand or by computation
 that the patterns shown in Figure \ref{3,1-4xn finish} are the best ways to finish the domination set.

For $n \geq 4$ the cardinality of the set $D_n$ satisfies:  
\[ |D_n| = \begin{cases} |D_{n-1}|+1 & \text{ when } n \equiv 2,4, 6\mod 7   \\ |D_{n-1}| & \text{ otherwise } \end{cases}. \]
A simple induction argument shows that $|D_n| = \left\lfloor \frac{n+1}{7} \right\rfloor + \left\lfloor \frac{n+3}{7} \right\rfloor + \left\lfloor \frac{n+5}{7} \right\rfloor +1 $ for all $n \geq 4$.  
\end{proof}

\subsection{(3,2) broadcast domination numbers of 3 by n grids}

In this section, we will construct a dominating set for $(3,2)$  broadcast domination of a $3 \times n$ grid. Then we will show that this construction finds a minimal dominating set.

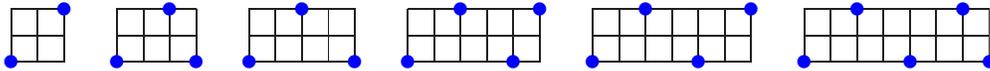
\begin{figure}[H]
\centering
\begin{picture}(405,20)(0,0)
	\linethickness{.1 mm}
	\multiput(10,0)(0,10){3}{\line(1,0){20}}
	\multiput(10,0)(10,0){3}{\line(0,1){20}}
	\put(10,0){\color{blue}{\circle*{5}}}
	\put(30,20){\color{blue}{\circle*{5}}}

	\multiput(50,0)(0,10){3}{\line(1,0){30}}
	\multiput(50,0)(10,0){4}{\line(0,1){20}}
	
	\put(50,0){\color{blue}{\circle*{5}}}
	\put(70,20){\color{blue}{\circle*{5}}}
	\put(80,0){\color{blue}{\circle*{5}}}
	\multiput(100,0)(0,10){3}{\line(1,0){40}}
	\multiput(100,0)(10,0){5}{\line(0,1){20}}
	\put(100,0){\color{blue}{\circle*{5}}}
	\put(120,20){\color{blue}{\circle*{5}}}
	\put(140,0){\color{blue}{\circle*{5}}}
	\multiput(160,0)(0,10){3}{\line(1,0){50}}
	\multiput(160,0)(10,0){6}{\line(0,1){20}}
    \put(160,0){\color{blue}{\circle*{5}}}
	\put(180,20){\color{blue}{\circle*{5}}}
	\put(200,00){\color{blue}{\circle*{5}}}
	\put(210,20){\color{blue}{\circle*{5}}}
	\multiput(230,0)(0,10){3}{\line(1,0){60}}
	\multiput(230,0)(10,0){7}{\line(0,1){20}}
	\put(230,0){\color{blue}{\circle*{5}}}
	\put(250,20){\color{blue}{\circle*{5}}}
	\put(270,0){\color{blue}{\circle*{5}}}
	\put(290,20){\color{blue}{\circle*{5}}}
	\multiput(310,0)(0,10){3}{\line(1,0){70}}
	\multiput(310,0)(10,0){8}{\line(0,1){20}}
	\put(310,0){\color{blue}{\circle*{5}}}
	\put(330,20){\color{blue}{\circle*{5}}}
	\put(350,0){\color{blue}{\circle*{5}}}
	\put(370,20){\color{blue}{\circle*{5}}}
	\put(380,0){\color{blue}{\circle*{5}}}
	
\end{picture}
\end{figure}

The construction of the above dominating sets follows a similar methodology to the previous cases. 
We start with choosing the bottom vertex in the first column. Then we choose a vertex in every other column.  Finally, we always choose a vertex in the last column.

\begin{theorem} Let $n \geq 3$.  
The $(3,2)$  broadcast domination number of $G_{3,n}$ is 
\[
 \gamma_{3,2} (G_{3,n}) = \left\lceil \frac{n+1}{2} \right\rceil .
\]
\end{theorem}

\begin{proof}
We now prove that $D_n$ is a minimal dominating set for $G_{4,n}$.
The dominating set $D_n$ uses one vertex from every other column.  It is easy to verify (using the MPS algorithm or by hand) that no set of vertices with two adjacent empty columns can dominate $G_{3,n}$.  Hence
the pattern -1-0-0-1- is never a dominating subpattern of $G_{3,n}$, and the dominating set $D_n$ is minimal.

In constructing $D_n$ we choose a vertex in every other column, and we add a vertex in the last column if there is not already one there.  One can easily verify by inspection that 
$|D_n| = \left\lceil \frac{n+1}{2} \right\rceil$ for $n \leq 5$, and a  simple induction argument shows that 
$|D_n| = \left\lceil \frac{n+1}{2} \right\rceil$ for all $n >5$. 
\end{proof}

\subsection{(3,2) broadcast domination numbers of 4 by n grids}
In this section we construct a $(3,2)$ broadcast domination set for arbitrary $4 \times n$ grids, and then we find a closed formula for the $(3,2)$ broadcast domination number for all $4 \times n$ grids, where $n \geq 4$.
\begin{figure}[H]
\centering
\begin{picture}(450,30)(0,0)
	\linethickness{.1 mm}
	\multiput(10,0)(0,10){4}{\line(1,0){30}}
	\multiput(10,0)(10,0){4}{\line(0,1){30}}
	\put(10,0){\color{blue}{\circle*{5}}}
 	\put(20,30){\color{blue}{\circle*{5}}}
 	\put(40,10){\color{blue}{\circle*{5}}}
	\multiput(60,0)(0,10){4}{\line(1,0){40}}
	\multiput(60,0)(10,0){5}{\line(0,1){30}}
 	\put(60,00){\color{blue}{\circle*{5}}}
    \put(70,30){\color{blue}{\circle*{5}}}
    \put(90,10){\color{blue}{\circle*{5}}}
    \put(100,20){\color{blue}{\circle*{5}}}
	\multiput(120,0)(0,10){4}{\line(1,0){50}}
 	\multiput(120,0)(10,0){6}{\line(0,1){30}}
 	\put(120,0){\color{blue}{\circle*{5}}}
 	\put(130,30){\color{blue}{\circle*{5}}}
	\put(150,10){\color{blue}{\circle*{5}}}
 	\put(160,30){\color{blue}{\circle*{5}}}
 	\put(170,0){\color{blue}{\circle*{5}}}
	\multiput(190,0)(0,10){4}{\line(1,0){60}}
	\multiput(190,0)(10,0){7}{\line(0,1){30}}
	\put(190,0){\color{blue}{\circle*{5}}}
	\put(200,30){\color{blue}{\circle*{5}}}
	\put(220,10){\color{blue}{\circle*{5}}}
	\put(240,30){\color{blue}{\circle*{5}}}
	\put(250,0){\color{blue}{\circle*{5}}}
	\multiput(270,0)(0,10){4}{\line(1,0){70}}
	\multiput(270,0)(10,0){8}{\line(0,1){30}}
	\put(270,0){\color{blue}{\circle*{5}}}
	\put(280,30){\color{blue}{\circle*{5}}}
	\put(300,10){\color{blue}{\circle*{5}}}
	\put(320,30){\color{blue}{\circle*{5}}}
	\put(330,0){\color{blue}{\circle*{5}}}
	\put(340,20){\color{blue}{\circle*{5}}}
	\multiput(360,0)(0,10){4}{\line(1,0){80}}
	\multiput(360,0)(10,0){9}{\line(0,1){30}}
	\put(360,0){\color{blue}{\circle*{5}}}
	\put(370,30){\color{blue}{\circle*{5}}}
	\put(390,10){\color{blue}{\circle*{5}}}
	\put(410,30){\color{blue}{\circle*{5}}}
	\put(420,0){\color{blue}{\circle*{5}}}
	\put(440,20){\color{blue}{\circle*{5}}}
\end{picture}
\end{figure}

\begin{figure}[H]
\centering
\begin{picture}(380,30)(0,0)
	\linethickness{.1 mm}
	\multiput(10,0)(0,10){4}{\line(1,0){90}}
	\multiput(10,0)(10,0){10}{\line(0,1){30}}
	\put(10,0){\color{blue}{\circle*{5}}}
	\put(20,30){\color{blue}{\circle*{5}}}
	\put(40,10){\color{blue}{\circle*{5}}}
	\put(60,30){\color{blue}{\circle*{5}}}
	\put(70,0){\color{blue}{\circle*{5}}}
	\put(90,20){\color{blue}{\circle*{5}}}
	\put(100,10){\color{blue}{\circle*{5}}}
	\multiput(120,0)(0,10){4}{\line(1,0){100}}
	\multiput(120,0)(10,0){11}{\line(0,1){30}}
	\put(120,0){\color{blue}{\circle*{5}}}
	\put(130,30){\color{blue}{\circle*{5}}}
	\put(150,10){\color{blue}{\circle*{5}}}
	\put(170,30){\color{blue}{\circle*{5}}}
	\put(180,0){\color{blue}{\circle*{5}}}
    \put(200,20){\color{blue}{\circle*{5}}}
    \put(210,0){\color{blue}{\circle*{5}}}
    \put(220,30){\color{blue}{\circle*{5}}}
	\multiput(240,0)(0,10){4}{\line(1,0){110}}
	\multiput(240,0)(10,0){12}{\line(0,1){30}}
	\put(240,0){\color{blue}{\circle*{5}}}
	\put(250,30){\color{blue}{\circle*{5}}}
	\put(270,10){\color{blue}{\circle*{5}}}
	\put(290,30){\color{blue}{\circle*{5}}}
	\put(300,0){\color{blue}{\circle*{5}}}
	\put(320,20){\color{blue}{\circle*{5}}}
	\put(340,0){\color{blue}{\circle*{5}}}
	\put(350,30){\color{blue}{\circle*{5}}}
	
\end{picture}
\end{figure}

\begin{figure}[H]
\centering
\begin{picture}(380,30)(40,0)
	\multiput(10,0)(0,10){4}{\line(1,0){120}}
	\multiput(10,0)(10,0){13}{\line(0,1){30}}
	\put(10,0){\color{blue}{\circle*{5}}}
	\put(20,30){\color{blue}{\circle*{5}}}
	\put(40,10){\color{blue}{\circle*{5}}}
	\put(60,30){\color{blue}{\circle*{5}}}
	\put(70,0){\color{blue}{\circle*{5}}}
	\put(90,20){\color{blue}{\circle*{5}}}
    \put(110,0){\color{blue}{\circle*{5}}}
    \put(120,30){\color{blue}{\circle*{5}}}
    \put(130,10){\color{blue}{\circle*{5}}}

	\multiput(150,0)(0,10){4}{\line(1,0){130}}
	\multiput(150,0)(10,0){14}{\line(0,1){30}}
	\put(150,0){\color{blue}{\circle*{5}}}
	\put(160,30){\color{blue}{\circle*{5}}}
	\put(180,10){\color{blue}{\circle*{5}}}
	\put(200,30){\color{blue}{\circle*{5}}}
	\put(210,0){\color{blue}{\circle*{5}}}
	\put(230,20){\color{blue}{\circle*{5}}}
    \put(250,0){\color{blue}{\circle*{5}}}
    \put(260,30){\color{blue}{\circle*{5}}}
    \put(280,10){\color{blue}{\circle*{5}}}
	
	\multiput(300,0)(0,10){4}{\line(1,0){140}}
	\multiput(300,0)(10,0){15}{\line(0,1){30}}
	\put(300,0){\color{blue}{\circle*{5}}}
	\put(310,30){\color{blue}{\circle*{5}}}
	\put(330,10){\color{blue}{\circle*{5}}}
	\put(350,30){\color{blue}{\circle*{5}}}
	\put(360,0){\color{blue}{\circle*{5}}}
	\put(380,20){\color{blue}{\circle*{5}}}
    \put(400,0){\color{blue}{\circle*{5}}}
    \put(410,30){\color{blue}{\circle*{5}}}
    \put(430,10){\color{blue}{\circle*{5}}}
    \put(440,20){\color{blue}{\circle*{5}}}

\end{picture}
\caption{4x4 to 4x15 Domination Pattern for (3,2) domination} \label{fig:3,2-4xn}
\end{figure}
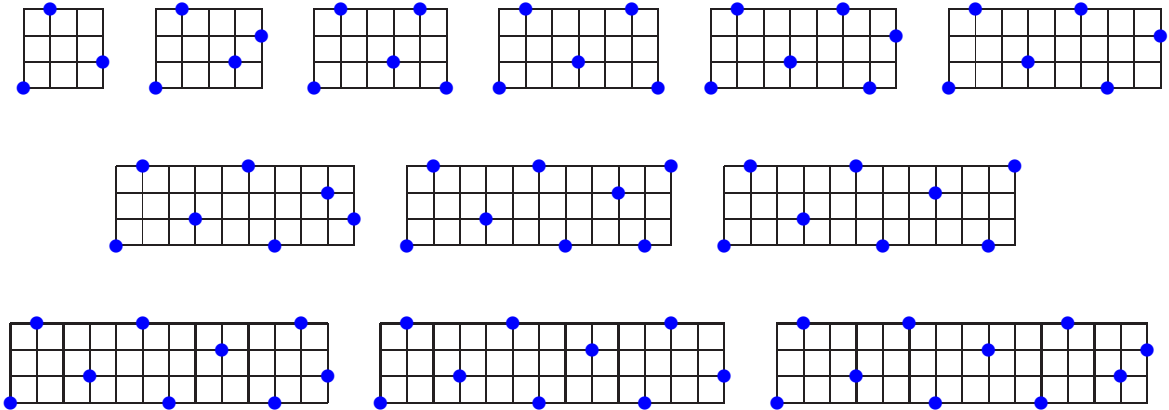

We start the construction by using the dominating pattern 1-1-0-1-0 shown below. 
We alternate between the two configurations shown below in Figure \ref{32pattern}.
We repeat this pattern until there are between 2 and 6 columns remaining.

\begin{figure}[H]
\centering
\begin{picture}(100,30)(50,0)
	\multiput(10,0)(0,10){4}{\line(1,0){60}}
	\multiput(20,0)(10,0){5}{\line(0,1){30}}
	\put(20,0){\color{blue}{\circle*{5}}}
	\put(30,30){\color{blue}{\circle*{5}}}
	\put(50,10){\color{blue}{\circle*{5}}}

	\multiput(110,0)(0,10){4}{\line(1,0){60}}
	\multiput(120,0)(10,0){5}{\line(0,1){30}}
	\put(120,30){\color{blue}{\circle*{5}}}
	\put(130,0){\color{blue}{\circle*{5}}}
	\put(150,20){\color{blue}{\circle*{5}}}
\end{picture}
\caption{Repeated $(3,2)$  broadcast domination pattern for $4 \times n$ grid.} \label{32pattern}
\end{figure}
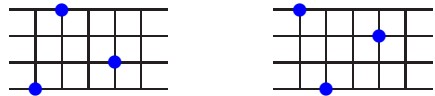

All that remains to be done, is to finish the dominating set off with one of the 5 patterns listed below.

\begin{figure}[H]
\centering
\begin{picture}(220,30)(40,0)
	\multiput(10,0)(0,10){4}{\line(1,0){20}}
	\multiput(20,0)(10,0){2}{\line(0,1){30}}
	\put(20,30){\color{blue}{\circle*{5}}}
	\put(30,0){\color{blue}{\circle*{5}}}

	\multiput(40,0)(0,10){4}{\line(1,0){30}}
	\multiput(50,0)(10,0){3}{\line(0,1){30}}
	\put(50,30){\color{blue}{\circle*{5}}}
	\put(60,0){\color{blue}{\circle*{5}}}
	\put(70,20){\color{blue}{\circle*{5}}}

	\multiput(80,0)(0,10){4}{\line(1,0){40}}
	\multiput(90,0)(10,0){4}{\line(0,1){30}}
	\put(90,30){\color{blue}{\circle*{5}}}
	\put(100,0){\color{blue}{\circle*{5}}}
	\put(120,20){\color{blue}{\circle*{5}}}

	\multiput(130,0)(0,10){4}{\line(1,0){50}}
	\multiput(140,0)(10,0){5}{\line(0,1){30}}
	\put(140,30){\color{blue}{\circle*{5}}}
	\put(150,0){\color{blue}{\circle*{5}}}
	\put(170,20){\color{blue}{\circle*{5}}}
	\put(180,10){\color{blue}{\circle*{5}}}

	\multiput(190,0)(0,10){4}{\line(1,0){60}}
	\multiput(200,0)(10,0){6}{\line(0,1){30}}
	\put(200,30){\color{blue}{\circle*{5}}}
	\put(210,0){\color{blue}{\circle*{5}}}
	\put(230,20){\color{blue}{\circle*{5}}}
	\put(240,0){\color{blue}{\circle*{5}}}
	\put(250,30){\color{blue}{\circle*{5}}}
	
 \end{picture}
\caption{The ending patterns for $(3,2)$ broadcast domination of $4 \times n$ grids.} \label{3,2-4xn finish}
\end{figure}
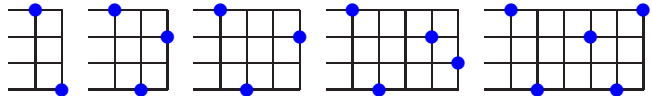

\begin{theorem}
Let $n \geq 4$.  The $(3,2)$ broadcast domination number of $G_{4,n}$ is 
\[
 \gamma_{3,2}(G_{4,n}) = \left\lceil \frac{n+4}{5}\right\rceil + \left\lceil \frac{n+2}{5}\right\rceil + \left\lceil \frac{n}{5}\right\rceil +1 
\]
\end{theorem}

\begin{proof}
We now show that the $D_n$ are minimal dominating sets.
Consider that the pattern used in the above dominating sets is 1-1-0-1-0 repeating until the graph ends. 
The MPS algorithm confirms that all five-column patterns containing fewer vertices than 1-1-0-1-0 are not dominating subpatterns. \footnote{The calculation is available at \\ {\tt https://cloud.sagemath.com/projects/26b983ae-a894-47c0-bd54-c73b071e555c/files/(3,2)BDPC.sagews }}  Similarly, the program also shows that the patterns we 
use to dominate the last 2, 3, 4, 5, or 6 columns are the minimal dominating patterns.   
Therefore, we conclude that the dominating sets presented above are the smallest (3,2) broadcast dominating sets for $G_{4,n}$.

We note that by its construction that $D_n$ satisfies
\[ |D_n| = \begin{cases} |D_{n-1}|+1 & \text{ when } n \equiv 0,1,3 \mod 5 \\  |D_{n-1}| & \text{otherwise}  \end{cases}  \]
Letting $F(n) =\left\lceil \frac{n+4}{5}\right\rceil + \left\lceil \frac{n+2}{5}\right\rceil + \left\lceil \frac{n}{5}\right\rceil +1,$ a simple induction argument shows that $F(n)$ follows the same pattern:
\[ F(n) = \begin{cases} F(n-1)+1 & \text{ when } n \equiv 0,1,3 \mod 5 \\  F(n-1) & \text{otherwise}  \end{cases}  .\] 

Since $D_n$ has the smallest cardinality of any $(3,2)$ dominating set for $G_{4,n}$, we conclude that $ \gamma_{3,2}(G_{4,n}) =  |D_n| =F(n) =
\left\lceil \frac{n+4}{5}\right\rceil + \left\lceil \frac{n+2}{5}\right\rceil + \left\lceil \frac{n}{5}\right\rceil +1$ for $n \geq 4$.
\end{proof}

\section{Upper Bounds on Broadcast Domination Numbers} \label{section:theorems}

In this section we construct efficient dominating sets for the $(2,2)$, $(3,1)$, $(3,2)$ and $(3,3)$ broadcast domination of $m \times n$ grids $G_{m,n}$. 
We conjecture these sets are in fact minimum dominating sets when $m$ and $n$ are sufficiently large. 
Then we prove formulas counting the cardinality of each dominating set, giving upper bounds for the broadcast domination numbers of $G_{m,n}$.  

To construct these broadcast dominating sets, we start by identifying a family of optimal dominating sets for $\ZZ \times \ZZ$ under $(2,2)$, $(3,1)$, $(3,2)$, and $(3,3)$ broadcast domination.
The intersection of these optimal sets with the appropriate neighborhood of $G_{m,n}$ will dominate $G_{m,n}$.
We then show how to condense the resulting dominating set to construct an efficient dominating set that is entirely contained in $G_{m,n}$.
We adapt several of the techniques used by Chang in finding upper bounds for the regular domination number of $m \times n$ grids \cite{Cha92}, which in our notation is the $(2,1)$ broadcast domination number.
However, the implementation of these techniques in more general $(\ww,\tw)$ broadcast domination is remarkably more nuanced than in regular domination.

\subsection{(2,2) broadcast domination}
In this subsection we describe (2,2) broadcast dominating sets of $G_{m,n}$ with the smallest cardinality. 
When $m$ and $n$ are sufficiently large, 
we conjecture that this set is an optimal dominating set of $G_{m,n}$ in the sense that it contains the minimum number of vertices of any dominating set
of $G_{m,n}$.  We conclude this section by counting the number of elements in this efficient dominating set, thus giving an upper bound on the broadcast domination number.

We start by describing a family of dominating sets for $\ZZ \times \ZZ$ under $(2,2)$ broadcast domination.
Define a map $\phi: \ZZ \times \ZZ \rightarrow \ZZ_{3}$ by $(x,y) \mapsto x+2y$. 
Let $P(i) = \phi^{-1}(i)$ denote the inverse image  for $i \in \ZZ_3$.  
The sets $P(0)$, $P(1)$, and $P(2)$ are shown in Figure \ref{P(0)}.

\begin{figure}[H]
\begin{picture}(410,110)(-40,0)
	\linethickness{.1 mm}
	\multiput(10,10)(0,10){9}{\line(1,0){100}}
	\multiput(20,0)(10,0){9}{\line(0,1){100}}
	\put(60,50){\color{blue}{\circle*{5}}}
	\put(70,60){\color{blue}{\circle*{5}}}
	\put(80,70){\color{blue}{\circle*{5}}}
	\put(90,80){\color{blue}{\circle*{5}}}
	\put(100,90){\color{blue}{\circle*{5}}}
	\put(50,40){\color{blue}{\circle*{5}}}  
	\put(40,30){\color{blue}{\circle*{5}}} 
	\put(30,20){\color{blue}{\circle*{5}}} 
	\put(20,10){\color{blue}{\circle*{5}}} 
	\put(30,50){\color{blue}{\circle*{5}}}
	\put(40,60){\color{blue}{\circle*{5}}}
	\put(50,70){\color{blue}{\circle*{5}}}
	\put(60,80){\color{blue}{\circle*{5}}}
	\put(70,90){\color{blue}{\circle*{5}}}
	\put(20,40){\color{blue}{\circle*{5}}}   
	\put(20,70){\color{blue}{\circle*{5}}}
	\put(30,80){\color{blue}{\circle*{5}}}
	\put(40,90){\color{blue}{\circle*{5}}}
	\put(90,50){\color{blue}{\circle*{5}}}
	\put(100,60){\color{blue}{\circle*{5}}}
	\put(80,40){\color{blue}{\circle*{5}}}  
	\put(70,30){\color{blue}{\circle*{5}}} 
	\put(60,20){\color{blue}{\circle*{5}}} 
	\put(50,10){\color{blue}{\circle*{5}}} 
	\put(100,30){\color{blue}{\circle*{5}}} 
	\put(90,20){\color{blue}{\circle*{5}}} 
	\put(80,10){\color{blue}{\circle*{5}}} 
  
	\linethickness{0.4mm}
	\put(60,0){\line(0,1){100}}
	\put(10,50){\line(1,0){100}}

	\linethickness{.1 mm}
	\multiput(130,10)(0,10){9}{\line(1,0){100}}
	\multiput(140,0)(10,0){9}{\line(0,1){100}}
	\put(190,50){\color{blue}{\circle*{5}}}
	\put(200,60){\color{blue}{\circle*{5}}}
	\put(210,70){\color{blue}{\circle*{5}}}
	\put(220,80){\color{blue}{\circle*{5}}}
	\put(140,90){\color{blue}{\circle*{5}}}
	\put(180,40){\color{blue}{\circle*{5}}}  
	\put(170,30){\color{blue}{\circle*{5}}} 
	\put(160,20){\color{blue}{\circle*{5}}} 
	\put(150,10){\color{blue}{\circle*{5}}} 
	\put(160,50){\color{blue}{\circle*{5}}}
	\put(170,60){\color{blue}{\circle*{5}}}
	\put(180,70){\color{blue}{\circle*{5}}}
	\put(190,80){\color{blue}{\circle*{5}}}
	\put(200,90){\color{blue}{\circle*{5}}}
	\put(150,40){\color{blue}{\circle*{5}}}   
	\put(150,70){\color{blue}{\circle*{5}}}
	\put(160,80){\color{blue}{\circle*{5}}}
	\put(170,90){\color{blue}{\circle*{5}}}
	\put(220,50){\color{blue}{\circle*{5}}}
	\put(140,60){\color{blue}{\circle*{5}}}
	\put(210,40){\color{blue}{\circle*{5}}}  
	\put(200,30){\color{blue}{\circle*{5}}} 
	\put(190,20){\color{blue}{\circle*{5}}} 
	\put(180,10){\color{blue}{\circle*{5}}} 
	\put(140,30){\color{blue}{\circle*{5}}} 
	\put(220,20){\color{blue}{\circle*{5}}} 
	\put(210,10){\color{blue}{\circle*{5}}} 
  
	\linethickness{0.4mm}
	\put(180,0){\line(0,1){100}}
	\put(130,50){\line(1,0){100}}

	\linethickness{.1 mm}
	\multiput(250,10)(0,10){9}{\line(1,0){100}}
	\multiput(260,0)(10,0){9}{\line(0,1){100}}
	\put(320,50){\color{blue}{\circle*{5}}}
	\put(330,60){\color{blue}{\circle*{5}}}
	\put(340,70){\color{blue}{\circle*{5}}}
	\put(260,80){\color{blue}{\circle*{5}}}
	\put(270,90){\color{blue}{\circle*{5}}}
	\put(310,40){\color{blue}{\circle*{5}}}  
	\put(300,30){\color{blue}{\circle*{5}}} 
	\put(290,20){\color{blue}{\circle*{5}}} 
	\put(280,10){\color{blue}{\circle*{5}}} 
	\put(290,50){\color{blue}{\circle*{5}}}
	\put(300,60){\color{blue}{\circle*{5}}}
	\put(310,70){\color{blue}{\circle*{5}}}
	\put(320,80){\color{blue}{\circle*{5}}}
	\put(330,90){\color{blue}{\circle*{5}}}
	\put(280,40){\color{blue}{\circle*{5}}}   
	\put(280,70){\color{blue}{\circle*{5}}}
	\put(290,80){\color{blue}{\circle*{5}}}
	\put(300,90){\color{blue}{\circle*{5}}}
	\put(260,50){\color{blue}{\circle*{5}}}
	\put(270,60){\color{blue}{\circle*{5}}}
	\put(340,40){\color{blue}{\circle*{5}}}  
	\put(330,30){\color{blue}{\circle*{5}}} 
	\put(320,20){\color{blue}{\circle*{5}}} 
	\put(310,10){\color{blue}{\circle*{5}}} 
	\put(270,30){\color{blue}{\circle*{5}}} 
	\put(260,20){\color{blue}{\circle*{5}}} 
	\put(340,10){\color{blue}{\circle*{5}}} 
  
	\linethickness{0.4mm}
	\put(300,0){\line(0,1){100}}
	\put(250,50){\line(1,0){100}}
\end{picture}
\caption{ Three dominating sets $P(0)$, $P(1)$, and $P(2)$ } \label{P(0)}
\end{figure}
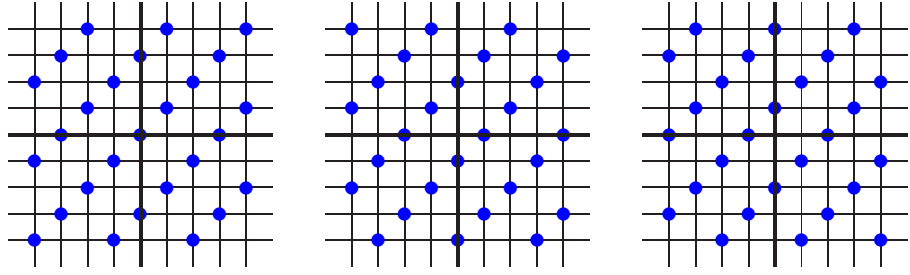

Since the elements of $P(i)$ appear on every third diagonal of $\ZZ \times \ZZ$, and every diagonal dominates itself and the two diagonals closest to it with a reception strength of $2$,
 it follows that each set $P(i)$ dominates $\ZZ \times \ZZ$ optimally. 

Embed $G_{m,n}$ into $\ZZ \times \ZZ$ as the following set: \[ G_{m,n} = \{ (a,b) \in \ZZ \times \ZZ  \mid 1 \leq a \leq n  \text{ and } 1 \leq b \leq m \}, \]
and let $Y_{m,n} \cong G_{m+2,n+2} $ denote the neighborhood of $G_{m,n}$: \[Y_{m,n} = \{ (a,b) \in \ZZ \times \ZZ \mid 0 \leq a \leq n+1 \text{ and }  0 \leq b \leq m+1 \}. \]
For each $i \in \ZZ_3$ the set $P(i) \cap Y_{m,n}$ completely dominates $G_{m,n}$ under $(2,2)$ broadcast domination.
Figure \ref{GYexample} shows $G_{7,8} \subset Y_{7,8}$ and the intersection $P(0) \cap Y_{7,8}$. One can easily verify by inspection that $P(0) \cap Y_{7,8}$ dominates $G_{7,8}$

\begin{figure}[H]
\begin{picture}(410,110)(-100,0)
	\linethickness{.1 mm}
	\multiput(10,10)(0,10){10}{\line(1,0){110}}
	\multiput(20,0)(10,0){10}{\line(0,1){110}}
    
	\linethickness{0.4mm}
	\put(20,0){\line(0,1){100}}
	\put(10,10){\line(1,0){100}}
        \put(30,20){\color{red}{\line(1,0){70}}}
        \put(100,20){\color{red}{\line(0,1){60}}}
        \put(30,20){\color{red}{\line(0,1){60}}}
        \put(30,80){\color{red}{\line(1,0){70}}}

        \put(20,10){\color{green}{\line(1,0){90}}}
        \put(110,10){\color{green}{\line(0,1){80}}}
        \put(20,10){\color{green}{\line(0,1){80}}}
        \put(20,90){\color{green}{\line(1,0){90}}}
	\linethickness{.1 mm}
	\multiput(130,10)(0,10){10}{\line(1,0){110}}
	\multiput(140,0)(10,0){10}{\line(0,1){110}}
    
	\put(180,50){\color{blue}{\circle*{5}}}
	\put(190,60){\color{blue}{\circle*{5}}}
 	\put(200,70){\color{blue}{\circle*{5}}}
 	\put(210,80){\color{blue}{\circle*{5}}}
 	\put(220,90){\color{blue}{\circle*{5}}}
	\put(170,40){\color{blue}{\circle*{5}}}  
	\put(160,30){\color{blue}{\circle*{5}}} 
	\put(150,20){\color{blue}{\circle*{5}}} 
	\put(140,10){\color{blue}{\circle*{5}}} 
	\put(150,50){\color{blue}{\circle*{5}}}
	\put(160,60){\color{blue}{\circle*{5}}}
 	\put(170,70){\color{blue}{\circle*{5}}}
 	\put(180,80){\color{blue}{\circle*{5}}}
 	\put(190,90){\color{blue}{\circle*{5}}}
	\put(140,40){\color{blue}{\circle*{5}}}   
 	\put(140,70){\color{blue}{\circle*{5}}}
 	\put(150,80){\color{blue}{\circle*{5}}}
 	\put(160,90){\color{blue}{\circle*{5}}}
 	\put(210,50){\color{blue}{\circle*{5}}}
 	\put(220,60){\color{blue}{\circle*{5}}}
	\put(200,40){\color{blue}{\circle*{5}}}  
	\put(190,30){\color{blue}{\circle*{5}}} 
	\put(180,20){\color{blue}{\circle*{5}}} 
	\put(170,10){\color{blue}{\circle*{5}}} 
 	\put(220,30){\color{blue}{\circle*{5}}} 
 	\put(210,20){\color{blue}{\circle*{5}}} 
 	\put(200,10){\color{blue}{\circle*{5}}} 
 	\put(230,10){\color{blue}{\circle*{5}}} 
        \put(230,40){\color{blue}{\circle*{5}}} 
        \put(230,70){\color{blue}{\circle*{5}}} 
  
	\linethickness{0.4mm}
	\put(140,0){\line(0,1){110}}
	\put(130,10){\line(1,0){110}}
        \put(150,20){\color{red}{\line(1,0){70}}}
        \put(220,20){\color{red}{\line(0,1){60}}}
        \put(150,20){\color{red}{\line(0,1){60}}}
        \put(150,80){\color{red}{\line(1,0){70}}}

        \put(140,10){\color{green}{\line(1,0){90}}}
        \put(230,10){\color{green}{\line(0,1){80}}}
        \put(140,10){\color{green}{\line(0,1){80}}}
        \put(140,90){\color{green}{\line(1,0){90}}}
\end{picture}
\caption{The sets $G_{7,8} \subset Y_{7,8}$ and the set $P(0) \cap Y_{7,8}$ } \label{GYexample}
\end{figure}
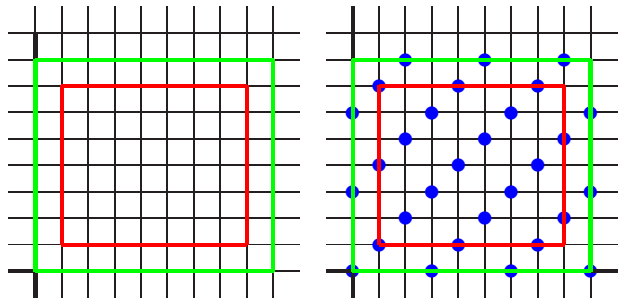

The following lemmas describe how to remove and rearrange some of the vertices in 
$P(i) \cap Y_{m,n}$ to obtain a smaller dominating set for $G_{m,n}$.

\begin{lemma}
There are three configurations of the NW and SE corners of $P(i) \cap Y_{m,n}$. 
In all three configurations we can create a smaller dominating set for that corner of $G_{m,n} \subset Y_{m,n}$ by moving some of the vertices of 
$P(i) \cap Y_{m,n} $ so that they lie inside $G_{m,n}$. In the process we can remove one vertex from our original dominating set $P(i) \cap Y_{m,n}$.
\end{lemma}
\begin{proof}
The three configurations for the NW corner of $P(i) \cap Y_{m,n}$ where $i =0,1,2$ are portrayed in order in Figure \ref{NW}. 
In each instance we have highlighted how  $k$ vertices circled in red can
be replaced with $k-1$ vertices highlighted by a bullseye.  
The pictures for the SE corner are the same up to $180^\circ$ rotation.
\end{proof}

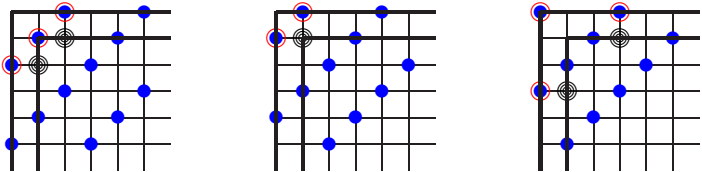
\begin{figure}[H]
\begin{picture}(300,70)(0,0)
	\linethickness{.1 mm}
	\multiput(10,10)(0,10){6}{\line(1,0){60}}
	\multiput(10,0)(10,0){6}{\line(0,1){60}}
	\put(10,40){\color{blue}{\circle*{5}}}
	\put(10,40){\color{red}{\circle{7}}}
	\put(20,50){\color{blue}{\circle*{5}}}
	\put(20,50){\color{red}{\circle{7}}}
	\put(30,60){\color{blue}{\circle*{5}}}
	\put(30,60){\color{red}{\circle{7}}}
	\put(20,40){\color{black}{\circle{5}}}
	\put(20,40){\color{black}{\circle{3}}}
	\put(20,40){\color{black}{\circle{7}}}
	\put(30,50){\color{black}{\circle{5}}}
	\put(30,50){\color{black}{\circle{3}}}
	\put(30,50){\color{black}{\circle{7}}}
	\put(10,10){\color{blue}{\circle*{5}}}
	\put(20,20){\color{blue}{\circle*{5}}}
	\put(30,30){\color{blue}{\circle*{5}}}
	\put(40,40){\color{blue}{\circle*{5}}}
	\put(50,50){\color{blue}{\circle*{5}}}
	\put(60,60){\color{blue}{\circle*{5}}}
	\put(40,10){\color{blue}{\circle*{5}}}
	\put(50,20){\color{blue}{\circle*{5}}}
	\put(60,30){\color{blue}{\circle*{5}}}
	
	\linethickness{0.4mm}
	\put(10,0){\line(0,1){60}}
	\put(10,60){\line(1,0){60}}
	\put(20,0){\line(0,1){50}}
	\put(20,50){\line(1,0){50}}
	\linethickness{.1 mm}
	\multiput(110,10)(0,10){6}{\line(1,0){60}}
	\multiput(110,0)(10,0){6}{\line(0,1){60}}
	\put(110,50){\color{blue}{\circle*{5}}}
	\put(110,50){\color{red}{\circle{7}}}
	\put(120,60){\color{blue}{\circle*{5}}}
	\put(120,60){\color{red}{\circle{7}}}
	\put(120,50){\color{black}{\circle{5}}}
	\put(120,50){\color{black}{\circle{3}}}
	\put(120,50){\color{black}{\circle{7}}}
	\put(110,20){\color{blue}{\circle*{5}}}
	\put(120,30){\color{blue}{\circle*{5}}}
	\put(130,40){\color{blue}{\circle*{5}}}
	\put(140,50){\color{blue}{\circle*{5}}}
	\put(150,60){\color{blue}{\circle*{5}}}
	\put(130,10){\color{blue}{\circle*{5}}}
	\put(140,20){\color{blue}{\circle*{5}}}
	\put(150,30){\color{blue}{\circle*{5}}}
	\put(160,40){\color{blue}{\circle*{5}}}
	
	\linethickness{0.4mm}
	\put(110,0){\line(0,1){60}}
	\put(110,60){\line(1,0){60}}
	\put(120,0){\line(0,1){50}}
	\put(120,50){\line(1,0){50}}
\linethickness{.1 mm}
\multiput(210,10)(0,10){6}{\line(1,0){60}}
\multiput(210,0)(10,0){6}{\line(0,1){60}}

 
 \put(210,60){\color{blue}{\circle*{5}}}
 \put(210,60){\color{red}{\circle{7}}}
%

 \put(210,30){\color{blue}{\circle*{5}}}
 \put(210,30){\color{red}{\circle{7}}}

  \put(220,30){\color{black}{\circle{5}}}
  \put(220,30){\color{black}{\circle{3}}}
  \put(220,30){\color{black}{\circle{7}}}

 \put(220,40){\color{blue}{\circle*{5}}}
 \put(230,50){\color{blue}{\circle*{5}}}
 \put(240,60){\color{blue}{\circle*{5}}}
 \put(250,40){\color{blue}{\circle*{5}}}
 \put(260,50){\color{blue}{\circle*{5}}}
 \put(240,60){\color{red}{\circle{7}}}
 \put(240,50){\color{black}{\circle{5}}}
 \put(240,50){\color{black}{\circle{3}}}
 \put(240,50){\color{black}{\circle{7}}}

 \put(220,10){\color{blue}{\circle*{5}}}
 \put(230,20){\color{blue}{\circle*{5}}}
 \put(240,30){\color{blue}{\circle*{5}}}

\linethickness{0.4mm}
\put(210,0){\line(0,1){60}}
\put(210,60){\line(1,0){60}}
\put(220,0){\line(0,1){50}}
\put(220,50){\line(1,0){50}}
\end{picture}
\caption{The three possible configurations for NW corner} \label{NW}
\end{figure}

\begin{lemma} \label{lemma:stones1}
There are three configurations of the SW and NE corners of $P(i) \cap Y_{m,n}$.  When $i =1$ or $2$ we can remove one vertex from $P(i) \cap Y_{m,n} $ and still dominate that corner of $G_{m,n}$.
For $P(0) \cap Y_{m,n} $ we can remove two vertices from that corner and dominate $G_{m,n}$.  
\end{lemma}

\begin{proof}
The configurations of $P(i) \cap Y_{m,n}$ for $i =0,1,2$ are depicted in Figure \ref{SW}.  
In each picture the vertices circled in red can be replaced by the vertices in the black bullseyes.
The NE corners have the same configuration after a $180^\circ$ rotation.
\end{proof}

 \begin{figure}[H]
\begin{picture}(70,70)(120,0)
\linethickness{.1 mm}
\multiput(10,10)(0,10){6}{\line(1,0){60}}
\multiput(10,10)(10,0){6}{\line(0,1){60}}

 \put(10,40){\color{blue}{\circle*{5}}}
 \put(10,40){\color{red}{\circle{7}}}

 \put(20,50){\color{blue}{\circle*{5}}}
 \put(30,60){\color{blue}{\circle*{5}}}
 
  \put(20,30){\color{black}{\circle{5}}}
  \put(20,30){\color{black}{\circle{3}}}
  \put(20,30){\color{black}{\circle{7}}}
 
  \put(30,20){\color{black}{\circle{5}}}
  \put(30,20){\color{black}{\circle{3}}}
  \put(30,20){\color{black}{\circle{7}}}

 \put(10,10){\color{blue}{\circle*{5}}}
 \put(10,10){\color{red}{\circle{7}}}

 \put(20,20){\color{blue}{\circle*{5}}}
 \put(20,20){\color{red}{\circle{7}}}

 \put(30,30){\color{blue}{\circle*{5}}}
 \put(40,40){\color{blue}{\circle*{5}}}
 \put(50,50){\color{blue}{\circle*{5}}}
 \put(60,60){\color{blue}{\circle*{5}}}

 \put(40,10){\color{blue}{\circle*{5}}}
 \put(40,10){\color{red}{\circle{7}}}

 \put(50,20){\color{blue}{\circle*{5}}}
 \put(60,30){\color{blue}{\circle*{5}}} 

\linethickness{0.4mm}
\put(10,10){\line(0,1){60}}
\put(10,10){\line(1,0){60}}
\put(20,20){\line(0,1){50}}
\put(20,20){\line(1,0){50}}

\linethickness{.1 mm}
\multiput(110,10)(0,10){6}{\line(1,0){60}}
\multiput(110,10)(10,0){6}{\line(0,1){60}}
 \put(160,40){\color{blue}{\circle*{5}}}
 \put(110,50){\color{blue}{\circle*{5}}}
 \put(120,60){\color{blue}{\circle*{5}}}
%
   \put(130,20){\color{black}{\circle{5}}}
   \put(130,20){\color{black}{\circle{3}}}
   \put(130,20){\color{black}{\circle{7}}}
 \put(160,10){\color{blue}{\circle*{5}}}
 \put(110,20){\color{blue}{\circle*{5}}}
 \put(110,20){\color{red}{\circle{7}}}
 \put(120,30){\color{blue}{\circle*{5}}}
 \put(130,40){\color{blue}{\circle*{5}}}
 \put(140,50){\color{blue}{\circle*{5}}}
 \put(150,60){\color{blue}{\circle*{5}}}
 \put(130,10){\color{blue}{\circle*{5}}}
 \put(130,10){\color{red}{\circle{7}}}
 \put(140,20){\color{blue}{\circle*{5}}}
 \put(150,30){\color{blue}{\circle*{5}}} 

\linethickness{0.4mm}
\put(110,10){\line(0,1){60}}
\put(110,10){\line(1,0){60}}
\put(120,20){\line(0,1){50}}
\put(120,20){\line(1,0){50}}

\linethickness{.1 mm}
\multiput(210,10)(0,10){6}{\line(1,0){60}}
\multiput(210,10)(10,0){6}{\line(0,1){60}}

 \put(250,40){\color{blue}{\circle*{5}}}
 \put(260,50){\color{blue}{\circle*{5}}}
 \put(210,60){\color{blue}{\circle*{5}}}
 
   \put(220,30){\color{black}{\circle{5}}}
   \put(220,30){\color{black}{\circle{3}}}
   \put(220,30){\color{black}{\circle{7}}}
 \put(250,10){\color{blue}{\circle*{5}}}
 \put(260,20){\color{blue}{\circle*{5}}}
 \put(210,30){\color{blue}{\circle*{5}}}
 \put(220,40){\color{blue}{\circle*{5}}}
 \put(230,50){\color{blue}{\circle*{5}}}
 \put(240,60){\color{blue}{\circle*{5}}}
 \put(220,10){\color{blue}{\circle*{5}}}
 \put(220,10){\color{red}{\circle{7}}}
 \put(210,30){\color{red}{\circle{7}}}
 \put(230,20){\color{blue}{\circle*{5}}}
 \put(240,30){\color{blue}{\circle*{5}}} 

\linethickness{0.4mm}
\put(210,10){\line(0,1){60}}
\put(210,10){\line(1,0){60}}
\put(220,20){\line(0,1){50}}
\put(220,20){\line(1,0){50}}
\end{picture}
\caption{The three possible configurations for SW corner} \label{SW}
\end{figure}
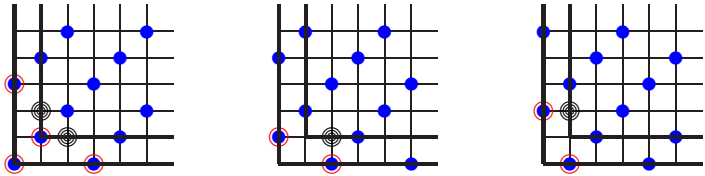

Lemmas \ref{lemma:stones1} and \ref{lemma:stones} show how to delete certain vertices from the corners of $P(i) \cap Y_{m,n}$ to obtain smaller dominating sets for $G_{m,n}$.  
However, these dominating sets contain points in the complement $Y_{m,n}-G_{m,n}$.
To create a dominating set that is completely contained in $G_{m,n}$ simply move each vertex in $Y_{m,n} - G_{m,n}$ to its nearest neighbor inside $G_{m,n}$. 
Denote the resulting dominating set $D_{m,n}$.  Figure \ref{Dmnexample} shows how the set $P(0) \cap Y_{7,8}$ can be modified to obtain the dominating set $D_{7,8}$ 
\begin{figure}[H]
\begin{picture}(410,110)(-30,0)
	
\linethickness{.1 mm}
	\multiput(130,10)(0,10){10}{\line(1,0){110}}
	\multiput(140,0)(10,0){10}{\line(0,1){110}}
    
	\put(180,50){\color{blue}{\circle*{5}}}
	\put(190,60){\color{blue}{\circle*{5}}}
 	\put(200,70){\color{blue}{\circle*{5}}}
 	\put(210,80){\color{blue}{\circle*{5}}}
	\put(170,40){\color{blue}{\circle*{5}}}  
	\put(160,30){\color{blue}{\circle*{5}}} 
	\put(150,30){\color{black}{\circle{3}}} 
        \put(150,30){\color{black}{\circle{5}}}
        \put(150,30){\color{black}{\circle{7}}}
 
        \put(160,20){\color{black}{\circle{3}}} 
        \put(160,20){\color{black}{\circle{5}}}
        \put(160,20){\color{black}{\circle{7}}}
	\put(150,50){\color{blue}{\circle*{5}}}
	\put(160,60){\color{blue}{\circle*{5}}}
        \put(170,70){\color{blue}{\circle*{5}}}
 	\put(180,80){\color{blue}{\circle*{5}}}
 	\put(190,80){\color{black}{\circle{5}}} 
	\put(190,80){\color{black}{\circle{3}}} 
	\put(190,80){\color{black}{\circle{7}}}
 	\put(150,70){\color{black}{\circle{5}}}
        \put(150,70){\color{black}{\circle{3}}}
        \put(150,70){\color{black}{\circle{7}}}
     	\put(160,80){\color{black}{\circle{5}}}
        \put(160,80){\color{black}{\circle{3}}}
        \put(160,80){\color{black}{\circle{7}}}
 	\put(210,50){\color{blue}{\circle*{5}}}
        \put(220,60){\color{blue}{\circle*{5}}}
 	\put(220,70){\color{black}{\circle{5}}}
 	\put(220,70){\color{black}{\circle{3}}}
 	\put(220,70){\color{black}{\circle{7}}}
	\put(200,40){\color{blue}{\circle*{5}}}  
	\put(190,30){\color{blue}{\circle*{5}}} 
	\put(180,20){\color{blue}{\circle*{5}}} 
 	\put(220,30){\color{blue}{\circle*{5}}} 
 	\put(210,20){\color{blue}{\circle*{5}}} 
 	\put(200,20){\color{black}{\circle{5}}} 
	\put(200,20){\color{black}{\circle{3}}}
	\put(200,20){\color{black}{\circle{7}}}
        \put(220,40){\color{black}{\circle{5}}}
        \put(220,40){\color{black}{\circle{3}}}
        \put(220,40){\color{black}{\circle{7}}} 
  
	\linethickness{0.4mm}
	\put(140,0){\line(0,1){110}}
	\put(130,10){\line(1,0){110}}
        \put(150,20){\color{red}{\line(1,0){70}}}
        \put(220,20){\color{red}{\line(0,1){60}}}
        \put(150,20){\color{red}{\line(0,1){60}}}
        \put(150,80){\color{red}{\line(1,0){70}}}

        \put(140,10){\color{green}{\line(1,0){90}}}
        \put(230,10){\color{green}{\line(0,1){80}}}
        \put(140,10){\color{green}{\line(0,1){80}}}
        \put(140,90){\color{green}{\line(1,0){90}}}
\end{picture}
\caption{The dominating set $D_{7,8}$} \label{Dmnexample}
\end{figure}
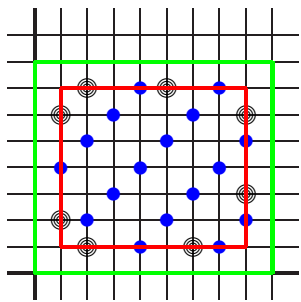
The cardinality of the set $D_{m,n}$ gives an upper bound on the domination number of $G_{m,n}$.
We now set about counting the number of elements in this dominating set.

First we will see that for any $m$ and $n$ in $\ZZ$ and $i \in \ZZ_3$ the number of elements in the intersection $P(i) \cap Y_{m,n}$
is either $ \left \lceil \frac{ (m+2)(n+2)}{3}  \right \rceil$ or $\left \lfloor \frac{(m+2)(n+2)}{3} \right \rfloor $.
The exact number depends on the three values of $i$, $\overline{m}$, and $\overline{n}$, where $\overline{m}$ and $\overline{n}$ are the
residues of $m$ and $n$ modulo $3$. 

\begin{lemma} \label{lemma:stones} The number of elements in $P(i) \cap Y_{m,n}$ is
\[ |P(i) \cap Y_{m,n} | =  \left \lceil \frac{ (m+2)(n+2)}{3} \right \rceil  \]
when the triple $(\overline{m}, \overline{n}, i)$ is contained in the set $\Psi$ where 
\[ \Psi = \left \{  \begin{matrix}     	
(0, 0, 0) ,
(0, 1, 0) ,
(0, 1, 1) ,
(0, 1, 2),
(0, 2, 0),
(0, 2, 1) ,
(1, 0, 0),
(1, 0, 1) , \\
(1, 0, 2), 
(1, 1, 0), 
(1, 1, 1),
(1, 1, 2), 
(1, 2 ,0), 
(1, 2, 1),
(1, 2, 2),
(2, 0, 0) ,\\
(2, 0, 2),
(2, 1, 0),
(2, 1, 1), 
(2, 1, 2), 
(2, 2, 0) \end{matrix} \right \}  \]
For all other triples $(\overline{m}, \overline{n}, i)$
\[ |P(i) \cap Y_{m,n} | =  \left \lfloor \frac{ (m+2)(n+2)}{3} \right \rfloor  .\]

\end{lemma}
\begin{proof}
The set $Y_{m,n}$ is a subset of $\ZZ \times \ZZ$ containing $(m+2) \times (n+2)$ vertices.
Use the Euclidean algorithm to write $m+2 = 3a+\overline{m+2}$ and $n+2 = 3b + \overline{n+2}$ where $\overline{m+2}$ and $\overline{n+2}$ are the residues modulo $3$. 
Break $Y_{m,n}$ into four regions $R_1, R_2, R_3,$ and $R_4$ as depicted in Figure 
\ref{fourparts1} with $3a\times 3b$,  $3a \times \overline{n+2}$, $\overline{m +2 } \times 3b$, and $\overline{m+ 2} \times \overline{n+2}$ vertices in each region respectively. 
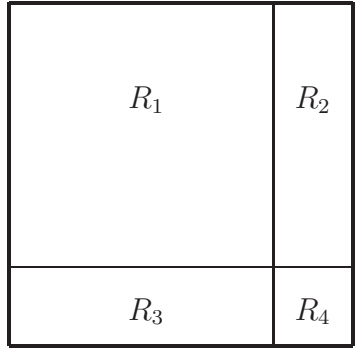
\begin{figure}[h]
\begin{picture}(130,130)(0,0)
	\linethickness{0.2mm}
	\put(100,0){\line(0,1){130}}
	\put(0,30){\line(1,0){130}}
        \put(45,90){$R_1$}
        \put(108,90){$R_2$}
        \put(45,10){$R_3$}
        \put(108,10){$R_4$}
	\linethickness{0.4mm}
	\put(0,0){\line(0,1){130}}
	\put(0,130){\line(1,0){130}}
	\put(130,0){\line(0,1){130}}
	\put(0,0){\line(1,0){130}}
\end{picture}
  \caption{Four regions of $Y_{m,n}$} \label{fourparts1}
\end{figure}

We now count how many elements of $P(i)$ are in each region.  
Since $\phi$ is a homomorphism from $\ZZ \times \ZZ$ to $\ZZ_3$, every three consecutive vertices in any row or column of $G_{m,n}$ will have exactly one element 
in the inverse image $\phi^{-1}(i) =P(i)$.  It follows that $|P(i) \cap R_1| = 3ab$, $|P(i) \cap R_2|= a (\overline{n+2})$, and $|P(i) \cap R_3| = (\overline{m+2})b$.

All that remains is to count the number of elements in $P(i) \cap R_4$.  There are $9$ possible sizes for the  $\overline{m+2} \times \overline{n+2}$ grid $R_4$ and three possible values for $i \in \ZZ_3$ defining $P(i)$.
Hence by checking all 27 of these cases, we can verify that when  $(\overline{m+2},\overline{n+2}, i)$ is in the set 
\[ \Psi' = \left \{ \begin{matrix} (0, 0, 0), 
(0, 0, 1),
(0, 0, 2), 
(0, 1, 0), 
(0, 1, 1),
(0, 1, 2), 
(0, 2, 0), 
(0, 2, 1), \\
(0, 2, 2),
(1, 0, 0),
(1, 0, 1),
(1, 0, 2), 
(1, 1, 0),
(1, 2, 0),
(1, 2, 2),
(2, 0, 0),\\
(2, 0, 1),
(2, 0, 2),
(2, 1, 0),
(2, 1, 1),
(2, 2, 0) \end{matrix} \right \} \] the cardinality of $P(i) \cap R_4$ is
$ | P(i) \cap R_4 | =  \left \lceil \frac{(\overline{m+2})(\overline{n+2})}{3} \right \rceil $, and 
$| P(i) \cap R_4 | =  \left \lfloor \frac{(\overline{m+2})(\overline{n+2})}{3} \right \rfloor $ otherwise \footnote{The calculation is available at the site \\
{\tt https://cloud.sagemath.com/projects/26b983ae-a894-47c0-bd54-c73b071e555c/files/(2,2)BDC.sagews }}.
However, a triple $(\overline{m+2}, \overline{n+2}, i) $ is contained in the set $\Psi'$ if and only if the triple $(\overline{m}, \overline{n} , i )$ is contained in the set
\[ \Psi = \left \{  \begin{matrix}     	
(0, 0, 0) ,
(0, 1, 0) ,
(0, 1, 1) ,
(0, 1, 2),
(0, 2, 0),
(0, 2, 1) ,
(1, 0, 0),
(1, 0, 1) , \\
(1, 0, 2), 
(1, 1, 0), 
(1, 1, 1),
(1, 1, 2), 
(1, 2 ,0), 
(1, 2, 1),
(1, 2, 2),
(2, 0, 0) ,\\
(2, 0, 2),
(2, 1, 0),
(2, 1, 1), 
(2, 1, 2), 
(2, 2, 0) \end{matrix} \right \} .\]
Thus, the cardinality of $P(i) \cap R_4$ is $ | P(i) \cap R_4| = \left \lceil \frac{(\overline{m+2})(\overline{n+2})}{3} \right \rceil $
when $(\overline{m}, \overline{n}, i)$ is in $\Psi$, and it is  $ | P(i) \cap R_4| = \left \lfloor \frac{(\overline{m+2})(\overline{n+2})}{3} \right \rfloor $ 
when $(\overline{m}, \overline{n}, i)$ is not in $\Psi$. 
It follows that \begin{align*} |P(i) \cap Y_{m,n}| & =  |R_1 \cap P(i) |+  |R_2 \cap P(i) | +  |R_3 \cap P(i)|+ |R_4 \cap P(i)| \\
          & =  3ab+ (\overline{m+2}) a + b( \overline{n+2})+  |R_4 \cap P(i)| \\
          & =  3ab+ (\overline{m+2}) a + b( \overline{n+2})+ \left \lceil \frac{(\overline{m+2})(\overline{n+2})}{3} \right \rceil \\
          & = \left \lceil \frac{( 3a+\overline{m+2})(3b + \overline{n+2} )}{3}  \right \rceil  \\ & = \left \lceil \frac{(m+2)(n+2)}{3}  \right \rceil  \end{align*}
when $(\overline{m}, \overline{n}, i)$ is in $\Psi$, and the analagous calculation shows that $|P(i) \cap Y_{m,n}| =  \left \lfloor \frac{(m+2)(n+2)}{3}  \right \rfloor$
when $(\overline{m}, \overline{n}, i)$ is not in $\Psi$.
\end{proof}

Having counted the number of elements in $P(i) \cap Y_{m,n}$, we are now ready to count the number of elements in the dominating sets $D_{m,n}$. 
All we have to do is count the number of elements in $P(i) \cap Y_{m,n}$ that can be deleted by using Lemmas \ref{lemma:stones1} and \ref{lemma:stones}.  
Note that in order to delete vertices from all four corners of $P(i) \cap Y_{m,n}$ we must assume that $m$ and $n$ are both at least $6 = 2 \cdot 3$. 
Otherwise there would be overlap between the regions depicted in Figures \ref{NW} and \ref{SW}, and we may not be able to delete a vertex from each corner.
\begin{theorem}
 When $m,n \geq 6$, the $(2,2)$  broadcast domination number of $G_{m,n}$ satisfies the inequality
\[ \gamma_{2,2} (G_{m,n}) \leq  \left \lceil  \frac{(m+2)(n+2)}{3}  \right \rceil  - 6 \] when $m \equiv n  \mod 3$, and otherwise it satisfies
\[ \gamma_{2,2} (G_{m,n}) \leq \left  \lceil \frac{(m+2)(n+2)}{3} \right \rceil- 5. \]

\end{theorem}

\begin{proof}
When $m \equiv n \mod 3$, Lemma \ref{lemma:stones1}  implies that we can delete one vertex from each of the NW and SE corners of $P(0) \cap Y_{m,n}$ and
 two vertices each from the NE and SW corners of $P(0) \cap Y_{m,n}$ to obtain the dominating set $D_{m,n}$. Thus $|D_{m,n}| = |P(0) \cap Y_{m,n}|-6$.  
Also, when $m \equiv n \mod 3$, there are $\left \lceil \frac{(m+2)(n+2)}{3} \right \rceil$ vertices in $P(0) \cap Y_{m,n}$ by Lemma \ref{lemma:stones}. 
The resulting dominating set $D_{m,n}$ has $\left  \lceil \frac{(n+2)(m+2)}{3} \right \rceil - 6$ vertices in it. 
On the other hand, for $i =1$ and $2$ there are $\left  \lfloor \frac{(n+2)(m+2)}{3} \right \rfloor$ vertices in $P(i) \cap Y_{m, n}$,
but we can only remove 4 vertices and still dominate $G_{m,n}$.
So the best bound on the $(2,2)$  broadcast domination number is  
\[ \gamma_{2,2}(G_{m,n}) \leq \left  \lceil \frac{(n+2)(m+2)}{3} \right \rceil - 6 \]
when $m \equiv n \mod 3$.

When $ m \not \equiv n \mod 3$, there are  $ \left  \lceil \frac{(n+2)(m+2)}{3} \right \rceil $ vertices in $P(0)\cap Y_{m,n}$ and $P(i) \cap Y_{m,n}$  where $m - n \equiv i \mod 3$. 
Also when $m-n \equiv i \mod 3$, we can delete two vertices from the SW corner of $P(0) \cap Y_{m,n}$
and two vertices from the NE corner of $P(i) \cap Y_{m,n}.$ Thus we can delete a total five vertices from both $P(0) \cap Y_{m,n}$ and $P(i) \cap Y_{m,n}$.  

On the other hand, there are  $ \left  \lfloor \frac{(n+2)(m+2)}{3} \right \rfloor $ vertices in $P(j) \cap Y_{m,n}$ where $j \neq 0$ and $m -n \not \equiv j \mod 3$.  
In this case, we can only remove four vertices from $P(j)\cap Y_{m,n}$ and still dominate $G_{m,n}$.  
In any case, the $(2,2)$ broadcast domination number satisfies
\[ \gamma_{2,2}(G_{m,n}) \leq \left  \lceil \frac{(n+2)(m+2)}{3} \right \rceil - 5 \]
when $m \not \equiv n \mod 3$.   
\end{proof}

\subsection{(3,1) broadcast domination}
In this subsection we give an upper bound for the $(3,1)$ broadcast domination number of an $m \times n$ grid.   
We start by identifying a family of optimal $(3,1)$ broadcast dominating sets for $\ZZ \times \ZZ$.
Then we show how these sets can be modified to define an efficient dominating set for each $G_{m,n}$.  
  
Let $i \in \ZZ_{13}$ and let  $P(i)= \phi^{-1}(i)$ denote the preimage of $i$ under the homomorphism $\phi: \ZZ \times \ZZ \rightarrow \ZZ_{13}$ defined by $\phi(x,y) = 4x+7y$. 
The set $P(0)$ is depicted in Figure \ref{P(t)}.  The other sets are obtained by shifting $P(0)$ appropriately.

The set $P(i)$ dominates $\ZZ \times \ZZ $ optimally in the sense that all of $\ZZ \times \ZZ$ is dominated, and no vertex in $\ZZ \times \ZZ$ is dominated by more than one vertex of $P(i)$.
For any $i \in \ZZ_{13}$ every thirteen consecutive vertices in a row or column contain exactly one element of $P(i).$

\begin{figure}[h]
\begin{picture}(210,210)(100,0)
	\linethickness{.1 mm}
	\multiput(110,10)(0,10){19}{\line(1,0){200}}
	\multiput(120,0)(10,0){19}{\line(0,1){200}}
        \put(210,100){\color{blue}{\circle*{5}}}
        \put(240,120){\color{blue}{\circle*{5}}}
        \put(180,80){\color{blue}{\circle*{5}}}
        \put(150,60){\color{blue}{\circle*{5}}}
        \put(120,40){\color{blue}{\circle*{5}}}
        \put(270,140){\color{blue}{\circle*{5}}}
        \put(300,160){\color{blue}{\circle*{5}}}

        \put(190,130){\color{blue}{\circle*{5}}}
        \put(220,150){\color{blue}{\circle*{5}}}
        \put(160,110){\color{blue}{\circle*{5}}}
        \put(130,90){\color{blue}{\circle*{5}}}
        \put(250,170){\color{blue}{\circle*{5}}}
        \put(280,190){\color{blue}{\circle*{5}}}
        \put(170,160){\color{blue}{\circle*{5}}}
        \put(200,180){\color{blue}{\circle*{5}}}
        \put(140,140){\color{blue}{\circle*{5}}}
        \put(150,190){\color{blue}{\circle*{5}}}
        \put(120,170){\color{blue}{\circle*{5}}}
        
        \put(230,70){\color{blue}{\circle*{5}}}
        \put(260,90){\color{blue}{\circle*{5}}}
        \put(200,50){\color{blue}{\circle*{5}}}
        \put(170,30){\color{blue}{\circle*{5}}}
        \put(140,10){\color{blue}{\circle*{5}}}
        \put(290,110){\color{blue}{\circle*{5}}}
       
        \put(250,40){\color{blue}{\circle*{5}}}
        \put(280,60){\color{blue}{\circle*{5}}}
        \put(220,20){\color{blue}{\circle*{5}}}
        \put(270,10){\color{blue}{\circle*{5}}}
        \put(300,30){\color{blue}{\circle*{5}}}

	\linethickness{0.4mm}
	\put(210,0){\line(0,1){200}}
	\put(110,100){\line(1,0){200}}
\end{picture}
\caption{ The optimal dominating set $P(0)$ of $\ZZ \times \ZZ$  } \label{P(t)}
\end{figure}
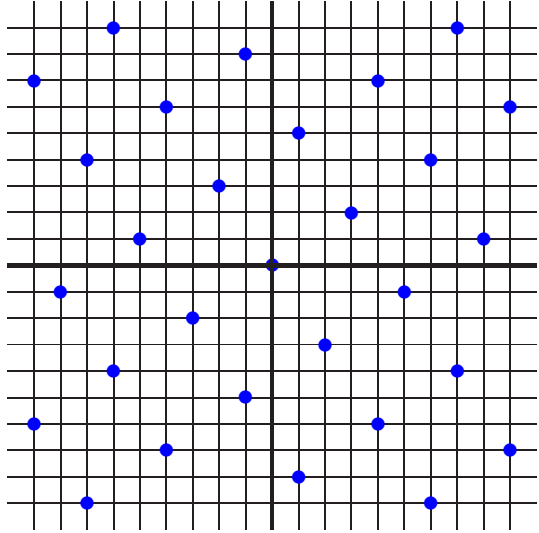
%
%

Embed $G_{m,n}$ into $\ZZ \times \ZZ$ as the following set: \[ G_{m,n} = \{ (a,b) \in \ZZ \times \ZZ  \mid 2 \leq a \leq n+1  \text{ and } 2 \leq b \leq m+1 \}, \]
and let $Y_{m,n} \cong G_{m+4,n+4} $ denote the neighborhood of $G_{m,n}$: \[Y_{m,n} = \{ (a,b) \in \ZZ \times \ZZ \mid 0 \leq a \leq n+3 \text{ and }  0 \leq b \leq m+3 \}. \]

\begin{lemma} \label{lemma:rearrange}
We can delete one vertex from any of the dominating sets $P(i) \cap Y_{m,n}$ at each corner of $G_{m,n}$ and the resulting set will still dominate $G_{m,n}$. 
\end{lemma}
\begin{proof}
 Let $P(i) = \phi^{-1}(i)$ be one of the optimal dominating sets of $\ZZ \times \ZZ$.
 There are 13 different cases for what a corner of $P(i) \cap Y_{m,n}$ can look like.  
They are all depicted in Figure \ref{delete}.  
We display the 13 different configurations of the Northwest corner of $P(i) \cap Y_{m,n}$.
The other three corners are identical up to a rotation or flip across the diagonal of the depicted configurations.
In each of the 13 cases, we can delete one vertex from each corner $P(i) \cap Y_{m,n}$ and still obtain a dominating set of $G_{m,n}$.
This is done by replacing the vertices circled in red with those highlighted by a bullseye.  
\end{proof}

\begin{figure}
\scalebox{.85}{
\begin{picture}(130,130)(150,0)
	\linethickness{.1 mm}
	\multiput(10,10)(0,10){13}{\line(1,0){130}}
	\multiput(10,0)(10,0){13}{\line(0,1){130}}
        \put(10,130){\color{red}{\circle{7}}}
	\put(10,130){\color{blue}{\circle*{5}}}
	\put(30,100){\color{blue}{\circle*{5}}}
	\put(50,70){\color{blue}{\circle*{5}}}
	\put(80,90){\color{blue}{\circle*{5}}}
	\put(110,110){\color{blue}{\circle*{5}}}
	\put(60,120){\color{blue}{\circle*{5}}}
	\put(20,50){\color{blue}{\circle*{5}}}
	\put(40,20){\color{blue}{\circle*{5}}}
	\put(70,40){\color{blue}{\circle*{5}}}
	\put(100,60){\color{blue}{\circle*{5}}}
	\put(130,80){\color{blue}{\circle*{5}}}
	\put(90,10){\color{blue}{\circle*{5}}}
	\put(120,30){\color{blue}{\circle*{5}}}
	
	\linethickness{0.4mm}
	\put(10,0){\line(0,1){130}}
	\put(10,130){\line(1,0){130}}
	\put(30,0){\line(0,1){110}}
	\put(30,110){\line(1,0){110}}

	\linethickness{.1 mm}
	\multiput(150,10)(0,10){13}{\line(1,0){130}}
	\multiput(150,0)(10,0){13}{\line(0,1){130}}
	\put(270,130){\color{blue}{\circle*{5}}}
	\put(160,100){\color{blue}{\circle*{5}}}
	\put(160,100){\color{red}{\circle{7}}}
	\put(180,70){\color{blue}{\circle*{5}}}
	\put(210,90){\color{blue}{\circle*{5}}}
	\put(240,110){\color{blue}{\circle*{5}}}
	\put(190,120){\color{blue}{\circle*{5}}}
	\put(190,120){\color{red}{\circle{7}}}
	\put(180,110){\color{black}{\circle{7}}}
	\put(180,110){\color{black}{\circle{3}}}
	\put(180,110){\color{black}{\circle{5}}}
	\put(170,70){\color{black}{\circle{5}}}
	\put(170,70){\color{black}{\circle{3}}}
	\put(170,70){\color{black}{\circle{7}}}
	\put(150,50){\color{blue}{\circle*{5}}}
	\put(150,50){\color{red}{\circle{7}}}
	\put(170,20){\color{blue}{\circle*{5}}}
	\put(200,40){\color{blue}{\circle*{5}}}
	\put(230,60){\color{blue}{\circle*{5}}}
	\put(260,80){\color{blue}{\circle*{5}}}
	\put(220,10){\color{blue}{\circle*{5}}}
	\put(250,30){\color{blue}{\circle*{5}}}
	
	\linethickness{0.4mm}
	\put(150,0){\line(0,1){130}}
	\put(150,130){\line(1,0){130}}
	\put(170,0){\line(0,1){110}}
	\put(170,110){\line(1,0){110}}

	\linethickness{.1 mm}
	\multiput(290,10)(0,10){13}{\line(1,0){130}}
	\multiput(290,0)(10,0){13}{\line(0,1){130}}
	\put(400,130){\color{blue}{\circle*{5}}}
	\put(290,100){\color{blue}{\circle*{5}}}
	\put(290,100){\color{red}{\circle{7}}}
	\put(310,70){\color{blue}{\circle*{5}}}
	\put(340,90){\color{blue}{\circle*{5}}}
	\put(370,110){\color{blue}{\circle*{5}}}
	\put(320,120){\color{blue}{\circle*{5}}}
	\put(320,120){\color{red}{\circle{7}}}
	\put(310,110){\color{black}{\circle{7}}}
	\put(310,110){\color{black}{\circle{5}}}
	\put(310,110){\color{black}{\circle{3}}}
	\put(410,50){\color{blue}{\circle*{5}}}
	\put(300,20){\color{blue}{\circle*{5}}}
	\put(330,40){\color{blue}{\circle*{5}}}
	\put(360,60){\color{blue}{\circle*{5}}}
	\put(390,80){\color{blue}{\circle*{5}}}
	\put(350,10){\color{blue}{\circle*{5}}}
	\put(380,30){\color{blue}{\circle*{5}}}
	
	\linethickness{0.4mm}
	\put(290,0){\line(0,1){130}}
	\put(290,130){\line(1,0){130}}
	\put(310,0){\line(0,1){110}}
	\put(310,110){\line(1,0){110}}
\end{picture}}\\

\vspace{12 pt}
\scalebox{.85}{
\begin{picture}(130,130)(150,0)
	\linethickness{.1 mm}
	\multiput(10,10)(0,10){13}{\line(1,0){130}}
	\multiput(10,0)(10,0){13}{\line(0,1){130}}
	\put(110,130){\color{blue}{\circle*{5}}}
	\put(130,100){\color{blue}{\circle*{5}}}
	\put(20,70){\color{blue}{\circle*{5}}}
	\put(50,90){\color{blue}{\circle*{5}}}
	\put(80,110){\color{blue}{\circle*{5}}}
	\put(30,120){\color{blue}{\circle*{5}}}
	\put(120,50){\color{blue}{\circle*{5}}}
	\put(10,20){\color{blue}{\circle*{5}}}
	\put(40,40){\color{blue}{\circle*{5}}}
	\put(70,60){\color{blue}{\circle*{5}}}
	\put(100,80){\color{blue}{\circle*{5}}}
	\put(60,10){\color{blue}{\circle*{5}}}
	\put(90,30){\color{blue}{\circle*{5}}}
	\put(110,130){\color{red}{\circle{7}}}
	\put(80,110){\color{red}{\circle{7}}}
	\put(20,70){\color{red}{\circle{7}}}
	\put(30,120){\color{red}{\circle{7}}}
	\put(50,90){\color{red}{\circle{7}}}
	\put(40,100){\color{black}{\circle{7}}}
	\put(40,100){\color{black}{\circle{5}}}
	\put(40,100){\color{black}{\circle{3}}}
	\put(60,90){\color{black}{\circle{7}}}
	\put(60,90){\color{black}{\circle{5}}}
	\put(60,90){\color{black}{\circle{3}}}
	\put(30,70){\color{black}{\circle{7}}}
	\put(30,70){\color{black}{\circle{5}}}
	\put(30,70){\color{black}{\circle{3}}}
	\put(90,110){\color{black}{\circle{7}}}
	\put(90,110){\color{black}{\circle{5}}}
	\put(90,110){\color{black}{\circle{3}}}
	
	\linethickness{0.4mm}
	\put(10,0){\line(0,1){130}}
	\put(10,130){\line(1,0){130}}
	\put(30,0){\line(0,1){110}}
	\put(30,110){\line(1,0){110}}
	\linethickness{.1 mm}
	\multiput(150,10)(0,10){13}{\line(1,0){130}}
	\multiput(150,0)(10,0){13}{\line(0,1){130}}
	\put(240,130){\color{blue}{\circle*{5}}}
	\put(260,100){\color{blue}{\circle*{5}}}
	\put(150,70){\color{blue}{\circle*{5}}}
	\put(150,70){\color{red}{\circle{7}}}
	\put(180,90){\color{blue}{\circle*{5}}}
	\put(180,90){\color{red}{\circle{7}}}
	\put(210,110){\color{red}{\circle{7}}}
	\put(210,110){\color{blue}{\circle*{5}}}
	\put(160,120){\color{blue}{\circle*{5}}}
	\put(160,120){\color{red}{\circle{7}}}
	\put(250,50){\color{blue}{\circle*{5}}}
	\put(270,20){\color{blue}{\circle*{5}}}
	\put(170,40){\color{blue}{\circle*{5}}}
	\put(200,60){\color{blue}{\circle*{5}}}
	\put(230,80){\color{blue}{\circle*{5}}}
	\put(200,10){\color{blue}{\circle*{5}}}
	\put(220,30){\color{blue}{\circle*{5}}}
	\put(170,110){\color{black}{\circle{7}}}
	\put(170,110){\color{black}{\circle{5}}}
	\put(170,110){\color{black}{\circle{3}}}
	\put(180,80){\color{black}{\circle{7}}}
	\put(180,80){\color{black}{\circle{5}}}
	\put(180,80){\color{black}{\circle{3}}}
	\put(210,100){\color{black}{\circle{7}}}
	\put(210,100){\color{black}{\circle{5}}}
	\put(210,100){\color{black}{\circle{3}}}
	
	\linethickness{0.4mm}
	\put(150,0){\line(0,1){130}}
	\put(150,130){\line(1,0){130}}
	\put(170,0){\line(0,1){110}}
	\put(170,110){\line(1,0){110}}

	\linethickness{.1 mm}
	\multiput(290,10)(0,10){13}{\line(1,0){130}}
	\multiput(290,0)(10,0){13}{\line(0,1){130}}
	\put(370,130){\color{blue}{\circle*{5}}}
	\put(390,100){\color{blue}{\circle*{5}}}
	\put(410,70){\color{blue}{\circle*{5}}}
	\put(310,90){\color{blue}{\circle*{5}}}
	\put(340,110){\color{blue}{\circle*{5}}}
	\put(290,120){\color{blue}{\circle*{5}}}
	\put(290,120){\color{red}{\circle{7}}}
	\put(380,50){\color{blue}{\circle*{5}}}
	\put(400,20){\color{blue}{\circle*{5}}}
	\put(300,40){\color{blue}{\circle*{5}}}
	\put(330,60){\color{blue}{\circle*{5}}}
	\put(360,80){\color{blue}{\circle*{5}}}
	\put(320,10){\color{blue}{\circle*{5}}}
	\put(350,30){\color{blue}{\circle*{5}}}
	
	\linethickness{0.4mm}
	\put(290,0){\line(0,1){130}}
	\put(290,130){\line(1,0){130}}
	\put(310,0){\line(0,1){110}}
	\put(310,110){\line(1,0){110}}
\end{picture}}\\

\vspace{12 pt}
\scalebox{.85}{
\begin{picture}(130,130)(150,0)

	\linethickness{.1 mm}
	\multiput(10,10)(0,10){13}{\line(1,0){130}}
	\multiput(10,0)(10,0){13}{\line(0,1){130}}
	\put(80,130){\color{blue}{\circle*{5}}}
	\put(80,130){\color{red}{\circle{7}}}
	\put(100,100){\color{blue}{\circle*{5}}}
	\put(120,70){\color{blue}{\circle*{5}}}
	\put(20,90){\color{blue}{\circle*{5}}}
	\put(20,90){\color{red}{\circle{7}}}
	\put(30,90){\color{black}{\circle{7}}}
	\put(30,90){\color{black}{\circle{5}}}
	\put(30,90){\color{black}{\circle{3}}}
	\put(50,110){\color{blue}{\circle*{5}}}
	\put(50,110){\color{red}{\circle{7}}}
	\put(60,110){\color{black}{\circle{7}}}
	\put(60,110){\color{black}{\circle{5}}}
	\put(60,110){\color{black}{\circle{3}}}
	\put(130,120){\color{blue}{\circle*{5}}}
	\put(90,50){\color{blue}{\circle*{5}}}
	\put(110,20){\color{blue}{\circle*{5}}}
	\put(10,40){\color{blue}{\circle*{5}}}
	\put(40,60){\color{blue}{\circle*{5}}}
	\put(70,80){\color{blue}{\circle*{5}}}
	\put(30,10){\color{blue}{\circle*{5}}}
	\put(60,30){\color{blue}{\circle*{5}}}
	
	\linethickness{0.4mm}
	\put(10,0){\line(0,1){130}}
	\put(10,130){\line(1,0){130}}
	\put(30,0){\line(0,1){110}}
	\put(30,110){\line(1,0){110}}

	\linethickness{.1 mm}
	\multiput(150,10)(0,10){13}{\line(1,0){130}}
	\multiput(150,0)(10,0){13}{\line(0,1){130}}
	\put(210,130){\color{blue}{\circle*{5}}}
	\put(210,130){\color{red}{\circle{7}}}
	\put(230,100){\color{blue}{\circle*{5}}}
	\put(250,70){\color{blue}{\circle*{5}}}
	\put(150,90){\color{blue}{\circle*{5}}}
	\put(150,90){\color{red}{\circle{7}}}
	\put(170,90){\color{black}{\circle{7}}}
	\put(170,90){\color{black}{\circle{5}}}
	\put(170,90){\color{black}{\circle{3}}}
	\put(180,110){\color{blue}{\circle*{5}}}
	\put(180,110){\color{red}{\circle{7}}}
	\put(190,110){\color{black}{\circle{7}}}
	\put(190,110){\color{black}{\circle{5}}}
	\put(180,110){\color{black}{\circle{3}}}
	\put(260,120){\color{blue}{\circle*{5}}}
	\put(220,50){\color{blue}{\circle*{5}}}
	\put(240,20){\color{blue}{\circle*{5}}}
	\put(270,40){\color{blue}{\circle*{5}}}
	\put(170,60){\color{blue}{\circle*{5}}}
	\put(200,80){\color{blue}{\circle*{5}}}
	\put(160,10){\color{blue}{\circle*{5}}}
	\put(190,30){\color{blue}{\circle*{5}}}
	
	\linethickness{0.4mm}
	\put(150,0){\line(0,1){130}}
	\put(150,130){\line(1,0){130}}
	\put(170,0){\line(0,1){110}}
	\put(170,110){\line(1,0){110}}

	\linethickness{.1 mm}
	\multiput(290,10)(0,10){13}{\line(1,0){130}}
	\multiput(290,0)(10,0){13}{\line(0,1){130}}
	\put(340,130){\color{blue}{\circle*{5}}}
	\put(340,130){\color{red}{\circle{7}}}
	\put(360,100){\color{blue}{\circle*{5}}}
	\put(380,70){\color{blue}{\circle*{5}}}
	\put(410,90){\color{blue}{\circle*{5}}}
	\put(310,70){\color{black}{\circle{7}}}
	\put(310,70){\color{black}{\circle{5}}}
	\put(310,70){\color{black}{\circle{3}}}
	\put(310,110){\color{blue}{\circle*{5}}}
	\put(310,110){\color{red}{\circle{7}}}
	\put(320,110){\color{black}{\circle{7}}}
	\put(320,110){\color{black}{\circle{5}}}
	\put(320,110){\color{black}{\circle{3}}}
	\put(390,120){\color{blue}{\circle*{5}}}
	\put(350,50){\color{blue}{\circle*{5}}}
	\put(370,20){\color{blue}{\circle*{5}}}
	\put(400,40){\color{blue}{\circle*{5}}}
	\put(300,60){\color{blue}{\circle*{5}}}
	\put(300,60){\color{red}{\circle{7}}}
	\put(330,80){\color{blue}{\circle*{5}}}
	\put(290,10){\color{blue}{\circle*{5}}}
	\put(320,30){\color{blue}{\circle*{5}}}
	
	\linethickness{0.4mm}
	\put(290,0){\line(0,1){130}}
	\put(290,130){\line(1,0){130}}
	\put(310,0){\line(0,1){110}}
	\put(310,110){\line(1,0){110}}
\end{picture}}\\

\vspace{12 pt}
\scalebox{.85}{
\begin{picture}(130,130)(150,0)
	\linethickness{.1 mm}
	\multiput(10,10)(0,10){13}{\line(1,0){130}}
	\multiput(10,0)(10,0){13}{\line(0,1){130}}
	\put(50,130){\color{blue}{\circle*{5}}}
	\put(50,130){\color{red}{\circle{7}}}
	\put(70,100){\color{blue}{\circle*{5}}}
	\put(90,70){\color{blue}{\circle*{5}}}
	\put(120,90){\color{blue}{\circle*{5}}}
	\put(20,110){\color{blue}{\circle*{5}}}
	\put(20,110){\color{red}{\circle{7}}}
	\put(30,110){\color{black}{\circle{7}}}
	\put(30,110){\color{black}{\circle{5}}}
	\put(30,110){\color{black}{\circle{3}}}
	\put(100,120){\color{blue}{\circle*{5}}}
	\put(60,50){\color{blue}{\circle*{5}}}
	\put(80,20){\color{blue}{\circle*{5}}}
	\put(110,40){\color{blue}{\circle*{5}}}
	\put(10,60){\color{blue}{\circle*{5}}}
	\put(40,80){\color{blue}{\circle*{5}}}
	\put(130,10){\color{blue}{\circle*{5}}}
	\put(30,30){\color{blue}{\circle*{5}}}
	
	\linethickness{0.4mm}
	\put(10,0){\line(0,1){130}}
	\put(10,130){\line(1,0){130}}
	\put(30,0){\line(0,1){110}}
	\put(30,110){\line(1,0){110}}

	\linethickness{.1 mm}
	\multiput(150,10)(0,10){13}{\line(1,0){130}}
	\multiput(150,0)(10,0){13}{\line(0,1){130}}
	\put(180,130){\color{blue}{\circle*{5}}}
	\put(200,100){\color{blue}{\circle*{5}}}
	\put(220,70){\color{blue}{\circle*{5}}}
	\put(250,90){\color{blue}{\circle*{5}}}
	\put(150,110){\color{blue}{\circle*{5}}}
	\put(150,110){\color{red}{\circle{7}}}
	\put(170,110){\color{black}{\circle{7}}}
	\put(170,110){\color{black}{\circle{5}}}
	\put(170,110){\color{black}{\circle{3}}}
	\put(230,120){\color{blue}{\circle*{5}}}
	\put(190,50){\color{blue}{\circle*{5}}}
	\put(210,20){\color{blue}{\circle*{5}}}
	\put(240,40){\color{blue}{\circle*{5}}}
	\put(270,60){\color{blue}{\circle*{5}}}
	\put(170,80){\color{blue}{\circle*{5}}}
	\put(260,10){\color{blue}{\circle*{5}}}
	\put(160,30){\color{blue}{\circle*{5}}}
	\put(180,130){\color{red}{\circle{7}}}
	
	\linethickness{0.4mm}
	\put(150,0){\line(0,1){130}}
	\put(150,130){\line(1,0){130}}
	\put(170,0){\line(0,1){110}}
	\put(170,110){\line(1,0){110}}

	\linethickness{.1 mm}
	\multiput(290,10)(0,10){13}{\line(1,0){130}}
	\multiput(290,0)(10,0){13}{\line(0,1){130}}
	\put(310,130){\color{blue}{\circle*{5}}}
	\put(310,130){\color{red}{\circle{7}}}
	\put(330,100){\color{blue}{\circle*{5}}}
	\put(350,70){\color{blue}{\circle*{5}}}
	\put(380,90){\color{blue}{\circle*{5}}}
	\put(410,110){\color{blue}{\circle*{5}}}
	\put(310,90){\color{black}{\circle{7}}}
	\put(310,90){\color{black}{\circle{5}}}
	\put(310,90){\color{black}{\circle{3}}}
	\put(360,120){\color{blue}{\circle*{5}}}
	\put(320,50){\color{blue}{\circle*{5}}}
	\put(340,20){\color{blue}{\circle*{5}}}
	\put(370,40){\color{blue}{\circle*{5}}}
	\put(400,60){\color{blue}{\circle*{5}}}
	\put(300,80){\color{blue}{\circle*{5}}}
	\put(300,80){\color{red}{\circle{7}}}
	\put(390,10){\color{blue}{\circle*{5}}}
	\put(290,30){\color{blue}{\circle*{5}}}
	
	\linethickness{0.4mm}
	\put(290,0){\line(0,1){130}}
	\put(290,130){\line(1,0){130}}
	\put(310,0){\line(0,1){110}}
	\put(310,110){\line(1,0){110}}
\end{picture}}\\

\vspace{12 pt}
\scalebox{.85}{
\begin{picture}(130,130)(150,0)
	\linethickness{.1 mm}
	\multiput(10,10)(0,10){13}{\line(1,0){130}}
	\multiput(10,0)(10,0){13}{\line(0,1){130}}
	\put(20,130){\color{blue}{\circle*{5}}}
	\put(20,130){\color{red}{\circle{7}}}
	\put(40,100){\color{blue}{\circle*{5}}}
	\put(60,70){\color{blue}{\circle*{5}}}
	\put(90,90){\color{blue}{\circle*{5}}}
	\put(120,110){\color{blue}{\circle*{5}}}

	\put(70,120){\color{blue}{\circle*{5}}}
	\put(30,50){\color{blue}{\circle*{5}}}
	\put(50,20){\color{blue}{\circle*{5}}}
	\put(80,40){\color{blue}{\circle*{5}}}
	\put(110,60){\color{blue}{\circle*{5}}}
	\put(10,80){\color{blue}{\circle*{5}}}
	
	\put(100,10){\color{blue}{\circle*{5}}}
	\put(130,30){\color{blue}{\circle*{5}}}
	
	\linethickness{0.4mm}
	\put(10,0){\line(0,1){130}}
	\put(10,130){\line(1,0){130}}
	\put(30,0){\line(0,1){110}}
	\put(30,110){\line(1,0){110}}
\end{picture}}
  \caption{Thirteen corner configurations for (3,1) broadcast domination } \label{delete}
\end{figure}
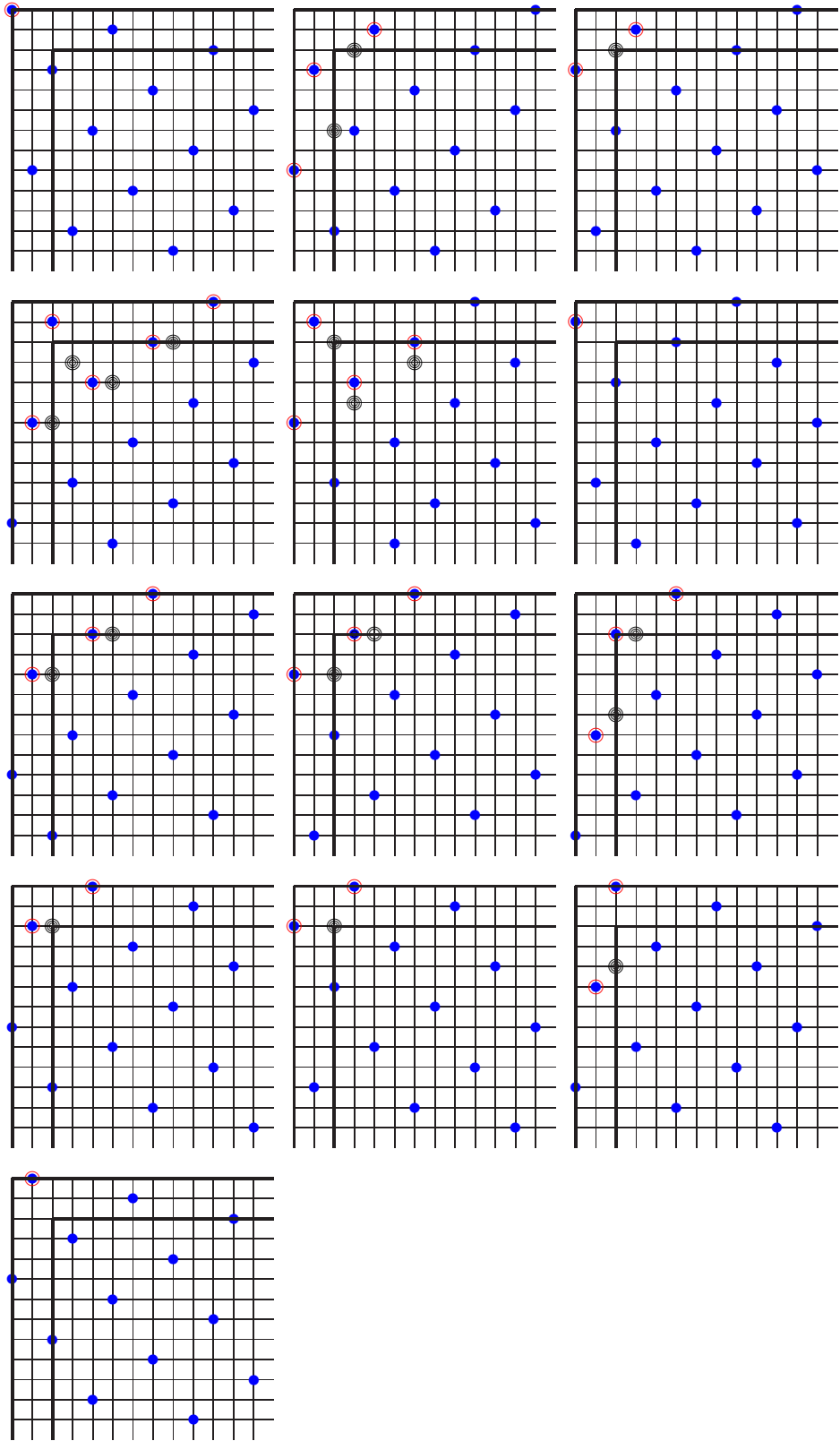

Taking the set obtained in Lemma \ref{lemma:rearrange}, we can obtain a minimal dominating subset $D_{m,n} \subset G_{m,n}$ by moving any vertex in $P(i) \cap Y_{m,n}$ to its nearest neighbor in $G_{m,n}$.
Since we can always remove four vertices from the set $P(i) \cap Y_{m,n}$ to obtain $D_{m,n}$, we wish to find the minimum number of vertices in $P(i) \cap Y_{m,n}$ for $i \in \ZZ_{13}$.  
Using the fact that $Y_{m,n} \cong G_{m+4, n+4}$ the following lemma does precisely that. 
\begin{lemma} \label{lemma:domin}
Let $i \in \ZZ_{13}$ and let  $P(i)= \phi^{-1}(i)$ denote the preimage of $i$ under the homomorphism $\phi: \ZZ \times \ZZ \rightarrow \ZZ_{13}$ defined by $\phi(x,y) = 4x+7y$. 
By the Euclidean algorithm, there exist integers $a$ and $b$ such that $m = 13a + \overline{m}$ and $n= 13b + \overline{n}$, where $\overline{m}$ and $\overline{n}$ are the residues of $m$ and $n$ modulo 13.
Let \[\Phi=\left  \{  \begin{matrix} (2, 7), (3, 9), (4, 4), (4, 7), (4, 10), (6, 7), (6, 9), (6, 11),\\ (7, 2), (7, 4), (7, 6), (9, 3), (9, 6), (9, 9), (10, 4), (11, 6)                   
                     \end{matrix}
\right  \} .\] 
Define \[\delta_\Phi= \begin{cases} 1 &  (\overline{m},\overline{n}) \in \Phi\\ 0 & \text{ else } \end{cases}. \]
Then \[ \min_{0\leq i \leq 12} | G_{m,n} \cap P(i) | =  \left \lfloor \frac{mn}{13}  \right \rfloor - \delta_\Phi \] 
\end{lemma}
\begin{proof}
The set $P(0) = \phi^{-1}(0)$ is the subset $\ZZ \times \ZZ $ consisting of all integer multiples of $(1,5)$ and $(3,2)$. 
To obtain each of the sets $P(i)$ where $0 < i \leq 12$ shift the set $P(0)$ horizontally by $r$ units where $ 0 < r \leq 12$.  
To count the number of vertices in each $G_{m,n} \cap P(i)$ we break the $m \times n$ grid $G_{m,n}$ into four regions: $R_1, R_2, R_3,$ and $R_4$ with dimensions 
$13a \times 13b, 13a \times \overline{n} , \overline{m} \times 13b $ , and $\overline{m} \times \overline{n}$ respectively,
 as depicted in Figure \ref{fourparts}.  

\begin{figure}[h]
\begin{picture}(130,130)(0,0)
	\linethickness{0.2mm}
	\put(100,0){\line(0,1){130}}
	\put(0,30){\line(1,0){130}}
        \put(45,90){$R_1$}
        \put(108,90){$R_2$}
        \put(45,10){$R_3$}
        \put(108,10){$R_4$}
	\linethickness{0.4mm}
	\put(0,0){\line(0,1){130}}
	\put(0,130){\line(1,0){130}}
	\put(130,0){\line(0,1){130}}
	\put(0,0){\line(1,0){130}}
\end{picture}
  \caption{Four regions of $G_{m,n}$} \label{fourparts}
\end{figure}
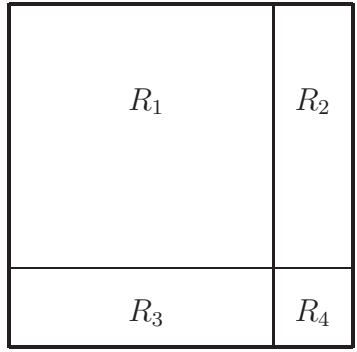

Since $\phi$ is a ring homomorphism from $\ZZ \times \ZZ $ to $\ZZ_{13}$, every thirteen consecutive vertices in a row or column of $G_{m,n} \subset \ZZ \times \ZZ$ 
will contain exactly one element of the preimage $P(i) = \phi^{-1}(i)$ for any $0 \leq i \leq 12$. 
Hence $R_1 \cap P(i)$ contains $\frac{(13a)(13b)}{13}= 13ab$ elements, 
and the $R_2 \cap P(i) $ and $R_3 \cap P(i)$ contain $a\overline{n}$ and $b \overline{m}$ elements of $P(i)$ respectively.

On the other hand, the number of vertices contained in $ R_4 \cap P(i)$ depends on the three quantities $i, \overline{ m},$ and $\overline{n}$.
There are $13^3$ cases to check, and we have written SAGE code to compute all of them.
\footnote{The code is available at the site \\ {  \tt https://cloud.sagemath.com/projects/26b983ae-a894-47c0-bd54-c73b071e555c/files/(3,1)BDC.sagews }}
The calculation shows that \[ |R_4 \cap P(i)| = \left \lfloor \frac{(\overline{m} \overline{n} )}{13}  \right \rfloor - 1 \] for the following values of $\overline{m}$ and $\overline{n}$
\[\Phi=\left  \{ \begin{matrix} (2, 7), (3, 9), (4, 4), (4, 7), (4, 10), (6, 7), (6, 9), (6, 11),\\ (7, 2), (7, 4), (7, 6), (9, 3), (9, 6), (9, 9), (10, 4), (11, 6)                   
                     \end{matrix}
\right  \} .\] 
Otherwise \[ |R_4 \cap P(i) | = \left \lfloor \frac{(\overline{m}\overline{n} )}{13}  \right \rfloor \] for all other values of $\overline{m}$ and $ \overline{n}$.

Hence the minimum number of elements in $P(i)$ as $i$ varies from $0$ to $12$ is 
\begin{align*} \min_{0 \leq i \leq 12} |P(i)| & = \min_{0 \leq i \leq 12} |R_1 \cap P(i) |+ \min_{0 \leq i \leq 12} |R_2 \cap P(i) | + \min_{0 \leq i \leq 12} 
|R_3 \cap P(i)|+ \min_{0 \leq i \leq 12} |R_4 \cap P(i)| \\
          & = 13ab+ a\overline{n} + b \overline{m} +  \left \lfloor \frac{(\overline{m}\overline{n})}{13}  \right \rfloor - \delta_\Phi   \\
          & = \left \lfloor \frac{( 13a+\overline{m})(13b + \overline{n} )}{13}  \right \rfloor - \delta_\Phi  \\ & = \left \lfloor \frac{(mn)}{13}  \right \rfloor - \delta_\Phi \end{align*}
\end{proof}

Using the construction of $D_{m,n}$ and the cardinality count in Lemma \ref{lemma:domin} we can obtain the upper bound on the $(3,1)$  broadcast domination number of $G_{m,n}$.  
\begin{theorem}
 For any sufficiently large $m$ and $n$, the $(3,1)$ broadcast domination number satisfies \[ \gamma_{3,1}(G_{m,n}) \leq \left \lfloor \frac{(m+4)(n+4)}{13} \right \rfloor - 5 \]
when $(\overline{m+4},\overline{n+4})$ are in $\Phi$ where
\[\Phi=\left  \{  \begin{matrix} (2, 7), (3, 9), (4, 4), (4, 7), (4, 10), (6, 7), (6, 9), (6, 11),\\ (7, 2), (7, 4), (7, 6), (9, 3), (9, 6), (9, 9), (10, 4), (11, 6)                   
                    \end{matrix} \right  \}, \] 
and 
 \[ \gamma_{3,1}(G_{m,n}) \leq \left \lfloor \frac{(m+4)(n+4)}{13} \right \rfloor - 4 \]
otherwise.
\end{theorem}
\begin{proof}
Lemma \ref{lemma:domin} states that for each $(m,n) \in \ZZ \times \ZZ$ there exists an $i' \in \ZZ_{13}$ such that 
\[ |Y_{m,n} \cap P(i')| =  \left \lfloor \frac{(m+4)(n+4)}{13} \right \rfloor - \delta_{\Phi} .\] 
The dominating set $D_{m,n}$ was constructed so that $|D_{m,n}|= |Y_{m,n} \cap P(i') |-4$.  \end{proof}

\subsection{Upper Bound for (3,2) Broadcast Domination}
In this section we describe a dominating set that gives an upper bound for $(3,2)$  broadcast domination numbers of $m \times n$ grids.  We then count the number of elements in this set.
This gives an upper bound on the $\gamma_{3,2}(G_{m,n})$ which we conjecture is tight for sufficiently large $m$ and $n$.  

Define a function $\phi: \ZZ \times \ZZ \rightarrow \ZZ_ 8 $ by $\phi(x,y) = x+ 3y$.  For each $i \in \ZZ_8$, the inverse image $P(i) = \phi^{-1}(i)$ is an optimal dominating set 
of $\ZZ \times \ZZ$ in the sense that every element of $\ZZ \times \ZZ$ is dominated with a weight of $3$ or $2$.  In fact, every 8 vertices in a row or column will have weights 
3-2-3-2-3-2-3-2.  The next most efficient dominating set produces weights 3-2-4-2-3-2-4-2.  
Figure \ref{32P} below shows $P(0)$.

\begin{figure}[h]
\begin{picture}(210,200)(100,0)
	\linethickness{.1 mm}
	\multiput(110,10)(0,10){19}{\line(1,0){200}}
	\multiput(120,0)(10,0){19}{\line(0,1){200}}
        \put(210,100){\color{blue}{\circle*{5}}}
        \put(230,120){\color{blue}{\circle*{5}}}
        \put(190,80){\color{blue}{\circle*{5}}}
        \put(170,60){\color{blue}{\circle*{5}}}
        \put(150,40){\color{blue}{\circle*{5}}}
        \put(250,140){\color{blue}{\circle*{5}}}
        \put(270,160){\color{blue}{\circle*{5}}}
        \put(290,180){\color{blue}{\circle*{5}}}
        \put(130,20){\color{blue}{\circle*{5}}}

        \put(200,130){\color{blue}{\circle*{5}}}
        \put(220,150){\color{blue}{\circle*{5}}}
        \put(180,110){\color{blue}{\circle*{5}}}
        \put(160,90){\color{blue}{\circle*{5}}}
        \put(140,70){\color{blue}{\circle*{5}}}
        \put(240,170){\color{blue}{\circle*{5}}}
        \put(260,190){\color{blue}{\circle*{5}}}
        \put(120,50){\color{blue}{\circle*{5}}}

        \put(190,160){\color{blue}{\circle*{5}}}
        \put(210,180){\color{blue}{\circle*{5}}}
        \put(170,140){\color{blue}{\circle*{5}}}
        \put(150,120){\color{blue}{\circle*{5}}}
        \put(130,100){\color{blue}{\circle*{5}}}
       
        \put(180,190){\color{blue}{\circle*{5}}}
        \put(160,170){\color{blue}{\circle*{5}}}
        \put(140,150){\color{blue}{\circle*{5}}}
        \put(120,130){\color{blue}{\circle*{5}}}

        \put(130,180){\color{blue}{\circle*{5}}}

        \put(220,70){\color{blue}{\circle*{5}}}
        \put(240,90){\color{blue}{\circle*{5}}}
        \put(200,50){\color{blue}{\circle*{5}}}
        \put(180,30){\color{blue}{\circle*{5}}}
        \put(160,10){\color{blue}{\circle*{5}}}
        \put(260,110){\color{blue}{\circle*{5}}}
        \put(280,130){\color{blue}{\circle*{5}}}
        \put(300,150){\color{blue}{\circle*{5}}}

        \put(230,40){\color{blue}{\circle*{5}}}
        \put(250,60){\color{blue}{\circle*{5}}}
        \put(210,20){\color{blue}{\circle*{5}}}
        \put(270,80){\color{blue}{\circle*{5}}}
        \put(290,100){\color{blue}{\circle*{5}}}
        
        \put(240,10){\color{blue}{\circle*{5}}}
        \put(260,30){\color{blue}{\circle*{5}}}
        \put(280,50){\color{blue}{\circle*{5}}}
        \put(300,70){\color{blue}{\circle*{5}}}
         
        \put(290,20){\color{blue}{\circle*{5}}}

	\linethickness{0.4mm}
	\put(210,0){\line(0,1){200}}
	\put(110,100){\line(1,0){200}}
\end{picture}
\caption{ The optimal dominating set $P(0)$ of $\ZZ \times \ZZ$  } \label{32P}
\end{figure}
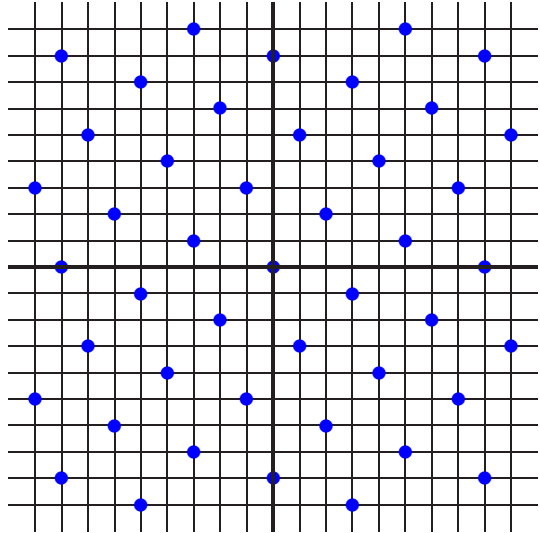

Embed $G_{m,n}$ into $\ZZ \times \ZZ$ as the following set: \[ G_{m,n} = \{ (a,b) \in \ZZ \times \ZZ  \mid 1 \leq a \leq n  \text{ and } 1 \leq b \leq m \}, \]
and let $Y_{m,n} \cong G_{m+2,n+2} $ denote the neighborhood of $G_{m,n}$: \[Y_{m,n} = \{ (a,b) \in \ZZ \times \ZZ \mid 0 \leq a \leq n+1 \text{ and }  0 \leq b \leq m+1 \}. \]

Unlike our previously studied examples of $(\ww, \tw)$  broadcast domination theories, we cannot always remove vertices from the corners of $P(i) \cap Y_{m,n}$ and still dominate $G_{m,n}$.  
There are eight possible configurations of $P(i) \cap Y_{m,n}$.  However, only three 
of these configurations allow us to remove a vertex from the set $P(i) \cap Y_{m,n}$ and still dominate the grid $G_{m,n}$.
Figures \ref{SWNE} and \ref{NWSE} show each of the eight configurations.  
When possible, they also show how to remove one vertex by replacing the vertices circled in red with the vertices decorated by a bullseye.

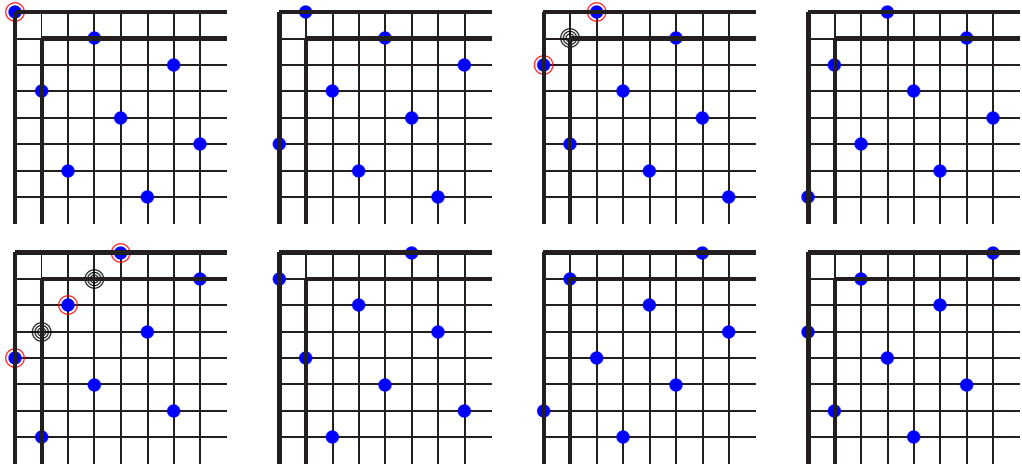
\begin{figure}[H]
\begin{picture}(330,90)(25,0)
	\linethickness{.1 mm}
	\multiput(10,10)(0,10){8}{\line(1,0){80}}
	\multiput(10,0)(10,0){8}{\line(0,1){80}}
	\put(10,80){\color{blue}{\circle*{5}}}
        \put(10,80){\color{red}{\circle{7}}}
	\put(20,50){\color{blue}{\circle*{5}}}
        \put(40,70){\color{blue}{\circle*{5}}}
	\put(30,20){\color{blue}{\circle*{5}}}
	\put(50,40){\color{blue}{\circle*{5}}}
	\put(60,10){\color{blue}{\circle*{5}}}
	\put(80,30){\color{blue}{\circle*{5}}}
        \put(70,60){\color{blue}{\circle*{5}}}
	\linethickness{0.4mm}
	\put(10,0){\line(0,1){80}}
	\put(10,80){\line(1,0){80}}
	\put(20,0){\line(0,1){70}}
	\put(20,70){\line(1,0){70}}

	\linethickness{.1 mm}
	\multiput(110,10)(0,10){8}{\line(1,0){80}}
	\multiput(110,0)(10,0){8}{\line(0,1){80}}
	\put(120,80){\color{blue}{\circle*{5}}}
	\put(130,50){\color{blue}{\circle*{5}}}
        \put(150,70){\color{blue}{\circle*{5}}}
	\put(140,20){\color{blue}{\circle*{5}}}
	\put(160,40){\color{blue}{\circle*{5}}}
	\put(170,10){\color{blue}{\circle*{5}}}
	\put(110,30){\color{blue}{\circle*{5}}}
        \put(180,60){\color{blue}{\circle*{5}}}
	\linethickness{0.4mm}
	\put(110,0){\line(0,1){80}}
	\put(110,80){\line(1,0){80}}
	\put(120,0){\line(0,1){70}}
	\put(120,70){\line(1,0){70}}

	\linethickness{.1 mm}
	\multiput(210,10)(0,10){8}{\line(1,0){80}}
	\multiput(210,0)(10,0){8}{\line(0,1){80}}
	\put(230,80){\color{blue}{\circle*{5}}}
	\put(240,50){\color{blue}{\circle*{5}}}
        \put(260,70){\color{blue}{\circle*{5}}}
	\put(250,20){\color{blue}{\circle*{5}}}
	\put(270,40){\color{blue}{\circle*{5}}}
	\put(280,10){\color{blue}{\circle*{5}}}
	\put(220,30){\color{blue}{\circle*{5}}}
        \put(220,70){\color{black}{\circle{5}}}
        \put(220,70){\color{black}{\circle{3}}}
        \put(220,70){\color{black}{\circle{7}}}
        \put(210,60){\color{blue}{\circle*{5}}}
        \put(210,60){\color{red}{\circle{7}}}
	\put(230,80){\color{red}{\circle{7}}}
	\linethickness{0.4mm}
	\put(210,0){\line(0,1){80}}
	\put(210,80){\line(1,0){80}}
	\put(220,0){\line(0,1){70}}
	\put(220,70){\line(1,0){70}}

	\linethickness{.1 mm}
	\multiput(310,10)(0,10){8}{\line(1,0){80}}
	\multiput(310,0)(10,0){8}{\line(0,1){80}}
	\put(340,80){\color{blue}{\circle*{5}}}
	\put(350,50){\color{blue}{\circle*{5}}}
        \put(370,70){\color{blue}{\circle*{5}}}
	\put(360,20){\color{blue}{\circle*{5}}}
	\put(380,40){\color{blue}{\circle*{5}}}
	\put(310,10){\color{blue}{\circle*{5}}}
	\put(330,30){\color{blue}{\circle*{5}}}
        \put(320,60){\color{blue}{\circle*{5}}}
	\linethickness{0.4mm}
	\put(310,0){\line(0,1){80}}
	\put(310,80){\line(1,0){80}}
	\put(320,0){\line(0,1){70}}
	\put(320,70){\line(1,0){70}}

\end{picture}

\begin{picture}(330,90)(25,0)
	\linethickness{.1 mm}
	\multiput(10,10)(0,10){8}{\line(1,0){80}}
	\multiput(10,0)(10,0){8}{\line(0,1){80}}
	\put(50,80){\color{blue}{\circle*{5}}}
        \put(50,80){\color{red}{\circle{7}}}
	\put(60,50){\color{blue}{\circle*{5}}}
        \put(80,70){\color{blue}{\circle*{5}}}
	\put(70,20){\color{blue}{\circle*{5}}}
	\put(10,40){\color{blue}{\circle*{5}}}
	\put(10,40){\color{red}{\circle{7}}}
	\put(20,10){\color{blue}{\circle*{5}}}
	\put(40,30){\color{blue}{\circle*{5}}}
        \put(30,60){\color{blue}{\circle*{5}}}
        \put(30,60){\color{red}{\circle{7}}}
        \put(40,70){\color{black}{\circle{7}}}
        \put(40,70){\color{black}{\circle{5}}}
        \put(40,70){\color{black}{\circle{3}}}
        \put(20,50){\color{black}{\circle{7}}}
        \put(20,50){\color{black}{\circle{5}}}
        \put(20,50){\color{black}{\circle{3}}}
	\linethickness{0.4mm}
	\put(10,0){\line(0,1){80}}
	\put(10,80){\line(1,0){80}}
	\put(20,0){\line(0,1){70}}
	\put(20,70){\line(1,0){70}}

	\linethickness{.1 mm}
	\multiput(110,10)(0,10){8}{\line(1,0){80}}
	\multiput(110,0)(10,0){8}{\line(0,1){80}}
	\put(160,80){\color{blue}{\circle*{5}}}
	\put(170,50){\color{blue}{\circle*{5}}}
        \put(110,70){\color{blue}{\circle*{5}}}
	\put(180,20){\color{blue}{\circle*{5}}}
	\put(120,40){\color{blue}{\circle*{5}}}
	\put(130,10){\color{blue}{\circle*{5}}}
	\put(150,30){\color{blue}{\circle*{5}}}
        \put(140,60){\color{blue}{\circle*{5}}}
	\linethickness{0.4mm}
	\put(110,0){\line(0,1){80}}
	\put(110,80){\line(1,0){80}}
	\put(120,0){\line(0,1){70}}
	\put(120,70){\line(1,0){70}}

	\linethickness{.1 mm}
	\multiput(210,10)(0,10){8}{\line(1,0){80}}
	\multiput(210,0)(10,0){8}{\line(0,1){80}}
	\put(270,80){\color{blue}{\circle*{5}}}
	\put(280,50){\color{blue}{\circle*{5}}}
        \put(220,70){\color{blue}{\circle*{5}}}
	\put(210,20){\color{blue}{\circle*{5}}}
	\put(230,40){\color{blue}{\circle*{5}}}
	\put(240,10){\color{blue}{\circle*{5}}}
	\put(260,30){\color{blue}{\circle*{5}}}
        \put(250,60){\color{blue}{\circle*{5}}}
	\linethickness{0.4mm}
	\put(210,0){\line(0,1){80}}
	\put(210,80){\line(1,0){80}}
	\put(220,0){\line(0,1){70}}
	\put(220,70){\line(1,0){70}}

	\linethickness{.1 mm}
	\multiput(310,10)(0,10){8}{\line(1,0){80}}
	\multiput(310,0)(10,0){8}{\line(0,1){80}}
	\put(380,80){\color{blue}{\circle*{5}}}
	\put(310,50){\color{blue}{\circle*{5}}}
        \put(330,70){\color{blue}{\circle*{5}}}
	\put(320,20){\color{blue}{\circle*{5}}}
	\put(340,40){\color{blue}{\circle*{5}}}
	\put(350,10){\color{blue}{\circle*{5}}}
	\put(370,30){\color{blue}{\circle*{5}}}
        \put(360,60){\color{blue}{\circle*{5}}}
	\linethickness{0.4mm}
	\put(310,0){\line(0,1){80}}
	\put(310,80){\line(1,0){80}}
	\put(320,0){\line(0,1){70}}
	\put(320,70){\line(1,0){70}}

\end{picture}
  \caption{Cases for the SW and NE Corner} \label{SWNE}
\end{figure}

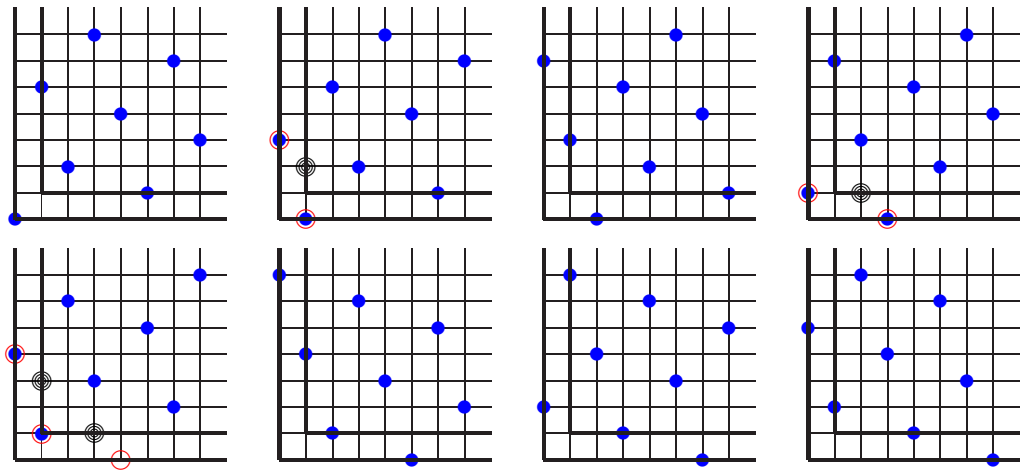
\begin{figure}[H]

\begin{picture}(330,90)(25,0)
	\linethickness{.1 mm}
	\multiput(10,0)(0,10){8}{\line(1,0){80}}
	\multiput(10,0)(10,0){8}{\line(0,1){80}}
	\put(10,0){\color{blue}{\circle*{5}}}
	\put(20,50){\color{blue}{\circle*{5}}}
        \put(40,70){\color{blue}{\circle*{5}}}
	\put(30,20){\color{blue}{\circle*{5}}}
	\put(50,40){\color{blue}{\circle*{5}}}
	\put(60,10){\color{blue}{\circle*{5}}}
	\put(80,30){\color{blue}{\circle*{5}}}
        \put(70,60){\color{blue}{\circle*{5}}}
	\linethickness{0.4mm}
	\put(10,0){\line(0,1){80}}
	\put(10,0){\line(1,0){80}}
	\put(20,10){\line(0,1){70}}
	\put(20,10){\line(1,0){70}}

	\linethickness{.1 mm}
	\multiput(110,0)(0,10){8}{\line(1,0){80}}
	\multiput(110,0)(10,0){8}{\line(0,1){80}}
	\put(120,0){\color{blue}{\circle*{5}}}
        \put(120,0){\color{red}{\circle{7}}}
        \put(120,20){\color{black}{\circle{7}}}
        \put(120,20){\color{black}{\circle{3}}}
        \put(120,20){\color{black}{\circle{5}}}
	\put(130,50){\color{blue}{\circle*{5}}}
        \put(150,70){\color{blue}{\circle*{5}}}
	\put(140,20){\color{blue}{\circle*{5}}}
	\put(160,40){\color{blue}{\circle*{5}}}
	\put(170,10){\color{blue}{\circle*{5}}}
	\put(110,30){\color{blue}{\circle*{5}}}
        \put(110,30){\color{red}{\circle{7}}}
        \put(180,60){\color{blue}{\circle*{5}}}
	\linethickness{0.4mm}
	\put(110,0){\line(0,1){80}}
	\put(110,0){\line(1,0){80}}
	\put(120,10){\line(0,1){70}}
	\put(120,10){\line(1,0){70}}

	\linethickness{.1 mm}
	\multiput(210,0)(0,10){8}{\line(1,0){80}}
	\multiput(210,0)(10,0){8}{\line(0,1){80}}
	\put(230,0){\color{blue}{\circle*{5}}}
	\put(240,50){\color{blue}{\circle*{5}}}
        \put(260,70){\color{blue}{\circle*{5}}}
	\put(250,20){\color{blue}{\circle*{5}}}
	\put(270,40){\color{blue}{\circle*{5}}}
	\put(280,10){\color{blue}{\circle*{5}}}
	\put(220,30){\color{blue}{\circle*{5}}}
        \put(210,60){\color{blue}{\circle*{5}}}
	\linethickness{0.4mm}
	\put(210,0){\line(0,1){80}}
	\put(210,0){\line(1,0){80}}
	\put(220,10){\line(0,1){70}}
	\put(220,10){\line(1,0){70}}

	\linethickness{.1 mm}
	\multiput(310,0)(0,10){8}{\line(1,0){80}}
	\multiput(310,0)(10,0){8}{\line(0,1){80}}
	\put(340,0){\color{blue}{\circle*{5}}}
        \put(340,0){\color{red}{\circle{7}}}
	\put(350,50){\color{blue}{\circle*{5}}}
        \put(370,70){\color{blue}{\circle*{5}}}
	\put(360,20){\color{blue}{\circle*{5}}}
	\put(380,40){\color{blue}{\circle*{5}}}
	\put(310,10){\color{blue}{\circle*{5}}}
        \put(310,10){\color{red}{\circle{7}}}
        \put(330,10){\color{black}{\circle{7}}}
	\put(330,10){\color{black}{\circle{5}}}
	\put(330,10){\color{black}{\circle{3}}}
	\put(330,30){\color{blue}{\circle*{5}}}
        \put(320,60){\color{blue}{\circle*{5}}}
	\linethickness{0.4mm}
	\put(310,0){\line(0,1){80}}
	\put(310,0){\line(1,0){80}}
	\put(320,10){\line(0,1){70}}
	\put(320,10){\line(1,0){70}}

\end{picture}

\begin{picture}(330,90)(25,0)
	\linethickness{.1 mm}
	\multiput(10,0)(0,10){8}{\line(1,0){80}}
	\multiput(10,0)(10,0){8}{\line(0,1){80}}
        \put(50,0){\color{red}{\circle{7}}}
	\put(60,50){\color{blue}{\circle*{5}}}
        \put(80,70){\color{blue}{\circle*{5}}}
	\put(70,20){\color{blue}{\circle*{5}}}
	\put(10,40){\color{blue}{\circle*{5}}}
        \put(10,40){\color{red}{\circle{7}}}
	\put(20,10){\color{blue}{\circle*{5}}}
        \put(20,10){\color{red}{\circle{7}}}
        \put(20,30){\color{black}{\circle{7}}}
        \put(20,30){\color{black}{\circle{5}}}
        \put(20,30){\color{black}{\circle{3}}}
        \put(40,10){\color{black}{\circle{3}}}
        \put(40,10){\color{black}{\circle{7}}}
        \put(40,10){\color{black}{\circle{5}}}
        
	\put(40,30){\color{blue}{\circle*{5}}}
        \put(30,60){\color{blue}{\circle*{5}}}
	\linethickness{0.4mm}
	\put(10,0){\line(0,1){80}}
	\put(10,0){\line(1,0){80}}
	\put(20,10){\line(0,1){70}}
	\put(20,10){\line(1,0){70}}

	\linethickness{.1 mm}
	\multiput(110,0)(0,10){8}{\line(1,0){80}}
	\multiput(110,0)(10,0){8}{\line(0,1){80}}
	\put(160,0){\color{blue}{\circle*{5}}}
	\put(170,50){\color{blue}{\circle*{5}}}
        \put(110,70){\color{blue}{\circle*{5}}}
	\put(180,20){\color{blue}{\circle*{5}}}
	\put(120,40){\color{blue}{\circle*{5}}}
	\put(130,10){\color{blue}{\circle*{5}}}
	\put(150,30){\color{blue}{\circle*{5}}}
        \put(140,60){\color{blue}{\circle*{5}}}
	\linethickness{0.4mm}
	\put(110,0){\line(0,1){80}}
	\put(110,0){\line(1,0){80}}
	\put(120,10){\line(0,1){70}}
	\put(120,10){\line(1,0){70}}

	\linethickness{.1 mm}
	\multiput(210,0)(0,10){8}{\line(1,0){80}}
	\multiput(210,0)(10,0){8}{\line(0,1){80}}
	\put(270,0){\color{blue}{\circle*{5}}}
	\put(280,50){\color{blue}{\circle*{5}}}
        \put(220,70){\color{blue}{\circle*{5}}}
	\put(210,20){\color{blue}{\circle*{5}}}
	\put(230,40){\color{blue}{\circle*{5}}}
	\put(240,10){\color{blue}{\circle*{5}}}
	\put(260,30){\color{blue}{\circle*{5}}}
        \put(250,60){\color{blue}{\circle*{5}}}
	\linethickness{0.4mm}
	\put(210,0){\line(0,1){80}}
	\put(210,0){\line(1,0){80}}
	\put(220,10){\line(0,1){70}}
	\put(220,10){\line(1,0){70}}

	\linethickness{.1 mm}
	\multiput(310,0)(0,10){8}{\line(1,0){80}}
	\multiput(310,0)(10,0){8}{\line(0,1){80}}
	\put(380,0){\color{blue}{\circle*{5}}}
	\put(310,50){\color{blue}{\circle*{5}}}
        \put(330,70){\color{blue}{\circle*{5}}}
	\put(320,20){\color{blue}{\circle*{5}}}
	\put(340,40){\color{blue}{\circle*{5}}}
	\put(350,10){\color{blue}{\circle*{5}}}
	\put(370,30){\color{blue}{\circle*{5}}}
        \put(360,60){\color{blue}{\circle*{5}}}
	\linethickness{0.4mm}
	\put(310,0){\line(0,1){80}}
	\put(310,0){\line(1,0){80}}
	\put(320,10){\line(0,1){70}}
	\put(320,10){\line(1,0){70}}

\end{picture}
  \caption{8 Cases for the NW and SE Corner} \label{NWSE}
\end{figure}

The following theorem gives an upper bound on the $(3,2)$ broadcast domination number of $G_{m,n}$.  
\begin{theorem} \label{32dominationbound}
 The $(3,2)$ broadcast domination number $\gamma_{3,2}(G_{m,n})$ of an $m \times n$ grid is bounded above by the following formula:

\[ \gamma_{3,2}(G_{m,n}) \leq \left \lfloor \frac{(m+2)(n+2)}{8} \right \rfloor - c_{\overline{m},\overline{n}} \]
where the constant $c_{\overline{m},\overline{n}}$ is defined by 
\[ c_{\overline{m},\overline{n}} = \begin{cases}1 &  \text{ if } (\overline{m}, \overline{n} ) \in A \\ 2 & \text{ if }  (\overline{m}, \overline{n} ) \in B \\ 3 & \text{ if }  (\overline{m}, \overline{n} ) \in C \end{cases} \] 
where the set $B$ is 
\[ B = \left \{  (0, 0), (0, 4), (0, 6), (2, 6), (4, 4), (4, 6), (4, 0), (6, 6), (6, 0), (6, 2), (6, 4) \right \} \]
the set $C$ is 
\[ C = \left \{  \right (0, 2), (2, 2), (2, 4), (2, 0), (4, 2) \} \]
and the set $A$ is all other possible residues $\overline{m} \equiv m \mod 8$ and $\overline{n} \equiv n \mod 8$.    
\end{theorem}

\begin{proof}
To find the minimum cardinality of any dominating set $D_{m,n}$, we use an algorithm written in SAGE
that computes the cardinality of $P(i) \cap Y_{ m,n}$.\footnote{The code is available at the site \\{
 \tt   https://cloud.sagemath.com/projects/26b983ae-a894-47c0-bd54-c73b071e555c/files/(3,2)BDC.sagews} }  As in the proofs of Lemmas \ref{lemma:stones} and \ref{lemma:domin} the cardinality of
 this set is determined by the values of the residues $\overline{m}, \overline{n}$, and $i \in \ZZ_8$.

We then note that when $m_1 \equiv m_2 \mod 8 $ and $n_1 \equiv n_2 \mod 8$ the corner configurations of the sets $P(i) \cap Y_{m_1, n_1}$ and $P(i) \cap Y_{m_2, n_2}$ are the same.  
Hence it suffices to check how many vertices can be deleted from each set $P(i) \cap Y_{ \overline{m}, \overline{n} }$ where $i \in \ZZ_8$ and where $\overline{m}$ and $\overline{n}$ 
are the residues of $m$ and $n \mod 8$ respectively.  
Figure \ref{SWNE} shows that we may delete a vertex from the Northeast or Southwest corners when there are two vertices in $P(i) \cap (Y_{m,n}-G_{m,n})$ 
within distance $4$ of that corner of $Y_{m,n}$
or when a vertex in $P(i) \cap (Y_{m,n}-G_{m,n}) $ is located at the corner of $Y_{m,n}$.
Figure   \ref{NWSE} illustrates that we are able to delete a vertex from the Northwest or Southeast corners of $Y_{m,n}$ when there are two vertices in $P(i) \cap (Y_{m,n}-G_{m,n})$ 
within distance $4$ of the corner of $Y_{m,n}$.

We have written code to count the number of corners that satisfy one of the three configurations in Figures \ref{SWNE} and \ref{NWSE}. 
We let $d(\overline{m},\overline{n},i) $ denote the number of such corners.  
Calculating the values $|P(i) \cap Y_{m,n}|- d(\overline{m},\overline{n},i)$ for each of the $8^3$ different 
cases takes only seconds on a personal computer, and the minimum of these values of $|P(i) \cap Y_{m,n}|- d(\overline{m},\overline{n},i)$  for each $\overline{m}$ and $\overline{n}$ are the ones recorded in the statement of the theorem.  
\end{proof}

\subsection{$(3,3)$  Broadcast Domination}
The optimal dominating set for $(3,3)$  broadcast domination is the preimage of the homomorphism $\phi: \ZZ \times \ZZ \rightarrow \ZZ_5$ with $\phi(x,y) = x+3y$.  
For $ i \in \ZZ_5$, let $P(i) = \phi^{-1}(i)$ denote its preimage.  
These sets $P(i)$ are precisely the same dominating sets as the regular dominating sets of $\ZZ \times \ZZ$ described in Chang's Ph.D. thesis \cite{Cha92}.
We shall see that Chang's bound for regular domination numbers also gives an upper bound for (3,3) broadcast domination numbers.  

\begin{figure}[H]
\begin{picture}(410,90)(-40,0)

	\linethickness{.1 mm}
	\multiput(130,10)(0,10){9}{\line(1,0){100}}
	\multiput(140,0)(10,0){9}{\line(0,1){100}}
	\put(180,50){\color{blue}{\circle*{5}}}
	\put(200,60){\color{blue}{\circle*{5}}}
	\put(220,70){\color{blue}{\circle*{5}}}
	\put(160,40){\color{blue}{\circle*{5}}}
	\put(140,30){\color{blue}{\circle*{5}}}
	\put(190,30){\color{blue}{\circle*{5}}}
	\put(170,20){\color{blue}{\circle*{5}}}
	\put(150,10){\color{blue}{\circle*{5}}}
	\put(210,40){\color{blue}{\circle*{5}}}
	\put(170,70){\color{blue}{\circle*{5}}}
	\put(150,60){\color{blue}{\circle*{5}}}
	\put(140,80){\color{blue}{\circle*{5}}}	
	\put(160,90){\color{blue}{\circle*{5}}}	
	\put(190,80){\color{blue}{\circle*{5}}}
	\put(210,90){\color{blue}{\circle*{5}}}
	\put(200,10){\color{blue}{\circle*{5}}}
	\put(220,20){\color{blue}{\circle*{5}}}
	\linethickness{0.4mm}
	\put(180,0){\line(0,1){100}}
	\put(130,50){\line(1,0){100}}

\end{picture}
\caption{ The dominating sets for (3,3) and (2,1) broadcast domination } \label{33P}
\end{figure}
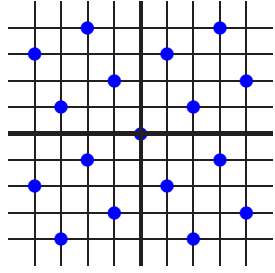 

Embed $G_{m,n}$ into $\ZZ \times \ZZ$ as the following set: \[ G_{m,n} = \{ (a,b) \in \ZZ \times \ZZ  \mid 1 \leq a \leq n  \text{ and } 1 \leq b \leq m \}, \]
and let $Y_{m,n} \cong G_{m+2,n+2} $ denote the neighborhood of $G_{m,n}$: \[Y_{m,n} = \{ (a,b) \in \ZZ \times \ZZ \mid 0 \leq a \leq n+1 \text{ and }  0 \leq b \leq m+1 \}. \]
Chang showed that there are $\left \lfloor \frac{(m+2)(n+2)}{5} \right \rfloor$ vertices in each $P(i) \cap Y_{m,n}$.  
Figure \ref{(3,3)Corners} shows that we can delete one vertex from each corner of $P(i) \cap (Y_{m,n}-G_{m,n})$ and still dominate $G_{m,n}$.  

\begin{figure}[H]
\begin{picture}(330,50)(80,30)
	\linethickness{.1 mm}
	\multiput(10,40)(0,10){5}{\line(1,0){50}}
	\multiput(10,30)(10,0){5}{\line(0,1){50}}
	\put(10,80){\color{blue}{\circle*{5}}}
	\put(10,80){\color{red}{\circle{7}}}
	
	\put(40,70){\color{blue}{\circle*{5}}}
	\put(20,60){\color{blue}{\circle*{5}}}
	\put(30,40){\color{blue}{\circle*{5}}}
         \put(50,50){\color{blue}{\circle*{5}}}
	\linethickness{0.4mm}
	\put(10,30){\line(0,1){50}}
	\put(10,80){\line(1,0){50}}
	\put(20,30){\line(0,1){40}}
	\put(20,70){\line(1,0){40}}

	\linethickness{.1 mm}
	\multiput(110,40)(0,10){5}{\line(1,0){50}}
	\multiput(110,30)(10,0){5}{\line(0,1){50}}
         \put(120,80){\color{red}{\circle{7}}}
	\put(120,80){\color{blue}{\circle*{5}}}
	\put(120,60){\color{black}{\circle{3}}}
         \put(120,60){\color{black}{\circle{5}}}
         \put(120,60){\color{black}{\circle{7}}}
	\put(150,70){\color{blue}{\circle*{5}}}
	\put(130,60){\color{blue}{\circle*{5}}}
	\put(140,40){\color{blue}{\circle*{5}}}
         \put(110,50){\color{blue}{\circle*{5}}}
         \put(110,50){\color{red}{\circle{7}}}

	\linethickness{0.4mm}
	\put(110,30){\line(0,1){50}}
	\put(110,80){\line(1,0){50}}
	\put(120,30){\line(0,1){40}}
	\put(120,70){\line(1,0){40}}
	\linethickness{.1 mm}
	\multiput(210,40)(0,10){5}{\line(1,0){50}}
	\multiput(210,30)(10,0){5}{\line(0,1){50}}
	\put(230,80){\color{blue}{\circle*{5}}}
	\put(210,70){\color{blue}{\circle*{5}}}
	\put(210,70){\color{red}{\circle{7}}}
	\put(240,60){\color{blue}{\circle*{5}}}
	\put(250,40){\color{blue}{\circle*{5}}}
         \put(220,50){\color{blue}{\circle*{5}}}
         \put(230,80){\color{red}{\circle{7}}}
         	\put(220,70){\color{black}{\circle{3}}}
         \put(220,70){\color{black}{\circle{5}}}
         \put(220,70){\color{black}{\circle{7}}}
	\linethickness{0.4mm}
	\put(210,30){\line(0,1){50}}
	\put(210,80){\line(1,0){50}}
	\put(220,30){\line(0,1){40}}
	\put(220,70){\line(1,0){40}}
	\linethickness{.1 mm}
	\multiput(310,40)(0,10){5}{\line(1,0){50}}
	\multiput(310,30)(10,0){5}{\line(0,1){50}}
	\put(340,80){\color{blue}{\circle*{5}}}
	\put(340,80){\color{red}{\circle{7}}}
	\put(320,70){\color{blue}{\circle*{5}}}
	\put(320,70){\color{red}{\circle{7}}}
        \put(330,70){\color{black}{\circle{7}}}
        \put(330,70){\color{black}{\circle{3}}}
        \put(330,70){\color{black}{\circle{5}}}
        \put(320,60){\color{black}{\circle{7}}}
        \put(320,60){\color{black}{\circle{3}}}
        \put(320,60){\color{black}{\circle{5}}}
	\put(350,60){\color{blue}{\circle*{5}}}
	\put(310,40){\color{blue}{\circle*{5}}}
        \put(310,40){\color{red}{\circle{7}}}
         \put(330,50){\color{blue}{\circle*{5}}}
	\linethickness{0.4mm}
	\put(310,30){\line(0,1){50}}
	\put(310,80){\line(1,0){50}}
	\put(320,30){\line(0,1){40}}
	\put(320,70){\line(1,0){40}}
	\linethickness{.1 mm}
	\multiput(410,40)(0,10){5}{\line(1,0){50}}
	\multiput(410,30)(10,0){5}{\line(0,1){50}}
	\put(450,80){\color{blue}{\circle*{5}}}
	\put(450,80){\color{red}{\circle{7}}}
	\put(430,70){\color{blue}{\circle*{5}}}
        \put(430,70){\color{red}{\circle{7}}}
	\put(410,60){\color{blue}{\circle*{5}}}
        \put(410,60){\color{red}{\circle{7}}}
	\put(420,40){\color{blue}{\circle*{5}}}
         \put(440,50){\color{blue}{\circle*{5}}}

        \put(440,70){\color{black}{\circle{7}}}
        \put(440,70){\color{black}{\circle{3}}}
        \put(440,70){\color{black}{\circle{5}}}
        \put(420,60){\color{black}{\circle{7}}}
        \put(420,60){\color{black}{\circle{3}}}
        \put(420,60){\color{black}{\circle{5}}}

	\linethickness{0.4mm}
	\put(410,30){\line(0,1){50}}
	\put(410,80){\line(1,0){50}}
	\put(420,30){\line(0,1){40}}
	\put(420,70){\line(1,0){40}}
\end{picture}
 \caption{5 Cases for (3,3) and (2,1) Broadcast domination } \label{(3,3)Corners}
\end{figure}
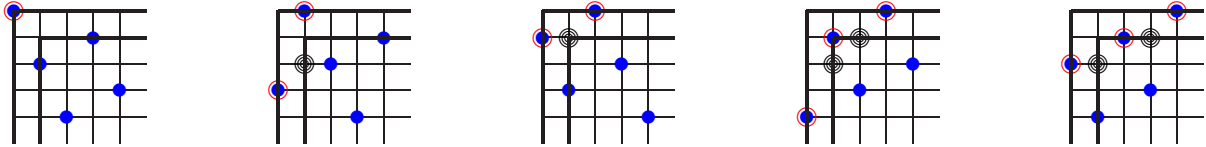

Thus we see that the optimal $(3,3)$ broadcast dominating set of $G_{m,n}$ is bounded above by \[ \gamma_{3,3}(G_{m,n}) \leq \left \lfloor \frac{(m+2)(n+2)}{5} \right \rfloor - 4 .\]
It then follows from the proof of Chang's conjecture given by Goncalves, Pinlou, Rao, and Thomass\'e \cite[Theorem 11]{GonPinRaoTho11} 
that this upper bound is in fact an equality when $n \geq m \geq 16$:
\begin{theorem} Let $ n \geq m \geq 16$.  Then the $(3,3)$ broadcast domination number of $G_{m,n}$ is the $(2,1)$ broadcast domination number, i.e.
 \[ \gamma_{3,3}(G_{m,n}) = \left \lfloor \frac{(m+2)(n+2)}{5} \right \rfloor - 4 .\]
\end{theorem}

\section{Broadcast Domination  Numbers and Open Problems } \label{section:algorithm}
 We have written a dynamic programming algorithm to find the optimal (t,r) broadcast dominating numbers of an $m \times n$ grid.
 The SAGE code for this algorithm and each calculation in this paper is available at the following website: \\
{ \small \tt https://cloud.sagemath.com/projects/26b983ae-a894-47c0-bd54-c73b071e555c/files/} \\
and the results of this algorithm for (2,2) and (3,1) broadcast domination are shown in the tables below.

\begin{figure}[H]
\centering
\begin{tabular}{|c|cccccccccc|}
	\hline
	m\textbackslash n&1&2&3&4&5&6&7&8&9&10\\  
	\hline
	1&1&&&&&&&&&\\
	2&2&2&&&&&&&&\\
	3&2&3&4&&&&&&&\\
	4&3&4&6&8&&&&&&\\
	5&3&5&7&10&11&&&&&\\
	6&4&6&8&12&14&16&&&&\\
	7&4&7&10&13&16&19&21&&&\\
	8&5&8&11&15&18&22&25&28&&\\
	9&5&9&12&17&20&24&28&32&35&\\
	10&6&10&14&19&22&27&30&35&39&42\\
	\hline
\end{tabular}\caption{(2,2) broadcast domination numbers for $m,n \leq 10$}
\end{figure}

\begin{figure}[H]
\centering
\begin{tabular}{|c|cccccccccc|}
	\hline
	m/n&1&2&3&4&5&6&7&8&9&10\\ 
	\hline
	1&1&&&&&&&&&\\
	2&1&1&&&&&&&&\\
	3&1&1&1&&&&&&&\\
	4&1&2&2&3&&&&&&\\
	5&1&2&2&3&4&&&&&\\
	6&2&2&2&4&4&4&&&&\\
	7&2&2&3&4&4&6&6&&&\\
	8&2&2&3&4&5&6&7&8&&\\
	9&2&3&3&5&6&6&7&8&9&\\
	10&2&3&4&5&6&7&8&9&10 &10 \\
	\hline
\end{tabular}\caption{(3,1) broadcast domination numbers for $m,n \leq 10$}
\end{figure}

\subsection{Conjectures and Open Problems}
In this subsection we state several open questions concerning the $(\ww,\tw)$  broadcast domination numbers of $m \times n$ grids.
\begin{enumerate}
 \item When are $m$ and $n$ sufficiently large the bounds given in this paper to be tight?  We conjecture that the bounds are tight when $n \geq m > 3 p$ where the optimal dominating set of 
$\ZZ \times \ZZ$ is the preimage of a homomorphism $\phi: \ZZ \times \ZZ \rightarrow \ZZ_p$.
 \item Can the techniques from Goncalves, et.al. \cite{GonPinRaoTho11}, be adapted to find the exact values of the $(\ww, \tw)$ broadcast domination numbers?
 \item For smaller grids, the $(\ww,\tw)$ broadcast domination is not yet well-understood and is much more unpredictable than it is for large grids.   
It would be interesting if one could develop a web app game to crowd-source the problem of finding optimal broadcast dominating sets for 
$m \times n$ grids when $m$ and $n$ are relatively small.  
 \item We have seen that $(3,3)$ and $(2,1)$ broadcast dominating sets are equal for large grids. We conjecture that the optimal $(\ww, \tw)$ and $(\ww-1, \tw-2)$ broadcast dominating sets of $\ZZ \times \ZZ$ are equal.
 If this is true, then all $(\ww, \tw)$ broadcast domination problems for large grids can be reduced to $(\ww,2)$ and $(\ww,1)$ broadcast domination problems.

\end{enumerate}


\subsection*{Acknowledgements}
The authors thank Shannon Talbott and Brian Johnson for helpful comments and discussions on this paper.  
All computations in this paper were conducted on the SageMath Cloud.  We thank the SAGE development team for their hard work in developing this excellent open-source mathematics software.

\end{document}